\documentclass[reqno,11pt]{amsart}

\usepackage[a4paper,left=23mm,right=23mm,top=30mm,bottom=30mm,marginpar=25mm]{geometry}
\usepackage{amsmath}
\usepackage{amssymb}
\usepackage{amsthm}
\usepackage{amscd}
\usepackage{mathtools}
\usepackage{url}
\usepackage[ansinew]{inputenc}
\usepackage{xcolor}
\usepackage[final]{graphicx}
\usepackage{esint} 
\usepackage{bbm}
\usepackage{subfigure}
\usepackage{tikz}
\usepackage{cite}
\usepackage{enumitem}
\usepackage[hidelinks]{hyperref}

\definecolor{colora}{RGB}{128,66,21}
\definecolor{colorb}{RGB}{21,66,128}

\renewcommand{\epsilon}{\varepsilon}

\numberwithin{equation}{section}

\newtheoremstyle{thmlemcorr}{10pt}{10pt}{\itshape}{}{\bfseries}{.}{10pt}{{\thmname{#1}\thmnumber{ #2}\thmnote{ (#3)}}}
\newtheoremstyle{thmlemcorr*}{10pt}{10pt}{\itshape}{}{\bfseries}{.}\newline{{\thmname{#1}\thmnumber{ #2}\thmnote{ (#3)}}}
\newtheoremstyle{defi}{10pt}{10pt}{\itshape}{}{\bfseries}{.}{10pt}{{\thmname{#1}\thmnumber{ #2}\thmnote{ (#3)}}}
\newtheoremstyle{remexample}{10pt}{10pt}{}{}{\bfseries}{.}{10pt}{{\thmname{#1}\thmnumber{ #2}\thmnote{ (#3)}}}
\newtheoremstyle{ass}{10pt}{10pt}{}{}{\bfseries}{.}{10pt}{{\thmname{#1}\thmnumber{ A#2}\thmnote{ (#3)}}}

\theoremstyle{thmlemcorr}
\newtheorem{theorem}{Theorem}
\numberwithin{theorem}{section}
\newtheorem{lemma}[theorem]{Lemma}
\newtheorem{corollary}[theorem]{Corollary}
\newtheorem{proposition}[theorem]{Proposition}

\theoremstyle{thmlemcorr*}
\newtheorem{theorem*}{Theorem}
\newtheorem{lemma*}[theorem]{Lemma}
\newtheorem{corollary*}[theorem]{Corollary}
\newtheorem{proposition*}[theorem]{Proposition}
\newtheorem{problem*}[theorem]{Problem}
\newtheorem{conjecture*}[theorem]{Conjecture}

\theoremstyle{defi}
\newtheorem{definition}[theorem]{Definition}

\theoremstyle{remexample}
\newtheorem{remarkx}[theorem]{Remark}
\newenvironment{remark}
  {\pushQED{\qed}\remarkx}
  {\popQED\endremarkx}

\theoremstyle{ass}

\newcommand{\Fcal}{\mathcal{F}}
\newcommand{\Gcal}{\mathcal{G}}

\newcommand{\Ical}{\mathcal{I}}
\newcommand{\Jcal}{\mathcal{J}}

\newcommand{\Pcal}{\mathcal{P}}

\newcommand{\Qcal}{\mathcal{Q}}

\newcommand{\Scal}{\mathcal{S}}

\DeclareMathOperator{\esssup}{ess\,sup}

\DeclareMathOperator{\supp}{supp}

\newcommand{\normb}[1]{\bigl\|#1\bigr\|}

\newcommand{\abslr}[1]{\left|#1\right|}

\newcommand{\absb}[1]{\bigl|#1\bigr|}

\newcommand{\N}{\mathbb{N}}
\newcommand{\R}{\mathbb{R}}

\newcommand{\C}{\mathbb{C}}

\newcommand{\weakly}{\rightharpoonup}

\newcommand{\eps}{\epsilon}
\newcommand{\phirho}{\varphi^{\rho}}

\DeclareMathOperator{\Div}{div}

\def\XXint#1#2#3{{\setbox0=\hbox{$#1{#2#3}{\int}$}
\vcenter{\hbox{$#2#3$}}\kern-.5\wd0}}

\DeclarePairedDelimiter\abs{\lvert}{\rvert}
\DeclarePairedDelimiter{\norm}{\lVert}{\rVert}
\DeclarePairedDelimiter{\inner}{\langle}{\rangle}

\newcommand{\Rn}{\R^{n}}

\renewcommand{\phi}{\varphi}

\newcommand{\dx}{\, dx}

\newcommand{\F}{\mathcal{F}}

\newcommand{\weakto}{\rightharpoonup}

\def\XXint#1#2#3{{\setbox0=\hbox{$#1{#2#3}{\int}$}
     \vcenter{\hbox{$#2#3$}}\kern-.5\wd0}}

\newcommand{\Rmn}{\mathbb{R}^{m \times n}}

\newcommand{\Lipb}{\mathrm{Lip}_{b}}
\newcommand{\Lip}{\mathrm{Lip}}
\makeatletter
\g@addto@macro\bfseries{\boldmath}
\makeatother

\newcommand{\Hspd}{H^{s,p,\delta}}

\newcommand{\Hspdgm}{H^{s,p,\delta}_g(\Omega;\R^m)}

\usepackage{esint}

\usepackage{appendix}

\DeclareMathOperator{\diver}{div}

\renewcommand{\O}{\Omega}

\renewcommand{\d}{\delta}

\newcommand{\f}{\varphi}

\usepackage{amsmath}
\usepackage{latexsym}
\usepackage{amssymb}

\def\XXint#1#2#3{{\setbox0=\hbox{$#1{#2#3}{\int}$}
		\vcenter{\hbox{$#2#3$}}\kern-.5\wd0}}

\usepackage{graphicx}
\usepackage{wrapfig}

\title[A variational theory involving finite-horizon fractional gradients]{A variational theory for integral functionals involving finite-horizon fractional gradients}

\author{Javier Cueto}
\address{Department of Mathematics, University of Nebraska-Lincoln, NE, USA }
\email{jcuetogarcia2@unl.edu}

\author{Carolin Kreisbeck}
\address{Mathematisch-Geographische Fakult\"at, Katholische Universit\"at Eichst\"att-Ingolstadt, Osten\-stra{\ss}e 28, 85072 Eichst\"att, Germany}
\email{carolin.kreisbeck@ku.de}

\author{Hidde Sch\"{o}nberger}
\address{Mathematisch-Geographische Fakult\"at, Katholische Universit\"at Eichst\"att-Ingolstadt, Osten\-stra{\ss}e 28, 85072 Eichst\"att, Germany}
\email{hidde.schoenberger@ku.de}

\begin{document}

\maketitle

\thispagestyle{empty}
\begin{abstract} 
The center of interest in this work are variational problems with integral functionals depending on nonlocal gradients with finite horizon that correspond to truncated versions of the Riesz fractional gradient.
We contribute several new aspects to both the existence theory of these problems and the study of their asymptotic behavior. Our overall proof strategy builds on finding suitable translation operators that allow to switch between the three types of gradients: classical, fractional, and nonlocal. These provide useful technical tools for transferring results from one setting to the other. 
Based on this approach, we show that quasiconvexity, which is the natural convexity notion in the classical calculus of variations, gives a necessary and sufficient condition for the weak lower semicontinuity of the nonlocal functionals as well. As a consequence of a general $\Gamma$-convergence statement, we obtain relaxation and homogenization results. The analysis of the limiting behavior for varying fractional parameters yields, in particular, a rigorous localization with a classical local limit model.

\vspace{8pt}

\vspace{8pt}

\vspace{8pt}

\noindent\textsc{Date:} \today.
\end{abstract}
	
	\section{Introduction}\label{sec:intro} 
Nonlocality has long been a recurring theme in the calculus of variations, appearing in various facets and applications. When modeling phenomena in nature and technology, nonlocal operators, whose values result from integrating over a neighborhood, have become a popular alternative to differential operators.  A main advantage of this derivative-free approach is that it allows functions to be less regular and, therefore, makes it possible to capture discontinuity effects, and also long-range interactions are naturally included. In the context of mechanics, this is exploited~in peridynamic modeling \cite{Sil00, MeD15} or to cover fracture and cavitation of deformed elastic materials \cite{BeCuMC, BeCuMC22b}.
From the analytical viewpoint, dealing with nonlocality brings along new mathematical challenges, since it is intrinsically opposed to the standard techniques for classical variational problems. And yet, local and nonlocal problems can be closely intertwined: while localization causes nonlocal features to vanish \cite{MeD15, MeS, BeMCPe},  they can, on the other hand, arise from local ones  e.g.,~through  limit processes such as homogenization and disrete-to-continuum passages~\cite{Bra00, BeB98}.

In a recent series of works, different authors have studied problems involving integral functionals that depend instead of usual gradients on fractional-order ones through the Riesz fractional gradients~\cite{Shieh1, Shieh2, BeCuMC, KrS22}. Even though the latter had appeared in the literature before~\cite{Hor59}, Shieh \& Spector brought it back into the spotlight in~\cite{Shieh1, Shieh2} and discussed properties of the associated fractional Sobolev spaces, which are equivalent to the Bessel potential spaces, see also~\cite{Comi1,Comi3,KrS22,BeCuMC}.  In contrast to the standard fractional Sobolev spaces defined via Gagliardo semi-norms, these spaces have a distributional character, and are, therefore, particularly well-suited for variational problems. Another asset is that the Riesz fractional gradient enjoys a unique combination of desirable homogeneity and invariance properties as shown by \v{S}ilhav\'y in~\cite{Silhavy2019}, which makes it the natural choice of a fractional derivative among operators with infinite interaction range. Motivated by mechanical models of hyperelastic materials, which call for operators on bounded domains with finite interaction, Bellido, Cueto~\& Mora-Corral~\cite{BeCuMC22} recently proposed to consider nonlocal operators that result from the Riesz fractional gradient by truncation with a suitable cut-off function. This is the same setting we are adopting in the following. \medskip

Overall, this paper deals with variational integrals in the truncated framework of~\cite{BeCuMC22}, for which we contribute new insights into the existence theory of minimizers as well as their asymptotic analysis.
More precisely, the set-up is as follows: Let $\Omega\subset \R^n$ be a bounded open set, $s\in (0,1)$ the fractional-order parameter, $\delta>0$ the horizon, which stipulates the maximal length scale of the interaction distance between points, and $\Omega_{\delta}=\Omega+B(0, \delta)$ the nonlocal closure of $\Omega$.

We consider functionals of the form
\begin{align}\label{nonlocalfunctional_intro}
		\mathcal{F}(u)=\int_{\O} f\bigl(x,u(x), D^s_\delta u(x)\bigr) \, dx,
	\end{align}
where the integrand function $f:\Omega\times \R^m\times \R^{m\times n}\to \R$ is Carath\'{e}odory with standard $p$-growth and $p$-coercivity for some $1<p<\infty$ and $D^s_\delta u$ is the truncated Riesz fractional gradient (see~\eqref{def_Ddeltas} below)
for functions $u$ in a suitable linear subspace of $L^p(\Omega_\delta;\R^m)$. This function space, which is called $H^{s, p,\delta}(\Omega;\R^m)$ and introduced in~Definition~\ref{de:Hspd}, is defined in analogy to the classical Sobolev spaces 
by requiring that the nonlocal gradient is $p$-integrable. In addition, we assume volumetric-type boundary conditions by prescribing complementary values in a tubular neighborhood or collar of radius $2\delta$ around $\O$; in the basic case of zero complementary values, we write $H_0^{s, p, \delta}(\Omega;\R^m)$ for the set of functions admissible for~\eqref{nonlocalfunctional_intro}. 

It remains to specify the nonlocal gradient $D^s_\delta u$. With $\Gcal_\rho$ a general nonlocal gradient with kernel $\rho$, that is, 
\begin{equation} \label{eq: general nonlocal gradient}
\Gcal_\rho u(x)= \int_{\R^n}  \frac{u(x)-u(y)}{|x-y|} \frac{x-y}{|x-y|}\rho(x-y)\, dy,
\end{equation}
whenever the integral exists for a function $u:\R^n\to \R$, we first recall that the Riesz fractional gradient is defined as the nonlocal gradient with the Riesz potential kernel $I_{1-s}$, i.e.,
\begin{align*}
D^s u \,\propto\, \Gcal_{I_{1-s}} u \quad \text{with\quad $I_{1-s}\propto \frac{1}{|\cdot|^{n+s-1}}$}. 
\end{align*}
To introduce the truncated version, let us consider a certain smooth, radial cut-off function $w_\delta:\R^n\to [0, \infty)$ supported in a ball of radius $\delta$ around the origin. Then,
\begin{align}\label{def_Ddeltas}
D_\delta^s u = \Gcal_{\rho_\delta^s} u \quad \text{with\quad $\rho_\delta^s\; \propto\; w_\delta I_{1-s}$.}
\end{align}
Throughout the paper, we refer to $D^s_\d$ simply as nonlocal gradient to keep the terminology short. For more details on these definitions of nonlocal and fractional gradients, we refer the reader to Section~\ref{subsec:nonlocalcalculus}.  Alternative choices for the kernel function in \eqref{eq: general nonlocal gradient} can be found in the literature, for example, kernels defined on half-balls \cite{LeeDu,HaT23}, and variable horizon kernels \cite{TaoTianDu,DuMenTian,SiLiSel}. \medskip

Our methodology for proving the results about the functionals~\eqref{nonlocalfunctional_intro} builds substantially on their relation with classical functionals with a dependence on the usual gradient, namely
\begin{align}\label{Fgradient}
	v\mapsto \int_{\O} f(x,v(x), \nabla v(x)) \, dx,
	\end{align} 
	and also the relation with the fractional variational integrals 
	\begin{align}\label{Ffractional} 
	u\mapsto  \int_{\R^n} f(x,u(x), D^su(x)) \, dx
	\end{align}
	 provides useful insights. To set a foundation for a comparison of $\Fcal$ with~\eqref{Fgradient} and~\eqref{Ffractional}, we discuss the connection between the three differential operators
\begin{center} 
classical gradient $\nabla$, fractional gradient $D^s$, nonlocal gradient $D_\delta^s$, 
\end{center}
 and the associated Sobolev-type function spaces 
 \begin{center}
$W^{1,p}(\R^n)$, \ $H^{s, p}(\R^n)$, \ $H^{s, p, \delta}(\R^n)$,  
\end{center}
respectively; for an illustrative overview, see Figure~\ref{fig:connections}.\smallskip 
\begin{figure}\label{fig:connections}
\centering
\tikzset{every picture/.style={line width=0.75pt}} 
\resizebox{11cm}{!}{
\begin{tikzpicture}[x=0.75pt,y=0.75pt,yscale=-1,xscale=1]
\draw  [color={rgb, 255:red, 0; green, 0; blue, 0 }  ,draw opacity=1 ] (240,30) -- (400,30) -- (400,70) -- (240,70) -- cycle ;
\draw  [color={rgb, 255:red, 0; green, 0; blue, 0 }  ,draw opacity=1 ] (70,180) -- (230,180) -- (230,220) -- (70,220) -- cycle ;
\draw  [color={rgb, 255:red, 0; green, 0; blue, 0 }  ,draw opacity=1 ] (410,180) -- (570,180) -- (570,220) -- (410,220) -- cycle ;
\draw    (400,70) .. controls (470,77) and (510,140) .. (510,180) ;
\draw [shift={(510,180)}, rotate = 267] [fill={rgb, 255:red, 0; green, 0; blue, 0 }  ][line width=0.08]  [draw opacity=0] (8.93,-4.29) -- (0,0) -- (8.93,4.29) -- cycle    ;
\draw    (480,180) .. controls (420,165) and (384,119) .. (390,70) ;
\draw [shift={(390,70)}, rotate = 93] [fill={rgb, 255:red, 0; green, 0; blue, 0 }  ][line width=0.08]  [draw opacity=0] (8.93,-4.29) -- (0,0) -- (8.93,4.29) -- cycle    ;
\draw    (240,70) .. controls (170,77) and (130,140) .. (130,180) ;
\draw [shift={(130,180)}, rotate = 273] [fill={rgb, 255:red, 0; green, 0; blue, 0 }  ][line width=0.08]  [draw opacity=0] (8.93,-4.29) -- (0,0) -- (8.93,4.29) -- cycle    ;
\draw    (160,180) .. controls (220,165) and (256,119) .. (250,70) ;
\draw [shift={(250,70)}, rotate = 87] [fill={rgb, 255:red, 0; green, 0; blue, 0 }  ][line width=0.08]  [draw opacity=0] (8.93,-4.29) -- (0,0) -- (8.93,4.29) -- cycle    ;
\draw    (230,180) .. controls (260,150) and (380,150) .. (410,180) ;
\draw [shift={(410,180)}, rotate = 215.57] [fill={rgb, 255:red, 0; green, 0; blue, 0 }  ][line width=0.08]  [draw opacity=0] (8.93,-4.29) -- (0,0) -- (8.93,4.29) -- cycle    ;
\draw    (410,220) .. controls (380,250) and (260,250) .. (230,220) ;
\draw [shift={(230,220)}, rotate = 38] [fill={rgb, 255:red, 0; green, 0; blue, 0 }  ][line width=0.08]  [draw opacity=0] (8.93,-4.29) -- (0,0) -- (8.93,4.29) -- cycle    ;

\draw (255,41) node [anchor=north west][inner sep=0.75pt]   [align=left] {$\displaystyle \nabla v \ \ \ \ v \in W^{1,p}\left(\mathbb{R}^{n}\right)$};
\draw (78,189) node [anchor=north west][inner sep=0.75pt]   [align=left] {$\displaystyle D_{\delta }^{s} u\ \ \ \ u \in H^{s,p,\delta }( \R^n )$};
\draw (422,192) node [anchor=north west][inner sep=0.75pt]   [align=left] {$\displaystyle D^{s} u\ \ \ \ u \in  H^{s,p}\left(\mathbb{R}^{n}\right)$};
\draw (410,94) node [anchor=north west][inner sep=0.75pt]  [rotate=-43,xslant=0] [align=left] {$\displaystyle v=^\dag I_{1-s} *u$};
\draw (465,56) node [anchor=north west][inner sep=0.75pt]  [rotate=-43.8] [align=left] {$\displaystyle u=( -\Delta )^{\frac{1-s}{2}} v$};
\draw (176,150) node [anchor=north west][inner sep=0.75pt]  [rotate=-317] [align=left] {$\displaystyle v= Q_{\delta }^{s} *u$};
\draw (130,107) node [anchor=north west][inner sep=0.75pt]  [rotate=-316.2] [align=left] {$\displaystyle u=\Pcal^s_\d v$}; 
\draw (247,192) node [anchor=north west][inner sep=0.75pt]   [align=left] {$\displaystyle {\textstyle D_{\delta }^{s} u-D^{s} u= \nabla R_{\delta }^{s} *u}$};
\end{tikzpicture}}
\caption{Illustration of the relations between classical, fractional, and nonlocal gradients, which enable the transfer of results between the corresponding settings. $^\dag$ When $I_{1-s} * u$ is well-defined.}
\end{figure}
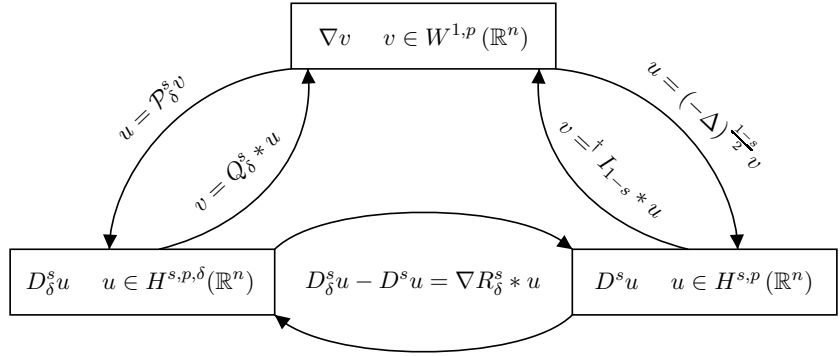

\textit{ Fractional vs.~classical:} 
For smooth compactly supported functions $\varphi\in C_c^\infty(\R^n)$, it is by now well-known that
  \begin{align} \label{eq: fractional to classical gradients}
		D^s \f= \nabla (I_{1-s} \ast \f)\quad \text{ and } \quad \nabla \f=D^s(-\Delta)^{\frac{1-s}{2}} \f,
\end{align}
where $I_{1-s}$ is the Riesz potential and $(-\Delta)^{\frac{1-s}{2}}$ is the fractional Laplacian of order $1-s$, see e.g.,~\cite{Shieh1, Silhavy2019}. In~\cite[Proposition~3.1]{KrS22}, two of the authors extended these identities to the setting of Sobolev and fractional Sobolev functions, showing that for any $u\in H^{s, p}(\R^n)$, there exists a $v\in W_{\rm loc}^{1, p}(\R^n)$ such that $\nabla v=D^su$, and for every $u\in H^{s, p}(\R^n)$, one can find a $v\in W^{1, p}(\R^n)$ with $D^su=\nabla v$.
The latter follows immediately from the observation that $(-\Delta)^{1-s/2}:W^{1, p}(\R^n)\to H^{s,p}(\R^n)$ is a bounded linear operator. This way, one can translate from the fractional gradient to the classical one and vice-versa, up to a gap related to an issue of local integrability. For a similar statement in the space of fractional $BV$-functions, we refer to~\cite[Lemma~3.28]{Comi1}.

 \smallskip
\textit{Nonlocal vs.~classical:} Providing analogous translation formulas between the nonlocal and classical setting is one of the major steps in the analysis of this paper. The fact that $D^s_\delta u$ is defined over a bounded domain brings about some technical complications compared with $D^su$; for instance, as opposed to $D^s u$, the operator $D^s_\d u$ is no longer homogeneous and it does not enjoy a semigroup property, which the fractional one inherits from its relationship with the Riesz potential. The foundations for finding a suitable replacement for the generalization of~\eqref{eq: fractional to classical gradients}, were laid by Bellido, Cueto \& Mora-Corral~\cite{BeCuMC22} (see also~\cite{BeCuMC22b}). They identified an integrable finite-horizon counterpart of the Riesz potential kernel, called $Q_\delta^s$, which provides one of the directions of the translation mechanism for smooth functions. For the other direction, we heuristically invert the convolution with $Q^s_\d$ in Fourier space, i.e., we consider the operator
\[
\mathcal{P}^s_\d\phi = \left(\frac{\widehat{\phi}}{\widehat{Q}^s_\d}\right)^{\vee}
\]
for any Schwartz function $\phi$. This operator can be considered as an analogue of the fractional Laplacian of order $(1-s)/2$ in the nonlocal framework. Another way of interpreting $\Pcal^s_\d\phi$ is as the convolution of the gradient of $\phi$ with the kernel from the nonlocal fundamental theorem of calculus in \cite[Theorem 4.5]{BeCuMC22}, see ~Remark~\ref{rem:connection}\,d)

Here, we prove that the convolution with $Q^s_\d$ and $\Pcal^s_\d$ can both be extended to the Sobolev spaces in such a way that they are each other's inverses. This gives a perfect isomorphism between $H^{s,p,\d}(\R^n)$ and $W^{1,p}(\R^n)$ with the property that
for any $u\in H^{s, p, \delta}(\R^n)$ and $v\in W^{1, p}(\R^n)$
\begin{align}\label{translation_intro} 
D_\delta^s u=
\nabla (Q_\delta^s\ast u)  \qquad \text{and} \qquad \nabla v= D^s_\delta\Pcal^s_\d v,
\end{align}
see Theorem~\ref{prop:connection} and the discussion thereafter. It is noteworthy that in the fractional case there is no such isomorphism, since the Riesz potential is only locally integrable as opposed to $Q^s_\d$.
\smallskip

\textit{Fractional vs.~nonlocal:} A comparison between the kernels $Q_\delta^s$ and $I_{1-s}$, where $R_\delta^s$ denotes their difference, gives us a basic and direct way for switching between the fractional and nonlocal setting. Indeed, we show in Section~\ref{subsec:connection_localfrac}, that 
\begin{align}\label{comparison_nonlocalfrac_intro}
D_\delta^s u = D^su + \nabla R_\delta^s\ast u,
\end{align} for all $u\in H^{s, p, \delta}(\R^n)=H^{s, p}(\R^n)$, 
where $\nabla R_\delta^s \in L^1(\R^n)\cap C^\infty(\R^n)$.
\medskip

Having the translation mechanism of~\eqref{translation_intro} and \eqref{comparison_nonlocalfrac_intro} at hand paves the way for shifting results between the three variational settings. Note, however, that not all results can be directly carried over, since boundary conditions are not preserved in the translation procedure and problems involving both the function and its nonlocal gradient require additional techniques.
Here, we list and discuss the main contributions of this paper to the existence and asymptotic analysis of the functionals $\Fcal$ in~\eqref{nonlocalfunctional_intro}:\smallskip

\textit{(1) Characterization of weak lower semicontinuity of $\Fcal$.}  One of the crucial steps to conclude the existence of minimizers of integral functionals, like $\Fcal$ or those in~\eqref{Fgradient} or~\eqref{Ffractional}, via the direct method, is to establish weak lower semicontinuity. 
A well-known fundamental result from the vectorial calculus of variations with roots in the 1950s states that, for the functionals~\eqref{Fgradient}, quasiconvexity (in the sense of Morrey) regarding the third variable of $f$ is necessary and sufficient for weak lower semicontinuity in $W^{1, p}(\Omega;\R^m)$, see~\cite{AcF84, Mey65, Morrey, Mar85}.  In the fractional setting~\eqref{Ffractional}, the efforts are more recent.  After convexity~\cite{Shieh1} and polyconvexity~\cite{BeCuMC} had been identified as sufficient conditions for weak lower semicontinuity in $H_0^{s, p}(\Omega;\R^m)$, the problem of characterization was solved in~\cite[Theorem~1.1]{KrS22}. Interestingly, the correct condition on $f$ is the same as in the local case, namely quasiconvexity.

We complement the picture in Theorem~\ref{theo:characterization}, by proving that, altogether, quasiconvexity is the intrinsic convexity notion in all three situations. In fact,
\begin{align}\label{Fwlsc_intro}
\begin{split}
&\text{$\Fcal$ is weakly lower semicontinuous in $H_0^{s, p, \delta}(\Omega;\R^m)$}\text{ if and only if }\\ &\qquad \qquad \qquad \text{$f(x, z, \cdot)$ is quasiconvex for a.e.~$x\in \Omega_{-\delta}$ and all $z\in \R^m$; }
\end{split}
\end{align}
 note that, due to a boundary layer effect, which yields even strong $L^p$-convergence of weakly convergent sequences in $H_0^{s,p,\delta}(\Omega;\R^m)$, quasiconvexity is not required in the collar. Moreover, we introduce a nonlocal notion of quasiconvexity defined through testing with nonlocal gradients that turns out to be equivalent with quasiconvexity, cf.~Remark~\ref{rem:dsdquasiconvexity}.  

The proof of~\eqref{Fwlsc_intro} exploits the parallels between the nonlocal and fractional gradient in their relation to the standard one (cf.~\eqref{translation_intro} and~\eqref{eq: fractional to classical gradients}) by using similar arguments and techniques as in~\cite{KrS22}. An alternative proof strategy that reduces~\eqref{Fwlsc_intro} directly to the statement of~\cite[Theorem~1.1]{KrS22} via \eqref{comparison_nonlocalfrac_intro} is also possible, as we demonstrate under simplified assumptions. \medskip

\textit{(2) Variational convergence, homogenization and relaxation.} Considering sequences of nonlocal functionals $\{\Fcal_{f_j}\}_{j\in \N}$ as in~\eqref{nonlocalfunctional_intro} with specific integrand functions $f_j$, we study their asymptotic behavior as $j\to \infty$. The intention of finding a versatile method that makes $\Gamma$-convergence (see~\cite{Bra02, Dal93}) accessible to a number of cases and applications motivates the statement of Theorem~\ref{th:gamgeneral}. If we denote the counterparts of $\Fcal_{f_j}$ with dependence on classical gradients defined on $W^{1, p}(\Omega_{-\delta};\R^m)$ by $\Ical_{f_j}$, it says that the convergence of 
 $\{\Ical_{f_j}\}_{j\in \N}$ to a $\Gamma$-limit $\Ical_{f_\infty}$ as $j\to \infty$ along with the pointwise convergence of the integrals over the collar, $L^p(\Omega_\delta;\R^{m\times n})\ni V\mapsto \int_{\Omega\setminus \Omega_{-\delta}}f_j(x,V)\, dx$ yields 
 \begin{align*}
 \Gamma\text{-}\lim_{j\to \infty}\Fcal_{f_j} =\Fcal_{f_\infty};
 \end{align*}
 note that all $\Gamma$-limits are taken with respect to the strong $L^p$-topology. 

To demonstrate how this observation can help to carry various $\Gamma$-convergence results in the literature from the local to the nonlocal setting, we choose homogenization theory as a specific case. Indeed, Corollary~\ref{cor:homogenization} shows that the fundamental $\Gamma$-limit of~\cite{Mul87, Bra85}, where the homogenized functional is again of integral form with integrand determined by a multi-cell formula, gives rise to a new homogenization limit for problems involving nonlocal gradients.
As an immediate consequence of this homogenization, one can obtain relaxation of nonlocal functions $\Fcal$, that is, a representation for their lower semicontinuous envelopes. In the case of a homogeneous integrand $f$, the latter arises from the quasiconvexification of $f$ on $\Omega_{-\delta}$, while $f$ remains unchanged in $\Omega\setminus \Omega_{-\delta}$, see Corollary~\ref{cor:relaxation}.

 \medskip
\textit{(3) Asymptotics for varying fractional order.}  
It is a natural question to investigate the dependence of our nonlocal variational problems, in particular, their minimizers and minima, on the fractional order $s\in (0,1)$; for an analogous study for functionals of the type~\eqref{Ffractional}, see~\cite{BeCuMC21}. 
To this end, we take functionals as in~\eqref{nonlocalfunctional_intro}, with $f$ independent of the second variable and quasiconvex in the third one, and highlight the dependence of $s$ with a subscript index $\Fcal_s$. The functional $\Fcal_1$ can be defined in the same way with $D^1_\delta u:=\nabla u$ the classical gradient and $\Fcal_0$, after extension of the definition in~\eqref{def_Ddeltas} to $s=0$, lives on $L^p(\Omega_\delta;\R^m)$.

The main result in this context is Theorem~\ref{th:gamgeneral}, which says the following: 
\begin{center}
The sequence $\{\Fcal_s\}_{s}$ $\Gamma$-converges to $\Fcal_{s'}$ as $s\to s'\in [0,1]$;
\end{center} 
since sequential compactness of bounded-energy sequences holds strongly in $L^p(\Omega;\R^m)$ when $s'\in (0,1]$ and weakly in $L^p(\Omega;\R^m)$ if $s'=0$, it is natural to state the $\Gamma$-convergence results regarding the strong and weak topology, respectively, see Theorem~\ref{le:ordercompactness}. 
We point out that the limit $s\to 1$ provides a localization statement, and as such, establishes another interesting connection between classical local and nonlocal theories.

The proof of the above-mentioned compactness for bounded-energy sequences in nonlocal spaces of different order involves, besides the continuous dependence of the nonlocal gradient $D^s_\delta u$ on $s$ (see Lemma~\ref{le:orderconvergencesmooth}), also a new technical tool that is worth mentioning in its own right.  
This is the nonlocal Poincar\'e inequality with a constant independent of the fractional order presented in~Theorem~\ref{th:poincareindep}; we refer to recent progress on nonlocal Poincar\'{e}-type inequalities, for example, in problems involving radial kernels \cite{BeCuMC22, DuTi18} or asymmetric and inhomogeneous kernels \cite{HaT23,FossPoincare}.   
The difficulty in establishing a parameter-independent bound is the fact that the kernel in the nonlocal fundamental theorem of calculus from \cite[Theorem 4.5]{BeCuMC22} is implicitly defined via a Fourier transform, which makes it hard to isolate the dependence on $s$  in the proof of the Poincar\'{e} inequality from \cite[Theorem~6.2]{BeCuMC22}. Instead, we utilize a fine analysis of the decay of the Fourier transform of $Q^s_\d$, an application of the Mihlin-H\"{o}rmander multiplier theorem and an extension of the nonlocal fundamental theorem to the case $s=0$ (see~Proposition~\ref{prop:ftoczero}) to prove the Poincar\'{e} inequality with an $s$-independent constant. \smallskip

This manuscript is organized as follows. We begin in Section~\ref{sec:prelim} with notations and a detailed introduction to our set-up and nonlocal calculus. Moreover, we collect and establish the relevant technical tools, especially, the connections between classical, nonlocal and fractional gradients along with the corresponding translation keys. 
Section~\ref{sec:asymptotics} deals then with the asymptotics of the nonlocal gradient, and we derive as a main application a Poincar{\'e} inequality with a constant uniform in $s$, which opens the way for compactness results for sequences in nonlocal spaces of different order.
The variational results for the nonlocal integral functionals are proven from Section~\ref{section:wls} onwards, based on the comparison with the classical and fractional setting. First, we prove the characterization of weak lower semicontinuity in terms of quasiconvexity of the integrand and state an existence statement for minimizers of $\Fcal$ (see Corollary~\ref{theo:existenceminimizers}) based on it. In Section~\ref{section:homrelax}, we then provide a general $\Gamma$-convergence result, from which homogenization and relaxation can be deduced as corollaries. Finally, we prove the convergence of minimizers of the functionals $\{\Fcal_s\}_s$ for the limit $s\to s'\in [0,1]$ in Section~\ref{sec:Gamma}, showing, in particular, the localization to a classical local limit as $s\to 1$.

\section[Preliminaries]{Preliminaries and technical tools} \label{sec:prelim}
The aim of this section is to introduce the notation and several important definitions and tools regarding the nonlocal gradient and Sobolev spaces.

\subsection[Notation]{Notation} 

\subsubsection{General notation} 
Unless mentioned otherwise, $s\in (0,1)$ and $Y=(0,1)^n\subset \R^n$. We use $\R_{\infty}$ to denote $\R \cup \{\infty\}$. We write $\abs{x}=\left(\sum_{i=1}^n x_i^2\right)^{1/2}$ for the Euclidean norm of a vector $x=(x_1,\cdots,x_n)\in \R^n$ and similarly, $\abs{A}$ for the Frobenius norm of a matrix $A \in \Rmn$. The ball centered at $x \in \R^n$ and with radius $\rho>0$ is denoted by $B(x,\rho)=\{ y \in \R^n : \abs{x-y}<\rho\}$ and the distance between $x \in \R^n$ and a set $E\subset \R^n$ is written as $d(x,E)$. For an open set $\Omega \subset \R^n$ and $\d>0$, we write $\Omega_\delta$ for its nonlocal closure, that is,
\begin{align*}\label{Omegadelta}
\Omega_\delta= \Omega+B(0, \delta)=\{x \in \R^n\,:\, d(x,\Omega) <\delta\}.
\end{align*}
The complement of a set $E\subset \R^n$ is indicated by $E^c:=\R^n \setminus E$ and its closure by $\overline{E}$. The notation $E \Subset F$ for sets $E, F\subset \R^n$ means that $E$ is compactly contained in $F$, i.e., $\overline{E} \subset F$ and $\overline{E}$ is compact. Let 
\[
\mathbbm{1}_E(x) = \begin{cases} 1 &\text{for} \ x \in E,\\
0 &\text{otherwise},
\end{cases}\qquad x\in \R^n,
\]
be the indicator function of a set $E \subset \R^n$.

Let $U\subset \R^n$ be an open set. 
The notation $C_{c}^{\infty}(U)$ symbolizes the smooth functions $\varphi:U\to \R$ with compact support in $U \subset \R^n$.  Our convention is that functions in $C_c^\infty(U)$ are identified with their trivial extension to $\R^n$ by zero. Further, by $C^\infty(\R^n)$, $C_0(\R^n)$ and $\Scal(\R^n)$ we denote the space of smooth functions, continuous functions vanishing at infinity and Schwartz functions on $\R^n$, respectively. We utilize multi-index notation, in particular, we write $\partial^{\alpha}$ for the partial derivative with respect to a multi-index $\alpha \in \N^n_0$.

By $\Lipb(\R^n)$, we refer to all the functions $\psi:\R^n \to \R$ that are Lipschitz continuous and bounded on $\R^n$ and we write $\Lip(\psi)$ for the Lipschitz constant of $\psi$.

The Lebesgue measure of $U \subset \R^n$ is written $\abs{U}$ and the convolution of two functions $u,v:\R^n \to \R$ is denoted by $u*v$. If one of the functions is vector-valued, the convolution should be understood componentwise. We use the common notation for Lebesgue- and Sobolev-spaces, that is, $L^p(U)$ for $p\in [1, \infty]$ is the space of $p$-real-valued integrable functions on $U$ with the norm 
\[
\norm{u}_{L^p(U)}=\begin{cases}\displaystyle\left( \int_{U} \abs{u(x)}\dx\right)^{1/p} &\text{if} \ p \in [1,\infty),\\
\esssup_{x \in U} \abs{u(x)} &\text{if} \ p=\infty,
\end{cases}\qquad u\in L^p(U).
\]  
Moreover, $W^{1,p}(U)$ for $p\in [1, \infty]$ consists of all $L^p$-functions on $U$ with $p$-integrable weak derivatives, endowed with the norm
\[
\norm{u}_{W^{1,p}(U)}=\norm{u}_{L^p(U)}+\norm{\nabla u}_{L^p(U;\R^n)};
\]
here $\nabla u$ stands for the weak gradient of $u$. 

The functions that lie locally in $L^p$ and $W^{1, p}$ are denoted by $L^p_{\rm loc}(\R^n)$ and $W^{1,p}_{\rm loc}(\R^n)$. Besides, $W^{1,p}_{0}(U)$ stands for those functions in $W^{1,p}(U)$ with zero boundary value in the sense of the trace and $W^{1,\infty}_{\#}(Y)$ indicates the $Y$-periodic functions in $W^{1,\infty}(\R^n)$. 

In general, the spaces defined above can be extended componentwise to vector-valued functions. The target space is explicitly mentioned in the notation, like, for example, $L^p(U;\R^m)$. Whenever convenient, we identify a function on a subset of $\R^n$ with its trivial extension by zero. Finally, we use $C$ to denote a generic constant, which may change from one estimate to the next without further mention. If we wish to indicate the dependence of $C$ on certain quantities, we add them in brackets.
	
\subsubsection{Riesz potential and Fourier transform}
We recall the definition of Riesz potential. 
Given $0<s<n$, the Riesz potential kernel $I_s : \R^n \setminus \{0\} \to \R$ is
\begin{equation}\label{Rieszkernel}
I_s (x)= \gamma_{n,s}^{-1}\frac{1}{|x|^{n-s}},
\end{equation} 
where 
\[
\gamma_{n,s}=\frac{\pi^{\frac{n}{2}} \, 2^s \, \Gamma(\frac{s}{2})}{\Gamma(\frac{n-s}{2})}
\]
with $\Gamma$ denoting the Gamma function. For notational convenience, we also define $I_1:\R \setminus \{0\} \to \R$ as
\begin{equation}\label{Rieszkernel2}
I_1(x) = -\frac{1}{\pi}\log(\abs{x})
\end{equation}
when $n=1$. The Riesz potential of a locally integrable function $f$ is given via convolution as 
\begin{equation*}
I_s * f(x)=\frac{1}{\gamma_{n, s}} \int_{\R^n} \frac{f(y)}{|x-y|^{n-s}}dy,
\end{equation*}
whenever the integral exists for a.e.~$x \in \R^n$.

Since we will also deal with the use of the Fourier transform, we clarify here the notation we are going to use. For $f \in L^1(\R^n)$, we define the Fourier transform of $f$ as
\begin{equation*}
	\widehat{f}(\xi)=\int_{\Rn} f(x) \, e^{-2\pi i x \cdot \xi} \, dx \quad \xi \in \R^n.
\end{equation*}
Notice that this definition can also be used in the Schwartz space $\mathcal{S}(\R^n;\C)$, where it defines an isomorphism. By continuity and duality extensions, it also defines isomorphism on the spaces $L^2(\R^n;\C)$ and in the space of tempered distributions $\mathcal{S}'(\R^n;\C)$. Moreover, the inverse Fourier transform is denoted by $f^{\vee}$ and corresponds with $x \mapsto \widehat{f}(-x)$. Notable references in Fourier analysis are \cite{duon2000,Gra14a}.

\subsection[Nonlocal Calculus]{Nonlocal calculus and function spaces}\label{subsec:nonlocalcalculus}
In this section, we present the definition of the nonlocal gradient used throughout this paper, introduce the naturally associated function spaces, and collect several auxiliary results. A delicate issue is the choice of suitable boundary values, which is addressed below in Section~\ref{subsec:complementary}.

 In what follows, let $\d>0$ and $w_\d: \Rn\to [0,\infty)$ be a non-negative cut-off function satisfying these hypotheses:
\smallskip
\begin{enumerate}[label = (H\arabic*)]
	\item \label{itm:h1}$w_\d$ is radial, i.e., there is a $\overline{w}_\d:\R \to [0,\infty)$ such that $w_\d(x)=\overline{w}_\d(\abs{x})$ for $x \in \R^n$;\\[-0.3cm]
	\item \label{itm:h2}$w_\delta$ is smooth and compactly supported in $B(0, \delta)$, i.e., $w_\d \in C_c^\infty(B(0,\delta))$;\\[-0.3cm]
	\item \label{itm:h3}there is a constant $b_0\in (0,1)$ such that $w_\d =1$ on $B(0, b_0\delta)$;\\[-0.3cm]
	\item \label{itm:h4}$w_\delta$ is radially decreasing, that is, $w_\d (x) \geq w_\d (y)$ if $|x| \leq |y|$.
\end{enumerate}
\smallskip
In accordance with~\cite[Definition~3.1]{BeCuMC22}, we define the nonlocal gradient and divergence for smooth functions as follows: 
	 For $s \in [0,1)$, the nonlocal gradient of $\varphi\in C^{\infty}(\Rn)$ is given by
		\begin{equation} \label{eq: definition of nonlocal gradient}
			D_\delta^s \varphi(x)= c_{n,s} \int_{ \R^n} \frac{\varphi(x)-\varphi(y)}{|x-y|}\frac{x-y}{|x-y|}\frac{w_\d(x-y)}{|x-y|^{n+s-1}} \, dy \quad \text{for $x \in \R^n$, }
		\end{equation}
and the nonlocal divergence of $\psi \in C^{\infty} (\Rn;\Rn)$ is 
		\begin{equation}\label{eq: equiv integrals NL divergence}
			\diver_{\delta}^s \psi(x)= c_{n,s} \int_{ \R^n} \frac{\psi(x)-\psi(y)}{|x-y|}\cdot\frac{x-y}{|x-y|}\frac{w_\d(x-y)}{|x-y|^{n+s-1}} \, dy \quad \text{for $x \in \R^n$,}
		\end{equation}
	with the scaling constant 
	\begin{equation*}
		c_{n,s}:= \frac{\Gamma\left(\frac{n+s+1}{2}\right)}{\pi^{n/2}2^{-s}\Gamma\left(\frac{1-s}{2}\right)}.
	\end{equation*}
Note that the integral in \eqref{eq: definition of nonlocal gradient} is absolutely convergent given that $\varphi$ is in particular locally Lipschitz continuous and $w_{\d}(\cdot)/\abs{\cdot}^{n+s-1} \in L^1 (\Rn)$ with compact support. Moreover, the above definitions show that $\supp (D_\delta^s \varphi) \subset \supp (\varphi) +\overline{B (0, \d)}$ and Proposition~\ref{Prop: convolution with the classical gradient} below establishes $D^s_\d \phi \in C^{\infty}(\R^n;\R^n)$. Analogous observations hold for the nonlocal divergence.
	
\begin{remark}\label{rem:nonlocalgradprop}
a) Due to the radial symmetry of $w_\delta$ from \ref{itm:h1}, an equivalent way of expressing $D_\delta^s\varphi$ for $\varphi\in C^\infty(\R^n)$ is as
\begin{align}\label{nonlocalgrad_alter}
 D_\delta^s \varphi(x) =  \lim_{r \downarrow 0}\int_{B(x,r)^c} \varphi(y) d_\delta^s(x - y) \, dy \quad \text{for $x\in \R^n$,} 
\end{align}
with
\begin{align}\label{defddeltas}
 d_\delta^s(x) =-c_{n,s} \frac{xw_\delta(x)}{|x|^{n+s+1}} \quad\text{for $x\in \R^n\setminus\{0\}$.}
\end{align}
When $x \not \in \supp(\phi)$, this allows us to write $D^s_\d \phi(x) = (d^s_\d * \phi)(x)$.
\smallskip

b) It is straightforward to check for the nonlocal gradient that it is translation and rotation invariant, i.e.,
$$D_\delta^s\bigl(\varphi(\cdot +b)) = D_\delta^s \varphi(\cdot+b) \quad \text{and} \quad D_\delta^s\bigl(\varphi(R\,\cdot)) = R^{-1}D_\delta^s \varphi(R\,\cdot) $$ 
for all $\varphi\in C^\infty(\R^n)$, $b\in \R^n$ and $R\in \mathrm{O}(n)$. The rotation invariance relies on the radiality of $w_\delta$. If, in addition, $w_\delta(\cdot/\lambda)=w_{\lambda\delta}(\cdot)$ for all $\lambda>0$, then $D_\delta^s$ is also positively $s$-homogeneous in the sense that 
$$D_\delta^s\bigl(\varphi(\lambda\, \cdot)) =\lambda^s  D_{\delta/\lambda}^s \varphi(\lambda\,\cdot) $$ 
for all $\varphi\in C_c^\infty(\R^n)$ and $\lambda>0$. 

To put this observation in context, we remark that \v{S}ilhav\'{y} in~\cite{Silhavy2019} identified the Riesz fractional gradient as the unique fractional derivative operator that is suitably continuous, rotation and translation invariant and $s$-homogeneous. Hence, one can view $D_\d^s$ as a nonlocal derivative operator with finite interaction range that enjoys the same desirable properties.
\end{remark}

As recently shown in~\cite{BeCuMC22}, the nonlocal gradient can be written as the convolution of a certain integrable kernel with the classical gradient. To formulate this result, which is in analogy to the representation  of the Riesz fractional gradient as the Riesz potential of the usual gradient, we first introduce for $s\in [0,1)$ the kernel
	\begin{align}\label{Qdeltas}
Q_\d^s : \Rn \setminus \{0\} \to \R, \quad Q^s_\d(x)=c_{n,s}\int_{\abs{x}}^\d \frac{\overline{w}_{\d}(t)}{t^{n+s}}\,dt.
	\end{align}
\begin{proposition} \label{Prop: convolution with the classical gradient}
Let $s \in [0,1)$. It holds for every $\varphi \in C^{\infty}(\Rn)$ that
	\begin{align*}
		D_\delta^s \varphi 
		= Q_\delta^s \ast \nabla \varphi\ \in C^\infty(\Rn). 
	\end{align*}
	In particular, when $\phi \in \Scal(\R^n)$ then $D^s_\d \phi \in \Scal(\R^n;\R^n)$.
\end{proposition}

\begin{proof} 
The statement for $\varphi \in C_c^{\infty}(\R^n)$ and $s \in (0,1)$ is exactly \cite[Proposition~4.3]{BeCuMC22}, and the case $s=0$ is proven analogously. Since any $\varphi \in C^{\infty}(\R^n)$ locally coincides with a smooth function with compact support, the same holds for such functions. Finally, since $Q^s_\d \in L^1(\R^n)$, the statement for Schwartz functions follows.
\end{proof}

\begin{remark}[Properties of $Q_\delta^s$]\label{rem:Qdeltas} 
For easier referencing, we list here a few relevant properties of $Q_\d^s$ for $s\in [0, 1)$ that will be used later in the paper. The details for $s\in(0,1)$ can be found in~\cite[Lemma~4.2, Propositions~5.2 and~5.5]{BeCuMC22}, and the same arguments extend also to the case $s=0$.\smallskip

a) The kernel $Q_\delta^s$ lies in $L^1(\R^n)$ with $\supp (Q_\delta^s) \subset B(0, \delta)$ and is radially decreasing.\smallskip

b) Since $Q_\delta^s$ has compact support, its Fourier transform is analytic and thus smooth. Moreover, $\widehat{Q}_\delta^s$ is bounded, radial, and strictly positive. 
\end{remark}

The nonlocal gradient and divergence as defined in~\eqref{eq: definition of nonlocal gradient} and~\eqref{eq: equiv integrals NL divergence} act as dual operators in the sense of integration by parts. While several versions of nonlocal integration by parts for related fractional or nonlocal operators have been studied in the literature \cite{MeS, Comi1, Silhavy2019}, we employ here the following formula, stated for smooth functions.

\begin{lemma}[Nonlocal integration by parts formula]\label{th:Nl parts}
	Let $s \in [0,1)$ and suppose that $\varphi \in C^{\infty}_c (\Rn)$ and $\psi \in C_c^{\infty} (\Rn; \Rn)$.
Then, 
	\begin{align*}
		\int_{\R^n} D_\delta^s \varphi \cdot \psi \, dx =& - \int_{\R^n} \varphi \diver_\delta^s \psi \, dx. 
	\end{align*}
\end{lemma}
\begin{proof}
According to Proposition~\ref{Prop: convolution with the classical gradient}, it holds that $D^s_\d \varphi = Q^s_\d* \nabla \varphi = \nabla (Q^s_\d * \varphi) \in C_c^{\infty}(\R^n;\R^n)$ and similarly, $\Div^s_\d \psi= Q^s_\d * \Div \psi \in C_c^{\infty}(\Rn)$. Hence, we may calculate
\begin{align*}
\int_{\Rn} D_\delta^s \varphi \cdot \psi \, dx & = \int_{\R^n} \nabla (Q^s_\d * \varphi) \cdot \psi \,dx\\
&= -\int_{\R^n} (Q^s_\d *\varphi) \Div \psi\,dx = -\int_{\R^n} \varphi \,(Q^s_\d*\Div \psi)\,dx=-\int_{\Rn} \varphi \Div^s_\d \psi\,dx;
\end{align*}
the second identity is due to classical integration by parts, while the third one follows via Fubini's theorem.
\end{proof}

In light of this integration by parts formula, the definition~in~\eqref{eq: definition of nonlocal gradient} can be extended to a broader class of functions using a distributional approach. 
We will work with functions defined on an open set $\Omega\subset\R^n$. As nonlocal boundary of this set, we choose a volumetric type as is common in nonlocal models, considering a tubular neighborhood or collar of radius $\d>0$ around $\O$. Precisely, $\O_{\d} = \Omega + B(0,\delta)$ is the nonlocal closure of $\Omega$ and  $\O_{\d} \setminus \O$  plays the role of nonlocal boundary.

\begin{definition}[Weak nonlocal gradient]\label{def: nonlocal gradient 2}
Let $s \in [0,1)$, $\d >0$, $\O \subset \R^n$ open and $u \in L^1_{\rm loc}(\O_\d)$. We say that $v \in L^1_{\rm loc}(\O;\R^n)$ is the weak nonlocal gradient of $u$, written as $v = D^s_\d u$, if
\[
\int_{\O}v \cdot \psi \,dx = - \int_{\O_\d} u \Div^s_\d \psi\,dx \qquad\text{for all $\psi \in C_c^{\infty}(\O;\R^n)$}.
\]
\end{definition}

\begin{remark}\label{rem:singular}
 In the case $s = 0$, it holds for each $\phi \in C^{\infty}(\R^n)$ by \eqref{nonlocalgrad_alter} that
\begin{align*}
D^0_\d \phi (x) =\lim_{r \downarrow 0}\int_{B(x,r)^c} \varphi(y) d_\delta^0(x - y) \, dy \quad \text{for $x\in \R^n$,}
\end{align*}
with $d^0_\d$ as in~\eqref{defddeltas}.
The theory of singular integrals (see e.g.,~\cite[Theorem~5.4.1]{Gra14a}) implies that $D^0_\d$ can be uniquely extended to a continuous linear operator from $L^p(\R^n)$ to $L^p(\R^n;\R^n)$ when $p \in (1,\infty)$; indeed, one can easily verify that $d^0_\d$ satisfies the size and cancellation conditions~\cite[Eq.~(5.4.1) and (5.4.3)]{Gra14a}, while the H\"ormander condition~\cite[Eq.~(5.4.2)]{Gra14a} follows from the stronger property
\[
\abs{\nabla d^0_\d} \leq \frac{C}{\abs{\,\cdot\, }^{n+1}},
\]
which holds due to $\nabla w_{\delta}=0$ in $B(0,b_0\d)$. 

We therefore find for each $u \in L^p(\Omega_{\d})$ (after extension to $\R^n$ by zero) that $D^0_\d u \in L^p(\Omega;\R^n)$ and
\[
\norm{D^0_\d u}_{L^p(\Omega;\R^n)} \leq C \norm{u}_{L^p(\Omega_\d)}
\]
with $C>0$ a constant independent of $u$. Note that via a density argument, $D^0_\d u$ coincides with the weak nonlocal gradient from Definition~\ref{def: nonlocal gradient 2}.
\end{remark}

In analogy with the definition of the standard and fractional Sobolev spaces, it is now quite natural 
 to consider the space of $L^p$-functions whose weak nonlocal gradient is also an $L^p$-function.

\begin{definition}[Nonlocal Sobolev spaces]\label{de:Hspd}
Let $s\in [0,1)$, $p \in [1,\infty]$ and $\O \subset \R^n$ be open. We define the nonlocal Sobolev space $\Hspd(\O)$ as
\[
\Hspd(\O):=\{u \in L^p(\O_\d) \,:\, D^s_\d u \in L^p(\O;\R^n)\},
\]
equipped with the norm
	\begin{displaymath}\label{normHspdelta}
		\lVert u\rVert_{H^{s,p,\d}(\O)} = \Bigl( \left\| u \right\|_{L^p (\O_{\d})}^p + \left\| D_\delta^s u \right\|_{L^p (\O;\R^n)}^p \Bigr)^{\frac{1}{p}} .
	\end{displaymath}
The corresponding spaces of vector-valued functions $H^{s,p, \delta}(\Omega;\R^m)$ are defined componentwise.
\end{definition}
 In parallel with the classical Sobolev spaces, $H^{s,p,\d}(\O)$ is a Banach space and, when $p \in (1,\infty)$, also reflexive. Moreover, a sequence $\{u_j\}_{j\in \N} \subset H^{s,p,\d}(\O)$ converges weakly to $u$ in $H^{s,p,\d}(\O)$ for $p \in (1,\infty)$ if and only if $u_j \weakto u$ in $L^p(\O_\d)$ and $D^s_\d u_j \weakto D^s_\d u$ in $L^p(\O;\R^n)$  as $j\to \infty$.  In view of Remark~\ref{rem:singular}, it holds that $H^{0,p,\d}(\O)=L^p(\O_\d)$ for $p \in (1,\infty)$ with an equivalent norm. Additionally, we set 
\begin{equation}\label{eq:classicalspace1}
H^{1,p,\d}(\R^n):=W^{1,p}(\R^n) \quad \text{with} \ \  D^1_\d u:= \nabla u \ \text{for $u \in H^{1,p,\d}(\R^n)$,}
\end{equation}
which provides a consistent notation for the range of fractional orders $s \in [0,1]$. \smallskip

When we consider the whole space, i.e., $\Omega=\R^n$, and $s\in (0,1)$, then by Lemma~\ref{le:nonlocalfrac} the nonlocal Sobolev spaces of Definition~\ref{normHspdelta}  correspond to the fractional Sobolev spaces $H^{s, p}(\R^n)$ consisting of $L^p$-functions with weak fractional gradient in $L^p$, which are known to be equivalent to the Bessel potential spaces for $p\in (1,\infty)$ \cite{Shieh1, KrS22, Comi1, Comi3}; in formulas,
$$H^{s, p, \delta}(\R^n) = H^{s, p}(\R^n).$$

We point out that for $s \in (0,1)$ and $p \in [1,\infty)$, the Definition~\ref{de:Hspd} is different from how nonlocal Sobolev spaces are introduced in \cite[Definition~3.3]{BeCuMC22}, where  the authors use the closure of $C_c^{\infty}(\R^n)$ functions under the norm in~\eqref{normHspdelta}. However, both definitions are equivalent for Lipschitz domains as the following density result shows. It corresponds to a nonlocal version of the Meyers-Serrin theorem for classical Sobolev spaces, and the proof, which is based on approximate extension, can be found in Appendix~\ref{appendix: density}.
\begin{theorem}\label{th:density}
Let $s\in [0,1)$, $p \in [1,\infty)$ and let $\O \subset \R^n$ be a bounded Lipschitz domain or $\O=\R^n$. Then, for every $u \in \Hspd(\O)$, there exists a sequence $\{\varphi_j\}_{j \in \N} \subset C_c^{\infty}(\R^n)$ that converges (when restricted to $\O_\d$) to $u$ in $\Hspd(\O)$.
\end{theorem}

An important ingredient for the analysis of the nonlocal gradient are suitable versions of the fundamental theorem of calculus (FTC). For the case $s \in (0,1)$, this has been proven in~\cite[Proposition~4.4]{BeCuMC22}. Now we generalize it to $s=0$ as well, which is needed to obtain a nonlocal Poincar\'{e} inequality independent of the fractional parameter. The proof takes inspiration from the arguments in \cite[Appendix]{BeCuMC22}. 

\begin{proposition}[Nonlocal FTC for $s=0$]\label{prop:ftoczero} 
There is a function $W_\delta \in L^{\infty}(\R^n;\R^n)$ such that every $\phi \in \Scal(\R^n)$ can be expressed as
\[
\phi=-RD^0_\d \phi +W_\delta\ast D_\delta^0\phi,
\]
where $\widehat{R\psi}(\xi)=\frac{i\xi \cdot \widehat{\psi}(\xi)}{\abs{\xi}}$ denotes the Riesz transform of $\psi \in \Scal(\R^n;\R^n)$ and $W_\delta\ast D_\delta^0\phi$ is the sum of the componentwise convolutions of $W_\d$ and $D^0_\d \phi$. 
\end{proposition}

\begin{proof} Consider the tempered distribution $Z_\d \in \Scal'(\R^n;\C^n)$, given by
\[
\inner{Z_\d,\eta} = \lim_{r \downarrow 0} \int_{B(0,r)^c} \left(\frac{i\xi}{\abs{\xi}}-\frac{i\xi}{2\pi\abs{\xi}^2\widehat{Q}^0_\d(\xi)}\right) \eta(\xi)\,d\xi \quad \text{for $\eta \in \Scal(\R^n)$}.
\]
We may decompose $Z_\d$ into the sum of another tempered distribution $Y_\d \in \Scal'(\R^n;\C^n)$ given by
\[
\inner{Y_\d,\eta} = \lim_{r \downarrow 0} \int_{B(0,r)^c} \mathbbm{1}_{B(0,1)}(\xi)\frac{-i\xi}{2\pi\abs{\xi}^2\widehat{Q}^0_\d(0)}\eta(\xi)\,d\xi \quad \text{for $\eta \in \Scal(\R^n)$},
\]
and the locally integrable function $X_\d \in L^1_{\rm loc}(\R^n;\C^n)$, 
\[
\xi \mapsto \left(\frac{i\xi}{2\pi|\xi|^2}\left(\frac{1}{\widehat{Q}^0_\d(0)}-\frac{1}{\widehat{Q}^0_\d(\xi)}\right) + \frac{i\xi}{\abs{\xi}}\right)\mathbbm{1}_{B(0,1)}(\xi) +\frac{i\xi}{|\xi|}  \left(1-\frac{1}{2\pi\abs{\xi}\widehat{Q}^0_\d(\xi)}\right)\mathbbm{1}_{B(0,1)^c}(\xi).
\]
The inverse Fourier transform of $Y_\d$ corresponds to a bounded function; for the case $n\geq 2$, this is because $Y_\d$ agrees with an integrable function, whereas for the case $n=1$ this follows from~\cite[Lemma~A.1\,$b)$]{BeCuMC22}. Moreover, we can show that $X_\d$ is actually integrable. For the first term this follows from the fact that $\widehat{Q}^0_\d$ is smooth and strictly positive, cf.~Remark~\ref{rem:Qdeltas}. For the second term, we use \eqref{Rhatsd} to write for $\abs{\xi}\geq 1$
\[
1-\frac{1}{2\pi\abs{\xi}\widehat{Q}^0_\d(\xi)} = 1-\frac{1}{1+2\pi\abs{\xi}\widehat{R}^0_\d(\xi)}=\frac{2\pi\abs{\xi}\widehat{R}^0_\d(\xi)}{1+2\pi\abs{\xi}\widehat{R}^0_\d(\xi)},
\]
which is integrable by Lemma~\ref{le:decayR}. We conclude that $X_\d$ also has a bounded inverse Fourier transform.

All in all, we conclude that there is a $W_\d \in L^{\infty}(\R^n;\R^n)$ such that 
\[
\widehat{W_\d} = Z_\d.
\]
Note that $W_\d$ takes values in $\R^n$ as $\inner{Z_\d,\eta(-\,\cdot)}=\overline{\inner{Z_\d,\eta}}$ for $\eta \in \Scal(\R^n)$. Finally, using Proposition \ref{Prop: convolution with the classical gradient} for the Fourier transform of $D^0_\d \f$, we have for $\phi, \eta\in \Scal(\R^n)$ and $\psi = -RD^0_\d \phi + W_\delta\ast D^0_\d \phi \in \Scal'(\R^n)$ that
\begin{align*}
\inner{\widehat{\psi},\eta} &= \int_{\R^n} \frac{-i\xi}{\abs{\xi}}\cdot \widehat{D^0_\d \phi(\xi)}\eta(\xi)\,d\xi+\inner{Z_\d\widehat{D^0_\d \phi},\eta}\\
&=\lim_{r \downarrow 0} \int_{B(0,r)^c} \frac{-i\xi}{2\pi\abs{\xi}^2\widehat{Q}^0_\d(\xi)}\cdot \widehat{D^0_\d \phi}(\xi)\eta(\xi)\,d\xi\\
&= \lim_{r \downarrow 0} \int_{B(0,r)^c} \frac{-i\xi}{2\pi\abs{\xi}^2\widehat{Q}^0_\d(\xi)}\cdot \widehat{Q}^0_\d(\xi)2\pi i \xi\widehat{\phi}(\xi)\eta(\xi)\,d\xi=\int_{\R^n} \widehat{\phi}(\xi)\eta(\xi)\,d\xi,
\end{align*}
which proves $\psi = \phi \in \Scal(\R^n)$ after taking the inverse Fourier transform.
\end{proof}

\subsection{Complementary-value spaces}\label{subsec:complementary}
Our study of variational problems involving the nonlocal gradient is carried out on affine subspaces of $\Hspd(\O)$ satisfying a complementary-value condition. For $\O \subset \R^n$ open and bounded, let
\[
\O_{-\d} = \{ x \in \O \, : \, d(x,\partial\O)> \d\}
\]
 assuming that $\delta>0$ is small enough so that $\Omega_{-\d}$ is non-empty; for an illustration of $\Omega$ and its inner and outer collar, see Figure~\ref{fig}. We define for $s\in[0,1)$ and $p \in [1,\infty)$,
\begin{align}\label{Hspd0}
\Hspd_0(\O) = \overline{C_c^{\infty}(\O_{-\d})}^{\Hspd(\O)}
\end{align}
and for $g \in \Hspd(\O)$ the complementary-value space
\[
\Hspd_g(\O) = g+ \Hspd_0(\O).
\]
In a similar vain, we set 
\begin{equation}\label{eq:classicalspace2}
H^{1,p,\d}_0(\Omega):=\overline{C_c^{\infty}(\O_{-\d})}^{W^{1,p}(\O_\d)},
\end{equation}
which will be used to study the asymptotics $s \to 1$.

In order to avoid confusion, we clarify that the notation used in this document for $H^{s,p,\d}_g(\O)$ slightly differs from the one used in \cite{BeCuMC22,BeCuMC22b}, where the same spaces were denoted by $H^{s,p,\d}_g(\O_{-\d})$.

\begin{figure}
\centering      

\resizebox{7cm}{!}{
\begin{tikzpicture}[x=0.75pt,y=0.75pt,yscale=-1,xscale=1]

\begin{scope}
\draw[line width=0.75, double, draw=colora,  double distance = 10mm]  (230,90) .. controls (278,47) and (369,53) .. (410,90) .. controls (451,127) and (473,279) .. (390,230) .. controls (307,181) and (256,280) .. (220,240) .. controls (184,200) and (182,133) .. (230,90) -- cycle ;
\end{scope}

\begin{scope}
\clip (230,90) .. controls (278,47) and (369,53) .. (410,90) .. controls (451,127) and (473,279) .. (390,230) .. controls (307,181) and (256,280) .. (220,240) .. controls (184,200) and (182,133) .. (230,90) -- cycle ;

\draw[line width=0.75, double, draw=colorb,  double distance = 10mm]  (230,90) .. controls (278,47) and (369,53) .. (410,90) .. controls (451,127) and (473,279) .. (390,230) .. controls (307,181) and (256,280) .. (220,240) .. controls (184,200) and (182,133) .. (230,90) -- cycle ;
\end{scope}

\draw [black!100, thick] (230,90) .. controls (278,47) and (369,53) .. (410,90) .. controls (451,127) and (473,279) .. (390,230) .. controls (307,181) and (256,280) .. (220,240) .. controls (184,200) and (182,133) .. (230,90) -- cycle ;

\draw (213,152) node [anchor=north west] {\color{colorb} $\Omega_{-\delta}$};
\draw (193,152) node [anchor=north west] { $\Omega$};
\draw (171,152) node [anchor=north west] {\color{colora} $\Omega_{\delta}$};
\draw [colorb] (392,100) -- (403,85);
\draw [colora] (403,85) -- (415, 69);
\draw (395,89) node [anchor=north west] {\color{colorb} $\delta$};
\draw (406,72) node [anchor=north west] {\color{colora} $\delta$};
\end{tikzpicture}}
\caption{Illustration of the set $\Omega$ with its nonlocal closure $\Omega_\delta$,  its nonlocal boundary $\Omega_\delta\setminus \Omega$, and the collar $\Omega_\delta\setminus \Omega_{-\delta}$ of thickness $2\delta$, where complementary values are prescribed.}\label{fig}
\end{figure}
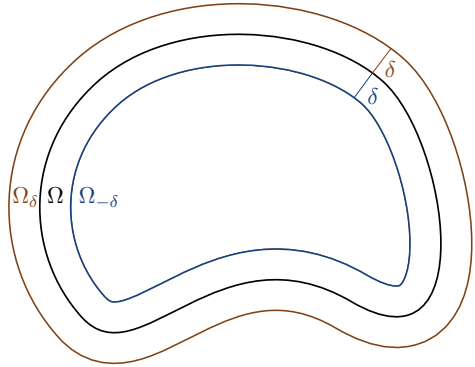

 When $\O_{-\d}$ is a Lipschitz domain, these affine subspaces comprise exactly those functions in $\Hspd(\O)$ that have prescribed values in $\O_\d \setminus \O_{-\d}$. Indeed, for the case $s=1$, it is well-known that 
\[
H^{1,p,\d}_0(\Omega)=\{u \in W^{1,p}(\O_\d)\,:\, u=0 \ \text{a.e.~in $\O_\d \setminus \O_{-\d}$}\},
\]
whereas the case $s \in [0,1)$ is treated in the next statement, which we prove in Appendix~\ref{appendix: density}.
 
\begin{proposition}\label{prop:densitycomplement}
Let $s \in [0,1)$, $p \in [1,\infty)$ and $\O \subset \R^n$ be open and bounded such that $\O_{-\d}$ is a Lipschitz domain. Then, 
\[
H^{s,p,\d}_g(\O) = \{ u \in H^{s,p,\d}(\O)\,:\, u=g \ \text{a.e.~in $\O_\d \setminus \O_{-\d}$}\}.
\]
\end{proposition}

It may be surprising at first glance that we prescribe values in a collar of width $2\d$ around the boundary of $\Omega$, yet, this choice leads to a natural treatment of the nonlocal variational problems in this paper, as can be seen for instance from the Poincar\'{e} inequality in~\cite[Theorem~6.2]{BeCuMC22} and the Euler-Lagrange equations in~\cite[Theorem~8.2]{BeCuMC22}.

For a discussion of relevant properties and useful results on these function spaces, like Poincar\'{e} inequalities and compact embeddings, we refer to \cite{BeCuMC22, BeCuMC22b}. Apart from those, there is the following Leibniz rule from~\cite[Lemma~3.2 and 3.3]{BeCuMC22b}, which we will use among other things to enforce complementary-values via cut-off procedures. 
\begin{lemma}[Nonlocal Leibniz rule]\label{le:leibniz}
Let $s \in [0,1)$, $\d >0$, $p \in [1,\infty]$, and $\Omega\subset \R^n$ open. If $u \in \Hspd(\O)$ and $\chi \in C^{\infty}_c(\R^n)$, then $\chi u \in \Hspd(\O)$ with 
\[
D^s_\d(\chi u) = \chi D^s_\d u + K_{\chi}(u),
\]
where $K_{\chi}:L^p(\O_\d) \to L^p(\O;\R^n)$ is the bounded linear operator given by
\begin{align*}\label{Kchi}
K_{\chi}(u)(x) = c_{n,s}\int_{B(x,\d)}u(y)\frac{\chi(x)-\chi(y)}{\abs{x-y}^{n+s}}\frac{x-y}{\abs{x-y}}w_\d(x-y)\,dy \quad\text{for $x \in \O$},
\end{align*}
and there is a $C>0$ such that
\[
\norm{K_{\chi}(u)}_{L^p(\O;\R^n)} \leq C \Lip(\chi) \norm{u}_{L^p(\O_\d)}.
\]
\end{lemma}
\begin{proof}
The statement for $u \in C_c^{\infty}(\R^n)$ with $s \in (0,1)$ and the bound for $K_{\chi}$ follow immediately from \cite[Lemma~3.2 and 3.3]{BeCuMC22b} (the arguments remain valid for unbounded sets). We can extend it to $u \in H^{s,p,\d}(\O)$ via a distributional argument as in Lemma~\ref{le:cutoff}. The case $s=0$ can be proven analogously.
\end{proof}
In a similar spirit, one obtains with a slight abuse of notation that
\begin{align}\label{Leibniz_div}
\Div^s_\d(\chi u) = \chi\Div^s_\d u + K_{\chi}(u^\intercal)
\end{align}
for $u \in\Hspd(\O;\R^n)$ and $\chi\in C_c^{\infty}(\R^n)$; here $\zeta^\intercal$ indicates the transpose of a vector $\zeta\in \R^n$.

As a consequence of the Leibniz rule above, we can prove that in complementary-value spaces, weak convergence of nonlocal gradients improves to strong convergence in the strip where the values are prescribed. The following result shows natural parallels with~\cite[Lemma~2.12]{KrS22} in the context of Riesz fractional gradients. 
\begin{lemma}[Strong convergence in the collar]\label{le:strongoutside}
Let $s \in (0,1)$, $p \in (1,\infty)$, $\Omega \subset \R^n$ open and bounded, $O\subset \Omega$ open with $\Omega_{-\d} \Subset O$, and $g \in H^{s,p,\d}(\O)$. If $\{u_j\}_{j \in \N} \subset H^{s,p,\d}_g(\O)$ converges weakly to $u$ in $H^{s,p,\d}(\O)$, then 
\[
D^s_\d u_j \to D^s_\d u \quad \text{in $L^p(\O \setminus O;\Rn)$.}
\]
\end{lemma}

\begin{proof}
Due to linearity, it suffices to prove the statement for the special case $u=0$ and $g=0$. Let us consider  therefore a sequence $\{u_j\}_{j\in \N}\subset H^{s,p,\d}_0(\O)$ with $u_j \weakto 0$ in $H^{s,p,\d}(\Omega)$. With $\chi \in C_c^{\infty}(O)$ a cut-off function with $\chi \equiv 1$ on $\Omega_{-\d}$, we obtain that $u_j=\chi u_j$ for $j\in \N$ and $u_j \to 0$ in $L^p(\O_{\d})$ as a consequence of \cite[Theorem 7.3]{BeCuMC22}. Hence, by Lemma~\ref{le:leibniz}, 
\begin{align*}
\norm{D^s_\d u_j}_{ L^p(\Omega\setminus O;\R^n)}&=\norm{D^s_\d (\chi u_j)}_{L^p(\Omega\setminus O;\R^n)}\\
&\leq \norm{D^s_\d (\chi u_j) - \chi D^s_\d u_j}_{L^p(\Omega;\Rn)}= \|K_{\chi}(u_j)\|_{L^p(\Omega;\R^n)} \to 0\quad \text{as $j\to \infty$,}
\end{align*}
exploiting the continuity of $K_{\chi}:L^p(\O_\d) \to L^p(\O;\Rn)$.
\end{proof}

\subsection[Connection between nonlocal and classical Sobolev spaces]{Connection between nonlocal and classical Sobolev spaces} \label{Section: Connection between nonlocal and classical Sobolev spaces}
One of the key tools for our analysis is the following proposition, which allows us to switch between nonlocal and classical gradients and is the technical basis for an effective translation mechanism.  It is the counterpart of \cite[Proposition~3.1]{KrS22}, where fractional gradients and their relation with classical ones are analyzed. 

We first introduce the operator
\begin{align}\label{Psdelta}
\Pcal^s_\d:\Scal(\R^n) \to \Scal(\R^n), \quad \phi \mapsto \left(\frac{\widehat{\phi}}{\widehat{Q}^s_\d}\right)^{\vee},
\end{align}
which is well-defined since $1/\widehat{Q}^s_\d$ is a smooth function with polynomially bounded derivatives (cf.~Remark~\ref{rem:Qdeltas} and \cite[Eq.~(29)]{BeCuMC22}). Moreover, as a consequence of  the Fourier representation, 
\begin{equation}\label{eq:psdinverse}
\Pcal^s_\d (Q^s_\d*\phi) = Q^s_\d *(\Pcal^s_\d \phi) = \phi \quad \text{for all $\phi \in \Scal(\R^n)$},
\end{equation}
which implies, in particular, that  $D^s_\d (\Pcal^s_\d \phi)  = \nabla (Q^s_\d *(\Pcal^s_\d \phi))=\nabla \phi$. We now extend these properties to the Sobolev spaces.

\begin{theorem}[Translating between nonlocal and classical gradients]\label{prop:connection}
Let $s\in (0,1)$, $p \in [1,\infty]$ and $\O \subset \R^n$ be open. The following two statements hold: 
\begin{itemize}
\item[$(i)$] The operator $\Qcal^s_\d:H^{s,p,\d}(\O) \to W^{1,p}(\O)$, $u \mapsto Q^s_\d *u$ is bounded and if $u \in H^{s,p,\d}(\O)$, then
 $v=\Qcal^s_\d  u$ satisfies $\nabla v= D^s_\d u$ on $\O$. \\[-0.3cm]
\item[$(ii)$] The operator $\Pcal^s_\d$ in~\eqref{Psdelta}  can be extended to a bounded linear operator from $W^{1,p}(\R^n)$ to $H^{s,p,\d}(\R^n)$ such that $\Pcal^s_\d = (\Qcal^s_\d)^{-1}$, i.e.,
\[
\Pcal^s_\d \Qcal^s_\d u = u \ \ \text{for $u \in H^{s,p,\d}(\R^n)$} \quad \text{and} \quad \Qcal^s_\d \Pcal^s_\d v =v \ \ \text{for $v \in W^{1,p}(\R^n)$};
\]
in particular, if $v \in W^{1,p}(\R^n)$, then $u=\Pcal^s_\d v$ satisfies $D^s_\d u = \nabla v$ on $\R^n$.
\end{itemize}
\end{theorem}

	\begin{proof} 
\textit{Part (i):}		Let $u \in H^{s,p,\d}(\O)$, then $v = \Qcal^s_\d u \in L^p(\Omega)$ since $Q^s_\d \in L^1(\R^n)$. For every $\phi \in C_c^{\infty}(\O;\R^n)$, we find that
\begin{align*}
\int_{\O}v \Div \phi\,dx &= \int_{\O_{\d}} u  \,(Q^s_\d * \Div \phi) \,dx\\
&= \int_{\O_{\d}} u \Div^s_\d \phi\,dx =-\int_{\O}D^s_\d u \cdot \phi\,dx,
\end{align*}
where the first identity uses Fubini's theorem, the second one follows from Proposition~\ref{Prop: convolution with the classical gradient}, and the third one is simply the definition of the weak nonlocal gradient. This proves $v \in W^{1,p}(\O)$ with $\nabla v = D^s_\d u$ on $\O$. The boundedness of $\Qcal^s_\d$ follows from Young's convolution inequality. \smallskip

\textit{Part (ii):} Since $\Pcal^s_\d$ is the inverse of the mapping $\Qcal^s_\d$ on $\Scal(\R^n)$ (cf.~\eqref{eq:psdinverse}), it is sufficient to prove that $\Qcal^s_\d$ is boundedly invertible. Indeed, we can then find the suitable extension by setting $\Pcal^s_\d:=(\Qcal^s_\d)^{-1}$. Since $\Qcal^s_\d$ is bounded by part~$(i)$, we only need to prove bijectivity to deduce the statement via Banach's isomorphism theorem. \smallskip

\textit{Step 1: Injectivity.} Suppose that $\Qcal^s_\d u =Q^s_\d * u=0$ for $u \in H^{s,p,\d}(\R^n)$. 
Then, $D^s_\d u = \nabla (Q^s_\d * u)=0$, and in particular,
\[
\int_{\R^n} u\Div^s_\d\phi\,dx = 0 \quad \text{for all $\phi \in C_c^{\infty}(\R^n;\R^n)$};
\]
by density, this also holds for all $\phi \in \Scal(\R^n;\R^n)$. By taking any $\psi \in C_c^{\infty}(\R^n;\R^n)$ and setting $\phi = \Pcal^s_\d \psi \in \Scal(\R^n;\R^n)$, we obtain
\[
0=\int_{\R^n} u\Div^s_\d\phi\,dx=\int_{\R^n} u\Div^s_\d \Pcal^s_\d \psi\,dx=\int_{\R^n} u\Div \psi\,dx.
\]
 Hence, $u$ is constant. Together with  $Q^s_\d*u=0$, this shows that $u=0$ and proves the injectivity of $\Qcal_\delta^s$. \smallskip 

\textit{Step 2: Surjectivity.} Take $v \in W^{1,p}(\R^n)$ and $\chi \in C_c^{\infty}(\R^n)$ an even function with $\chi \equiv 1$ on $B(0,1)$. Define the functions $\phi_1,\phi_2 \in \Scal(\R^n)$  by 
\[
\phi_1 = -\chi^{\vee} \quad \text{and} \quad \phi_2 = \left(\frac{\chi}{\widehat{Q}^s_\d}+(1-\chi)\left(\frac{1}{\widehat{Q}^s_\d}-\abs{2\pi\,\cdot\,}^{1-s}\right)\right)^{\vee};
\]
here, $\phi_2 \in \Scal(\R^n)$ since \[
\frac{1}{\widehat{Q}^s_\d}-\abs{2\pi\,\cdot\,}^{1-s}=\frac{-\abs{2\pi\,\cdot\,}^{1-s}\widehat{R}^s_\d}{\widehat{Q}^s_\d} \quad \text{in $B(0,1)^c$}
\]
 in view of~\eqref{Rhatsd}, 
and $\widehat{R}^s_\d$ agrees with a Schwartz function on $B(0,1)^c$ by \eqref{eq:decayRhat}. Note also that $\phi_1$ and $\phi_2$ are real-valued and even, since the same holds for their Fourier transforms. Because $(-\Delta)^{\frac{1-s}{2}}$ may be extended to a bounded linear operator from $W^{1,p}(\R^n)$ to $H^{s,p}(\R^n)$ (cf.~\cite[Proposition~3.1\,$(ii)$]{KrS22}) and $H^{s,p}(\R^n)=H^{s,p,\d}(\R^n)$ by Lemma~\ref{H0sp=Hospd} and Remark~\ref{rem:hspd=hsp}, we can define 
\[
w:=(-\Delta)^{\frac{1-s}{2}}v+\phi_1*(-\Delta)^{\frac{1-s}{2}}v+\phi_2*v \in H^{s,p,\d}(\R^n). 
\]
Using Fubini's theorem and the duality for the fractional Laplacian (see e.g.,~\cite[Eq.~(3.6)]{KrS22}), we find for $\phi \in C_c^{\infty}(\R^n;\R^n)$
\begin{align*}
\int_{\R^n}w\Div^s_\d \phi\,dx &= \int_{\R^n} v\bigl((-\Delta)^{\frac{1-s}{2}}\Div^s_\d \phi + (-\Delta)^{\frac{1-s}{2}}(\phi_1*\Div^s_\d \phi)+\phi_2*\Div^s_\d \phi\bigr)\,dx\\
&=\int_{\R^n} v \Div \phi\,dx,
\end{align*}
where the last inequality follows from
\begin{align*}
&\left( (-\Delta)^{\frac{1-s}{2}} 
\Div^s_\d \phi + (-\Delta)^{\frac{1-s}{2}}(\phi_1*\Div^s_\d \phi)+\phi_2*\Div^s_\d
 \phi\right)^{\wedge}(\xi) \\
&\qquad \qquad = \left(\abs{2\pi\xi}^{1-s} + \abs{2\pi \xi}^{1-s} \widehat{\phi_1}(\xi) + \widehat{\phi_2}(\xi)\right)\widehat{Q^s_\d}(\xi)2\pi i \xi \cdot \widehat{\phi}(\xi)=2\pi i \xi \cdot \widehat{\phi}(\xi) = \widehat{\Div \phi}(\xi).
\end{align*}
We conclude that $D^s_\d w = \nabla v$, which means that $\Qcal^s_\d  w -v\equiv c$ for some $c \in \R$; if $p<\infty$, then $c=0$ since both $\Qcal^s_\d  w$ and $v$ lie in $L^p(\R^n)$. Therefore,  we obtain
\[
\Qcal^s_\d  \left(w-\frac{c}{\norm{Q^s_\d}_{L^1(\R^n)}}\right)=v,
\]
which shows the surjectivity of $\Qcal^s_\d$ and finishes the proof.
\end{proof}

\begin{remark}\label{rem:connection}
a) Note that the proof of the fractional version of $(i)$ in~\cite[Proposition~3.1\,$(i)$]{KrS22} has to deal with a technical difficulty that the Riesz kernel $I_{1-s}$ is not integrable as opposed to $Q^s_\d$. Therefore, the convolution of $I_{1-s}$ with an $L^p$-function for large $p$ is not always well-defined, whereas $Q^s_\delta$ can be convolved with any $L^p$-function. In particular, there is also no perfect identification between $H^{s,p}(\R^n)$ and $W^{1,p}(\R^n)$ that turns fractional gradients into classical gradients as for the nonlocal case in part $(ii)$ above. \smallskip

b) Regarding part $(ii)$, when $p<\infty$ then the extension of $\Pcal^s_\d$ can also be seen as the unique extension via density. Moreover, if $\Omega$ is a Lipschitz domain, then any $v \in W^{1,p}(\Omega)$ can be extended to a function in $W^{1,p}(\R^n)$, after which we can apply the result to find a $u \in H^{s,p,\d}(\O)$ with $D^s_\d u = \nabla v$ on $\O$. \smallskip

c) The proof of the surjectivity in part~$(ii)$ shows that $\Pcal^s_\d v$ corresponds, up to a constant, to
\[
(-\Delta)^{\frac{1-s}{2}}v+\phi_1*(-\Delta)^{\frac{1-s}{2}}v+\phi_2*v,
\]
for $v \in W^{1,p}(\R^n)$; when $p<\infty$, then the correspondence is even an identity, given that there are no non-zero constants in $L^p(\R^n)$. 

As a particular consequence of this observation, along with the fact that the convolution with a periodic function remains periodic, we observe that both $\Qcal^s_\d$ and $\Pcal^s_\d$ preserve periodicity. Precisely, if  $Y$ denotes the unit cube $(0,1)^n$, and $W^{1,\infty}_{\#}(Y)$ and $H_\#^{s, \infty, \d}(Y)$ comprise all $Y$-periodic functions in $W^{1,\infty}(\R^n)$ and $H^{s, \infty, \d}(\R^n)$, respectively, then there is a bijection between the gradients of $W^{1,\infty}_{\#}(Y)$-functions and the nonlocal gradients of $H^{s,\infty,\d}_{\#}(Y)$-functions. \smallskip

d) For $\phi \in \Scal(\R^n)$, it holds that
\begin{equation}\label{eq:vsdrepr}
\Pcal^s_\d \phi (x) = \int_{\R^n} V^s_\d(x-y)\cdot \nabla \phi (y) \, dy \quad \text{for $x \in \R^n$},
\end{equation}
where $V^s_\d \in C^{\infty}(\Rn \backslash \{0\})$ is the kernel from the nonlocal version of the fundamental theorem of calculus \cite[Theorem 4.5]{BeCuMC22}. Indeed, this follows directly from the formula for the Fourier transform of $V^s_\d$ in \cite[Theorem~5.9]{BeCuMC22}. The representation in \eqref{eq:vsdrepr} extends naturally to functions in $W^{1,p}(\R^n)$ with compact support, given that $V^s_\d$ is locally integrable.\smallskip

 e) 
The translation procedure of Theorem~\ref{prop:connection} allows us to give an alternative proof for the nonlocal Poincar\'e inequality in \cite[Theorem~6.2]{BeCuMC22}. Since $\mathcal{Q}^s_\d$ maps $H^{s,p,\d}_0(\O)$ into $W^{1,p}_0(\O)$, we infer from the classical Poincar\'{e} inequality that 
\begin{align*}
\norm{u}_{L^p(\O)} = \norm{\Pcal^s_\d \mathcal{Q}^s_\d u}_{L^p(\O)} \leq C\norm{\mathcal{Q}^s_\d u}_{W^{1,p}_0(\O)} \leq C \norm{\nabla \mathcal{Q}^s_\d u}_{L^p(\O;\R^n)} = C \norm{D^s_\d u}_{L^p(\O;\R^n)},
\end{align*}
 for any $u \in H^{s,p,\d}_0(\O)$ with a constant $C>0$ depending on $s, \delta, p$, and $\Omega$. 
\end{remark}

We conclude this section with a compactness result that will be used below in the proof of~Theorem~\ref{theo:characterization}.
\begin{lemma}\label{lem:compactnessV}
Let $s\in (0,1)$, $p \in (1,\infty)$ and $\O \subset \R^n$ be open and bounded. 
If $\{v_j\}_{j \in \N} \subset W^{1,p}(\R^n)$ is a bounded sequence, then $\{\Pcal^s_\d v_j\}_{j \in \N}$ (when restricted to $\Omega_\d$) is relatively compact in $L^p(\Omega_\d)$.
\end{lemma}
\begin{proof}
Let $\chi \in C_c^{\infty}(\R^n)$ with $\chi \equiv 1$ on $\Omega_\d$ and set $R>0$ such that $\supp(\chi) \subset B(0,R-\d)$. Since $\{\Pcal^s_\d v_j\}_j$ is bounded in $H^{s,p,\d}(\R^n)$ by Theorem~\ref{prop:connection}\,$(ii)$, the sequence $\{\chi (\Pcal^s_\d v_j)\}_{j \in \N}$ is bounded in $H^{s,p,\d}_0(B(0,R))$ by Lemma~\ref{le:leibniz}. The relative compactness of $\{\chi (\Pcal^s_\d v_j)\}_{j \in \N}$ in $L^p(B(0,R))$ now follows from ~\cite[Theorem 7.3]{BeCuMC22} and since $\chi\equiv 1$ on $\Omega_\d$, the statement follows.
\end{proof}

\subsection{Connection between nonlocal and fractional gradients}\label{subsec:connection_localfrac}
After the comparison of the nonlocal gradients with classical weak gradients, let us now discuss their connection with the Riesz fractional gradient. We start by recalling that the nonlocal gradient is a truncated version of the latter. 

In the following, let $p\in [1, \infty)$ and $s\in (0,1)$. The upcoming lemma presents the equivalence between the nonlocal and fractional Sobolev spaces $H^{s,p}(\R^n)$ and $H^{s,p,\d}(\R^n)$ and also a version with prescribed complementary values.
To recall the definition of the fractional and nonlocal complementary-value spaces, we have for $\Omega\subset \R^n$ open and bounded that
$H^{s, p}_0(\Omega)$ comprises all functions $u\in H^{s, p}(\R^n)$ such that $u=0$ a.e. in $\Omega^c$ and $H^{s, p, \delta}_0(\Omega)$ is given as in~\eqref{Hspd0}. We mention that one of the inclusions was already provided by~\cite[Proposition 3.5]{BeCuMC22}. For the sake of the reader, we show here a complete proof.
\begin{lemma}\label{le:nonlocalfrac}
It holds that
\begin{align}\label{H0sp=Hospd}
H^{s, p}(\R^n) = H^{s, p, \delta}(\R^n)
\end{align}
with equivalent norms, and 
\begin{align}\label{comparison_nonlocalfrac}
D_\delta^s u = D^su + \nabla R_\delta^s\ast u
\end{align}
for all $u\in H^{s, p}(\R^n)=H^{s, p,\delta}(\R^n)$ with $\nabla R_\delta^s\in C^\infty(\R^n;\R^n)\cap L^1(\R^n;\R^n)$ as in~\eqref{eq:gradR}. Moreover, for $\Omega \subset \R^n$ open and bounded with $\Omega_{-\d}$ Lipschitz, it holds that
\[
H_0^{s, p}(\Omega_{-\d}) = H_0^{s, p, \delta}(\Omega),
\]
with equivalent norms, and \eqref{comparison_nonlocalfrac} holds for $u \in H_0^{s, p}(\Omega_{-\d}) = H_0^{s, p, \delta}(\Omega)$ on $\Omega$.
\end{lemma}
\begin{proof} Let $\varphi\in C_c^\infty(\R^n)$. Since by \eqref{eq:comparison_nonlocalfrac0}
\begin{align*}\label{comparison_nonlocalfrac2}
D^s_\delta \varphi-D^s\varphi = \nabla R_\delta^s\ast \varphi,
\end{align*}
we obtain the estimates
\begin{align*}
\norm{D^s\varphi}_{L^p(\R^n;\R^n)} \leq \norm{D^s_\delta\varphi}_{L^p(\R^n;\R^n)} + \norm{\nabla R_\delta^s}_{L^1(\R^n;\R^n)} \norm{\varphi}_{L^p(\R^n)} \leq C\norm{\varphi}_{H^{s, p, \delta}(\R^n)} 
\end{align*}
and
\begin{align*}
\norm{D^s_\d\varphi}_{L^p(\R^n;\R^n)} \leq \norm{D^s\varphi}_{L^p(\R^n;\R^n)} + \norm{\nabla R_\delta^s}_{L^1(\R^n;\R^n)} \norm{\varphi}_{L^p(\R^n)} \leq C\norm{\varphi}_{H^{s, p}(\R^n)} 
\end{align*}
with a constant $C>0$. In light of the density of $C_c^{\infty}(\R^n)$ in $H^{s,p}(\R^n)$ and $H^{s,p,\d}(\R^n)$ (see~\cite[Theorem~2.7]{KrS22} and Theorem~\ref{th:density}), the identity \eqref{H0sp=Hospd} and \eqref{comparison_nonlocalfrac} follow via approximation. For the case of a bounded domain, we note that 
\[
H^{s,p,\d}_0(\O)=\{u \in H^{s,p,\d}(\O)\,:\, u = 0 \ \text{a.e.~in $\Omega_\d \setminus \Omega_{-\d}$}\}
\]
since $\Omega_{-\d}$ is Lipschitz (cf.~Proposition~\ref{prop:densitycomplement}). Observe also that for any $u \in H^{s,p,\d}_0(\O)$ its extension $\bar{u}$ to $\R^n$ by zero lies in $H^{s,p,\d}(\R^n)$ with
\[
\norm{\bar{u}}_{L^p(\R^n)} + \norm{D^s_\d \bar{u}}_{L^p(\R^n;\R^n)}=\norm{u}_{L^p(\O)}+\norm{D^s_\d u}_{L^p(\O;\R^n)},
\]
since $D^s_\d \bar{u}$ is simply the extension of $D^s_\d u$ by zero. Hence, we may identify
\[
H^{s,p,\d}_0(\O)=\{u \in H^{s,p,\d}(\R^n)\,:\, u = 0 \ \text{a.e.~in $(\Omega_{-\d})^c$}\},
\]
after which the equality with $H^{s,p}(\Omega_{-\d})$ becomes obvious given \eqref{H0sp=Hospd}.
\end{proof}
\begin{remark}\label{rem:hspd=hsp}
We mention that it also holds that
\[
H^{s,\infty}(\R^n) = H^{s,\infty,\d}(\R^n),
\]
with equivalent norms and $D_\delta^s u = D^su + \nabla R_\delta^s\ast u$ for $u \in H^{s,\infty}(\R^n)$. This can be proven via a distributional approach instead of utilizing density as above.
\end{remark}

 As already indicated in the introduction, Lemma~\ref{subsec:connection_localfrac} opens up a new proof strategy for some of the results in this paper. Instead of exploiting well-known result for problems involving classical gradients, one can resort to established findings in the fractional setting. We illustrate this approach below by presenting an alternative proof for the characterization of lower semicontinuity in Section~\ref{section:wls}, which follows as a corollary of~\cite[Theorems~4.1 and~4.5]{KrS22}.
An analogous reasoning could also be used, for instance, to deduce the relaxation below in Corollary~\ref{cor:relaxation} from \cite[Theorem~1.2]{KrS22}.
Note that the transfer of results between the nonlocal and fractional set-up also works in the reverse direction, giving rise to analogues of the general $\Gamma$-convergence statement in Theorem~\ref{th:gamgeneral} and homogenization result of Corollary~\ref{cor:homogenization}.

\section{Asymptotics of the nonlocal gradient and applications}\label{sec:asymptotics}

Our next goal is to study the localization of the nonlocal gradient as $s \to 1$, and more generally, to understand how the nonlocal gradient depends on the fractional parameter $s$.
In particular, the findings in this section serve as necessary preparations for proving the $\Gamma$-convergence of nonlocal integral functionals in Section~\ref{sec:Gamma}.

We start by investigating the $s$-dependence of the convolution kernel $Q^s_\d$  from~\eqref{Qdeltas} and its Fourier transform.
\begin{lemma}\label{le:qnorm} 
Let $\eps>0$ and $R>0$. 
\begin{itemize}
\item[$(i)$]  The map $[0,1) \to L^1(\Rn),  \,s\mapsto Q^s_\d$
is continuous with 
\begin{align}\label{convergenceQtwice}
\lim_{s \to 1} \norm{Q^s_\d}_{L^1(\R^n)} =1 \quad \text{ and} \quad \lim_{s\to 1} \norm{Q^s_\d}_{L^1(B(0, \eps)^c)}=0.
\end{align}
\item[$(ii)$] The map $[0,1) \to C(\overline{B(0, R)}), \, s\mapsto \widehat{Q}^s_\d$ is continuous with $\widehat{Q}^s_\d \to 1$ uniformly on $\overline{B(0, R)}$ as $s \to 1$.
\end{itemize}
\end{lemma}
\begin{proof}
As for $(i)$, we calculate first that for $s\in [0,1)$,
\begin{align}\label{L1normQ}
\begin{split}
\norm{Q^s_\d}_{L^1(\R^n)}&=c_{n,s}\int_{B(0,\d)}\int_{\abs{x}}^\d \frac{\overline{w}_{\d}(r)}{r^{n+s}}\,dr\,dx =c_{n,s}\abs{\partial B(0,1)}\int_0^\d \int_\rho^\d \frac{\overline{w}_{\d}(r)}{r^{n+s}}\rho^{n-1}\,dr\,d\rho\\
&=c_{n,s}\abs{\partial B(0,1)}\int_0^\d \frac{\overline{w}_{\d}(r)}{r^{n+s}}\int_0^r \rho^{n-1}\,d\rho\,dr=c_{n,s}\frac{\abs{\partial B(0,1)}}{n}\int_0^\d \frac{\overline{w}_{\d}(r)}{r^{s}}\,dr\\
&=c_{n,s}\,\omega_n \int_0^\d \frac{\overline{w}_\d(r)}{r^s}\,dr,
\end{split}
\end{align}
where $\omega_n$ denotes the volume of the unit ball in $\R^n$.
Since $s\mapsto c_{n,s}$ is continuous on $[0,1)$, cf.~\cite[Lemma~2.4]{BeCuMC21}, it follows via Lebesgue's dominated convergence that $\norm{Q^s_\d}_{L^1(\R^n)}$ depends continuously on $s$. Now, if  $\{s_j\}_{j \in \N} \subset [0,1)$ is a sequence converging to $s\in [0,1)$, we can apply once again Lebesgue's dominated convergence theorem to find that $Q^{s_j}_\d \to Q^s_\d$ pointwise a.e.~as $j \to \infty$. Together with $\lim_{j\to \infty}\norm{Q^{s_j}_\d}_{L^1(\R^n)} =\norm{Q^{s}_\d}_{L^1(\R^n)}$ as shown above, this implies $Q^{s_j}_\d \to Q^s_\d$ in $L^1(\R^n)$ for $j \to \infty$.

To see the first convergence in~\eqref{convergenceQtwice}, we observe that $\mathbbm{1}_{B(0,b_0\d)} \leq w_\d \leq \mathbbm{1}_{B(0,\d)}$ by~\ref{itm:h3} and~\ref{itm:h4}, which gives
\begin{equation}\label{eq:boundinteg}
\frac{(b_0\delta)^{1-s}}{1-s}\leq \int_0^\d \frac{\overline{w}_\d(r)}{r^s}\,dr \leq \frac{\delta^{1-s}}{1-s},
\end{equation}
and exploit $c_{n,s}/(1-s) \to 1/\omega_n$ as $s \to 1$ according to \cite[Lemma~2.4]{BeCuMC21}. The localization of $Q_\delta^s$ for $s\to 1$ follows from a calculation similar to~\eqref{L1normQ} and \eqref{eq:boundinteg}, integrating instead over $B(\epsilon,\d)$ and using that $c_{n,s}/(1-s)$ stays bounded as $s \to 1$. 
\smallskip 

The first part of $(ii)$ can be deduced from the continuity of the map in $(i)$ in combination with the fact that the Fourier transform is a bounded linear operator from $L^1(\R^n)$ to $C_0(\R^n;\C)$. 

Due to~\eqref{convergenceQtwice}, the kernel $Q_\d^s$ behaves like a mollifier, satisfying
\begin{equation}\label{eq:mollifier}
\lim_{s \to 1} \norm{Q^s_\d*\phi - \phi}_{L^{\infty}(\R^n)}=0
\end{equation}
for all $\phi \in \Lipb(\R^n)$, where this convergence is uniform on bounded sets of $\Lipb(\R^n)$. Indeed, 
\begin{align*}
\norm{Q^s_\d*\phi - \phi}_{L^{\infty}(\R^n)} &\leq \norm{(\mathbbm{1}_{B(0,\epsilon)}Q^s_\d)*\phi - \phi}_{L^{\infty}(\R^n)} + \norm{Q^s_\d}_{L^{1}(B(0,\epsilon)^c)}\norm{\phi}_{L^{\infty}(\R^n)}\\
&\leq \epsilon\Lip(\phi)\norm{Q^s_\d}_{L^1(B(0,\epsilon))} + \left(\abs{1-\norm{Q^s_\d}_{L^1(B(0,\epsilon))}} + \norm{Q^s_\d}_{L^{1}(B(0,\epsilon)^c)}\right)\norm{\phi}_{L^{\infty}(\R^n)},
\end{align*}
for any $\epsilon >0$. Considering now $\phi_{\xi}(x) = e^{-2\pi i \xi \cdot x}$ for $\xi \in \overline{B(0,R)}$, we have
\[
\norm{\phi_{\xi}}_{L^{\infty}(\R^n;\C)} + \Lip(\phi_{\xi}) \leq 1+ 2\pi \abs{\xi} \leq 1+2\pi R,
\]
so that by~\eqref{eq:mollifier},
\[
\lim_{s \to 1} \widehat{Q^s_\d}(\xi) = \lim_{s \to 1} (Q^s_\d * \phi_{\xi}) (0) = \phi_{\xi}(0)=1,
\]
uniformly for $\xi \in \overline{B(0,R)}$.
\end{proof}

The next lemma addresses the continuous dependence of the nonlocal gradient and divergence on the fractional parameter in the case of smooth test functions with compact support. Recall the notation $D^1_\d u :=\nabla u$.
\begin{lemma}\label{le:orderconvergencesmooth}
Let $s \in [0,1]$ and $\{s_j\}_{j \in \N} \subset [0,1]$ a sequence converging to $s$. Then, it holds for every $\phi \in C_c^{\infty}(\R^n)$ and $\psi\in C_c^\infty(\Rn;\Rn)$ that 
\begin{align*}D^{s_j}_\d \phi \to D^s_\d \phi\qquad \text{and} \qquad \Div_\delta^{s_j} \psi\to \Div_\delta^s \psi
\end{align*}
uniformly on $\R^n$ as $j \to \infty$.
\end{lemma}
\begin{proof} 
It suffices to focus on proving the convergence of the nonlocal gradient; the argument for the divergence is an immediate consequence. If $s<1$, we conclude from Proposition~\ref{Prop: convolution with the classical gradient}, Young's convolution inequality, and Lemma~\ref{le:qnorm}\,$(i)$ that
\[
\norm{D^{s_j}_\d \phi - D^s_\d \phi}_{L^{\infty}(\R^n;\R^n)} \leq \norm{Q^{s_j}_\d-Q^{s}_\d}_{L^1(\R^n)}\norm{\nabla \phi}_{L^{\infty}(\R^n;\R^n)} \to 0 \ \ \text{as $j \to \infty$}.
\]
The case $s=1$ follows immediately from \eqref{eq:mollifier}, since $\nabla \phi \in \Lipb(\R^n)$ allows us to conclude that
\[
\lim_{j \to \infty}\norm{D^{s_j}_\d \phi - \nabla \phi}_{L^{\infty}(\R^n;\R^n)} = \lim_{j \to \infty} \norm{Q^{s_j}_\d * \nabla \phi - \nabla \phi}_{L^{\infty}(\R^n;\R^n)} =0.
\]
\end{proof}

Our approach to extending the previous results for smooth functions in a suitable way to nonlocal Sobolev spaces, relies on the following estimate (see Corollary~\ref{cor:orderinequality} below),
\begin{equation}\label{estst}
\norm{D^s_\d u}_{L^p(\R^n;\R^n)} \leq C \norm{D^t_\d u}_{L^p(\R^n;\R^n)} \quad \text{for all $u \in C_c^{\infty}(\R^n)$ and $0\leq s\leq t \leq 1$},
\end{equation}
with a constant $C>0$ depending only on $n, p, \delta$.  
If $t=1$,~\eqref{estst} simply follows from Young's convolution inequality
\begin{equation}\label{eq:classicalbound}
\norm{D^s_\d u}_{L^p(\R^n;\R^n)}\leq \norm{Q^s_\d}_{L^1(\R^n)} \norm{\nabla u}_{L^p(\R^n;\R^n)} \leq C\norm{\nabla u}_{L^p(\R^n;\R^n)},
\end{equation}
where we have exploited that $\norm{Q^s_\d}_{L^1(\R^n)}$ is bounded by $C$ uniformly in $s$ as a consequence of Lemma~\ref{le:qnorm}. If $t=s$, one can obviously take the constant to be $1$.
For the other cases, we build on Fourier multiplier theory (see e.g.,~\cite[Chapter~5]{Gra14a}) and show via the Mihlin-H\"ormander theorem that the maps 
\begin{align}\label{multiplier_mst}
m^{s}_t:\Rn 
 \to \R, \quad \xi\mapsto\frac{\widehat{Q}^s_\d(\xi)}{\widehat{Q}^t_\d(\xi)}
\end{align}
are $L^p$-multipliers with uniformly bounded norms. 
This requires control on the decay behavior of $m^{s}_t$ and its derivatives. The idea for deriving suitable bounds for large frequencies is to compare $Q^{s}_\d$ with the well-known Riesz potential kernel $I_{1-s}$ (cf.~\eqref{Rieszkernel}) and exploit the decay of the difference of their Fourier transforms uniformly in $s$ (see Lemma~\ref{le:decayR}). 

\begin{lemma}\label{lem:multiplier}
The map $m^{s}_t:\R^n\to \R$ from~\eqref{multiplier_mst} with $0\leq s \leq t<1$ is an $L^p$-multiplier for every $p \in (1,\infty)$ with multiplier norm independent of the parameters $s, t$. 
\end{lemma}

\begin{proof}
According to the Mihlin-H\"{o}rmander multiplier theorem, see e.g.,~\cite[Theorem~6.2.7]{Gra14a}, the statement follows immediately once these estimates have been established: There exists a constant $C>0$ depending only on $n$ and $\d$ such that for every $\alpha \in \N^n_0$ with $\abs{\alpha} \leq n/2 +1$ and every $0\leq s \leq t <1$,
\begin{equation}\label{eq:mihlinbounds}
\abs{\xi}^{\abs{\alpha}}\abslr{\partial^{\alpha}m^{s}_t(\xi)} \leq C \qquad \text{for all $\xi \in  \R^n$}. 
\end{equation}
The proof is split in two parts, where we distinguish bounds for large and small frequencies. Note that in the following all the constants $C,c>0$ are independent of $s,t$.\smallskip

\textit{Step 1: Bounds away from zero}. In this step, we show that there is some $R\geq 1$ such that \eqref{eq:mihlinbounds} holds for all $\abs{\xi} \geq R$.  Since 
\[
\widehat{Q}^s_\d(\xi) = \abs{2\pi\xi}^{-(1-s)} + \widehat{R}^s_\d(\xi) \qquad \text{for $\abs{\xi} \geq 1$}
\]
for any $s\in [0,1)$ by~\eqref{Rhatsd}, 
we can express $m^{s}_t$ on $B(0,1)^c$ as
\begin{align}\label{eq:simplifyfrac}
m^{s}_t(\xi)=\frac{\widehat{Q}^s_\d(\xi)}{\widehat{Q}^t_\d(\xi)} =\abs{2\pi\xi}^{-(t-s)} + r^{s}_t(\xi)
\end{align}
with $$r^{s}_t(\xi) := \frac{-|2\pi\xi|^{-(t-s)}\widehat{R}_\delta^t(\xi)+\widehat{R}^s_\d(\xi)}{\abs{2\pi\xi}^{-(1-t)}+\widehat{R}^t_\d(\xi)}.$$
Given $t\geq s$, it is clear that
\begin{equation}\label{eq:firstterm}
\partial^{\alpha} \bigl(\abs{2\pi\xi}^{-(t-s)}\bigr) \leq C\abs{\xi}^{-\abs{\alpha}} \quad \text{for $\abs{\xi} \geq 1$}.
\end{equation}
Along with~\eqref{eq:decayRhat}, one can estimate the denominator of $r^{s}_t$ and find some $R\geq 1$ such that for all $\xi\in \R^n$ with $|\xi|\geq R$, 
\begin{align}\label{est_denominator}
 \abs{2\pi\xi}^{-(1-t)} +\widehat{R}^{t}_\d(\xi) \geq |2\pi\xi|^{-1} - c|\xi|^{-2} \geq C\abs{\xi}^{-1}. 
\end{align}
If one takes the $\alpha$th derivative of $r^{s}_t$ on $B(0, R)^c$, the quotient rule gives rise to a quotient whose denominator results from raising the denominator of $r^{s}_t$ to the power $2^{|\alpha|}$ and whose numerator is a product of $\widehat{R}^s_\d$, $\widehat{R}^t_\d$ and their derivatives with terms bounded independently of $s, t$. We therefore obtain in view of~\eqref{est_denominator}, and again~\eqref{eq:decayRhat}, that
\begin{equation}\label{eq:secondterm}
|\partial^{\alpha}r^{s}_t(\xi)| \leq C |\xi|^{-2^{|\alpha|}}\leq  C\abs{\xi}^{-\abs{\alpha}} \qquad \text{for $\abs{\xi} \geq R$}.
\end{equation}
The combination of \eqref{eq:simplifyfrac}, \eqref{eq:firstterm} and \eqref{eq:secondterm} then yields \eqref{eq:mihlinbounds} on $B(0, R)^c$. \smallskip

\textit{Step 2: Local bounds.} To show that \eqref{eq:mihlinbounds} holds for $\abs{\xi} \leq R$, we observe first that, as a consequence of Lemma~\ref{le:qnorm}\,$(ii)$ and the non-negativity of $\widehat{Q}_\delta^s$ (cf.~Remark~\ref{rem:Qdeltas}), there is a constant $c>0$ such that
\begin{align*}
\widehat{Q}^{s}_\d(\xi) \geq c
\end{align*}
for all $\xi\in \overline{B(0, R)}$ and all $s\in [0,1)$.
Moreover, for any $\beta \in \N_0^n$ with $\abs{\beta} \leq n/2+1$, the estimate
\[
\norm{(-2\pi i\, \cdot)^{\beta}Q^s_\d}_{L^1(\R^n)} \leq C\norm{Q^s_\d}_{L^1(\R^n)}\delta^{\abs{\beta}}\leq C\delta^{\abs{\beta}},
\]
where the last inequality follows in view of~Lemma~\ref{le:qnorm}\,$(i)$, implies
\[
\bigl|\partial^{\beta}\widehat{Q}^s_\d(\xi)\bigr|\leq C \delta^{\abs{\beta}} \leq C \qquad \text{for all $\xi \in \R^n$ and $s \in [0,1)$.}
\]
To conclude, we use again the quotient rule to obtain
\begin{align*}
\abs{\xi}^{\abs{\alpha}}\abslr{\partial^{\alpha}m^{s}_t(\xi)}=
\abs{\xi}^{\abs{\alpha}}\abslr{\partial^{\alpha}\left(\frac{\widehat{Q}^s_\d(\xi)}{\widehat{Q}^t_\d(\xi)}\right)} \leq \frac{R^{\abs{\alpha}}C}{c^{2^{\abs{\alpha}}}} = C \qquad \text{for $\abs{\xi} \leq R$}.
\end{align*}
\end{proof}
We now obtain the next corollary based on the previous lemma;
recall the definitions of $H^{1,p,\d}(\R^n)$ and $H^{1,p,\d}_0(\Omega)$ in~\eqref{eq:classicalspace1} and~\eqref{eq:classicalspace2}.
\begin{corollary}\label{cor:orderinequality}
Let $0 \leq s \leq t \leq 1$ and $p \in (1,\infty)$. If $u \in H^{t,p,\d}(\R^n)$, then $u \in \Hspd(\R^n)$ and there is a constant $C>0$ depending only on $n$, $\d$ and $p$ such that
\begin{equation}\label{eq:orderinequality}
\norm{D^s_\d u}_{L^p(\R^n;\Rn)} \leq C\norm{D^t_\d u}_{L^p(\R^n;\Rn)}.
\end{equation}
If $\Omega \subset\R^n$ is open and bounded and $u \in H^{t,p,\d}_0(\O)$, then $u \in H^{s,p,\d}_0(\O)$ with
\[
\norm{D^s_\d u}_{L^p(\O;\Rn)} \leq C\norm{D^t_\d u}_{L^p(\O;\Rn)}.
\]
\end{corollary}
\begin{proof}
The case $t=1$ is covered by~\eqref{eq:classicalbound}. For the other cases, we deduce from the previous lemma that the map 
\[
M^s_t: \Scal(\R^n;\R^n) \to L^p(\R^n;\R^n), \qquad v \mapsto  (m^{s}_t\widehat{v})^\vee =\left(\frac{\widehat{Q}^s_\d}{\widehat{Q}^t_\d}\widehat{v}\right)^{\vee}
\]
can be extended to a bounded linear operator on $L^p(\R^n;\R^n)$ with
\begin{equation}\label{eq:mihlin}
\norm{M^s_t v}_{L^p(\R^n;\R^n)} \leq C\norm{v}_{L^p(\R^n;\R^n)}
\end{equation}
for all $v\in L^p(\R^n;\R^n)$, where $C>0$ is a constant independent of $s, t$. For $\phi \in C_c^{\infty}(\R^n)$, we also observe using Proposition~\ref{Prop: convolution with the classical gradient} that
\begin{align*}
M^s_t D^t_\d \varphi = M^s_t(Q^t_\d * \nabla \varphi)=\left(\frac{\widehat{Q}^s_\d}{\widehat{Q}^t_\d}\widehat{Q}^t_\d \widehat{\nabla \varphi}\right)^{\vee} = \Bigl(\widehat{Q}^s_\d \widehat{\nabla \varphi}\Bigr)^{\vee}=D^s_\d \varphi.
\end{align*}

With $u \in H^{t,p,\d}(\R^n)$, one can take an approximating sequence $\{\varphi_j\}_{j\in \N} \subset C_c^{\infty}(\R^n)$ with $\varphi_j \to u$ in $H^{t,p,\d}(\R^n)$ and infer from the continuity of the operator $M^s_t$ that $D^s_\d \varphi_j = M^s_t D^t_\d \varphi_j \to M^s_t D^t_\d u$ in $L^p(\R^n)$. This shows that $u \in H^{s,p,\d}(\R^n)$ with $D^s_\d u = M^s_t D^t_\d u\in L^p(\R^n)$. The bound \eqref{eq:orderinequality} follows now from \eqref{eq:mihlin}. 

Finally, the statement for $u \in H^{t,p,\d}_0(\O)$ follows by extending $u$ to $\R^n$ by zero, noting that then the nonlocal gradient of $u$ is zero in $\Omega^c$.
\end{proof}
\begin{remark}
a) We note that this approach does not extend to $p=1$, since the Mihlin-H\"{o}rmander theorem is not valid in this case. Moreover, this approach does not apply to $u \in H^{t,p,\d}(\Omega)$ because it requires functions to be defined on all of $\R^n$ for the Fourier transform techniques. In fact, there is no obvious way of how to extend functions in $H^{t,p,\d}(\Omega)$, as they can be ill-behaved in the strip $\Omega_\d \setminus \O$. \smallskip

b) An inequality of the type \eqref{eq:orderinequality} does not hold for the fractional gradient, which can be seen from the homogeneity property. Indeed, for $u \in C_c^{\infty}(\R^n)$ and $0\leq s<t\leq 1$, we may define for $\lambda >0$ the function $u_{\lambda}:=\lambda^{n/p-t}u(\lambda\,\cdot)$. Then, we can calculate that for $x \in \R^n$
\[
D^t u_{\lambda}(x) = \lambda^{n/p} D^t u(\lambda x) \quad \text{and} \quad D^s u_{\lambda}(x) = \lambda^{n/p-(t-s)}D^s u(\lambda x).
\]
This gives $\norm{D^t u_{\lambda}}_{L^p(\R^n;\R^n)} = \norm{D^tu}_{L^p(\R^n;\R^n)}$, whereas $\norm{D^s u_{\lambda}}_{L^p(\R^n;\R^n)} = \lambda^{-(t-s)}\norm{D^su}_{L^p(\R^n;\R^n)}$. Letting $\lambda \to 0$ shows that \eqref{eq:orderinequality} cannot hold for the fractional gradient.
\end{remark}
As a consequence, we derive the following generalization of the convergence result Lemma~\ref{le:orderconvergencesmooth} to the nonlocal Sobolev setting.
\begin{theorem} \label{th: continuity on s of Dsdelta}
Let $p \in (1,\infty)$ and let $\{s_j\}_{j \in \N} \subset [0,1]$ be  a sequence converging to $s \in [0,1]$ with $\bar{s}:=\sup_{j \in \N} s_j$. Then, it holds for every $u \in H^{\bar{s},p,\d}(\R^n)$ that
\[
D^{s_j}_\d u \to D^s_\d u \ \ \text{in $L^p(\R^n;\Rn)$ as $j \to \infty$}.
\]
If $\Omega \subset\R^n$ is open and bounded and $u \in H^{\bar{s},p,\d}_0(\O)$, then
\[
D^{s_j}_\d u \to D^s_\d u \ \ \text{in $L^p(\O;\R^n)$ as $j \to \infty$}.
\]
\end{theorem}
\begin{proof}
Take $\epsilon >0$ and $\phi_{\epsilon} \in C_c^{\infty}(\R^n)$ such that $\norm{u-\phi_{\epsilon}}_{H^{\bar{s},p,\d}(\R^n)} \leq \epsilon$, cf.~Theorem~\ref{th:density}. Then, due to Corollary~\ref{cor:orderinequality}, 
\[
\norm{D^{s_j}_\d(u-\phi_{\epsilon})}_{L^p(\R^n;\R^n)} \leq C\epsilon \quad \text{for all $j \in \N$} \quad \text{and} \quad \norm{D^{s}_\d(u-\phi_{\epsilon})}_{L^p(\R^n;\R^n)} \leq C\epsilon.
\]
If we choose $j$ large enough so that $\norm{D^{s}_\d \phi_{\epsilon} - D^{s_j}_\d \phi_{\epsilon}}_{L^p(\R^n;\R^n)} \leq \epsilon$, which is possible by Lemma~\ref{le:orderconvergencesmooth}, we obtain
\begin{align*}
\norm{D^{s}_\d u - D^{s_j}_\d u}_{L^p(\R^n;\R^n)}&\leq \norm{D^{s}_\d(u-\phi_{\epsilon})}_{L^p(\R^n;\R^n)}+\norm{D^{s}_\d \phi_{\epsilon} - D^{s_j}_\d \phi_{\epsilon}}_{L^p(\R^n;\R^n)}\\
&\qquad+\norm{D^{s_j}_\d(u-\phi_{\epsilon})}_{L^p(\R^n;\R^n)}\\
&\leq (2C +1)\epsilon,
\end{align*}
and letting $\epsilon \to 0$ yields the desired convergence. The case $u \in H^{\bar{s},p,\d}_0(\Omega)$ follows again via extension.
\end{proof}
\begin{remark}
		For the particular case of localization to the classical gradient, i.e., when $s_j \to 1$ as $j \to \infty$,  the convergence $D^{s_j}_\d u \to \nabla u$ in $L^p(\O;\Rn)$  with $u \in W^{1,p}(\O_{ \d})$ holds without imposing complementary values. Indeed, by Proposition \ref{Prop: convolution with the classical gradient} and Lemma \ref{le:qnorm}, we can bound $\|D^s_\d u\|_{L^p(\O;\R^n)}\leq C \|\nabla u\|_{L^p(\O_{ \d};\R^n)}$ uniformly in $s$, and then, a similar argument to that of the proof of Theorem \ref{th: continuity on s of Dsdelta} applies.
\end{remark}

As another consequence of \eqref{eq:orderinequality}, we establish a nonlocal Poincar\'{e} inequality with a constant independent of the fractional order $s$. The proof builds on two pillars, namely the  estimate of Corollary~\ref{cor:orderinequality}, which says that it is enough to prove the inequality for $s=0$, and in order to achieve the latter, a version of the fundamental theorem of calculus for the case $s=0$ from Proposition~\ref{prop:ftoczero}.

\begin{theorem}[Nonlocal Poincar\'{e} inequality with uniform constants in $s$]\label{th:poincareindep}
Let $s \in [0,1]$, $p \in (1,\infty)$ and $\Omega \subset \R^n$ be open and bounded. Then, there exists a constant $C>0$ depending only on $\Omega$, $\d$ and $p$ such that for all $u\in H^{s,p,\d}_0(\O)$,
\begin{align}\label{Poincare_uniform}
\norm{u}_{L^p(\O)} \leq {C}\norm{D^s_\d u}_{L^p(\O;\R^n)}.
\end{align}
\end{theorem}

\begin{proof}
Given Corollary~\ref{cor:orderinequality}, it suffices to prove~\eqref{Poincare_uniform} for $s=0$. Moreover, we may assume by density (cf.~\eqref{Hspd0}) that $u \in C_c^{\infty}(\O_{-\d})$.   Proposition~\ref{prop:ftoczero} together with the fact that $\supp(D^0_\d u) \subset \O$ then implies 
\[
\norm{u}_{L^p(\O)} \leq \norm{RD^0_\d u}_{L^p(\R^n)} + \abs{\O} \norm{W_\delta}_{L^{\infty}(\R^n;\R^n)}\norm{D^0_\d  u}_{L^p(\O;\R^n)} \leq C\norm{D^0_\d u}_{L^p(\O;\R^n)},
\]
where the second inequality uses the $L^p$-boundedness of the Riesz transform. 
\end{proof}

Finally, we present a compactness statement for sequences that are bounded in nonlocal spaces of different order. It will be used later in the proof of the $\Gamma$-convergence result in Section~\ref{sec:Gamma}.

\begin{lemma}[Weak compactness of sequences in varying order nonlocal spaces]\label{le:ordercompactness}
Let $p \in (1,\infty)$ and $\O \subset \R^n$ be open and bounded with $\O_{-\d}$ a Lipschitz domain. Consider any sequence $\{s_j\}_{j \in \N} \subset [0,1]$ converging to $s \in [0,1]$ and $u_j \in H^{s_j,p,\d}_0(\O)$ for $j \in \N$ with
\[
\sup_{j \in \N}\, \norm{D^{s_j}_\d u_j}_{L^p(\O;\R^n)} < \infty.
\]
Then, up to a non-relabeled subsequence, $u_j \weakto u$ in $L^p(\O_\d)$ with $u \in H^{s,p,\d}_0(\O)$ and as $j \to \infty$,
\[
D^{s_j}_\d u_j \weakto D^s_\d u \ \ \text{in $L^p(\O;\R^n)$} \quad \text{and} \quad D^{s_j}_\d u_j(x) \to D^{s}_\d u(x) \ \ \text{for a.e.~$x \in \O \setminus \O_{-\d}$.}
\] 
\end{lemma}

\begin{proof}
In view of the Poincar\'{e} inequality of Theorem~\ref{th:poincareindep}, we observe that
\[
\sup_{j \in \N} \,\norm{u_j}_{L^p(\O)} \leq C \sup_{j \in \N}\, \norm{D^{s_j}_\d u_j}_{L^p(\O;\R^n)} < \infty.
\]
Therefore, we can extract a subsequence of $\{u_j\}_{j\in \N}$ (non-relabeled) and find $u \in L^p(\O_\d)$ and $V \in L^p(\O;\R^n)$ such that
\[
u_j \weakto u \ \ \text{in $L^p(\O_\d)$} \quad \text{and} \quad D^{s_j}_\d u_j \weakto V \ \ \text{in $L^p(\O;\R^n)$} 
\]
as $j \to \infty$. Note that $u =0$ in $\O_\d \setminus \O_{-\d}$, since the same holds for the functions $u_j$. To show that $u \in H^{s,p,\d}_0(\O)$ and $V=D^s_\d u$,  take $\phi \in C_c^{\infty}(\O;\R^n)$ and observe that
\begin{align*}
\int_{\O}V \cdot \phi\,dx &= \lim_{j \to \infty} \int_{\O} D^{s_j}_\d u_j \cdot \phi\,dx =-\lim_{j\to \infty} \int_{\O_\d}u_j \Div^{s_j}_\d \phi\,dx =-\int_{\O_\d} u \Div^s_\d \phi\,dx,
\end{align*}
where the last equality results from the weak convergence $u_j \weakto u$ in $L^p(\O_\d)$ and the uniform convergence $\Div^{s_j}_\d \phi \to \Div^s_\d \phi$ by Lemma~\ref{le:orderconvergencesmooth}. Hence, $u \in H^{s,p,\d}(\O)$ with $D^s_\d u = V$, and Proposition~\ref{prop:densitycomplement} implies  $u \in H^{s,p,\d}_0(\O)$, since  $u=0$ a.e.~in $\O_{\d} \setminus \O_{-\d}$ and $\O_{-\d}$ is Lipschitz.

It remains to prove the pointwise convergence of the nonlocal gradients outside of $\Omega_{-\delta}$.  To this end, we observe in view of Remark~\ref{rem:nonlocalgradprop} that for any $t \in [0,1]$ and $v \in H^{t,p,\d}_0(\O)$, 
\begin{equation}\label{eq:gradstrip}
D^t_\d v (x) = \begin{cases}
(d^t_\d * v)(x) &\text{if $t \in [0,1)$,}\\
0 &\text{if $t=1$},
\end{cases}
\end{equation}
for a.e.~$x \in \O \setminus \O_{-\d}$;  note that $\abs{\partial \Omega_{-\d}}=0$, so that this set may be ignored. If $s \not =1$, it holds for any $\eps >0$ that $d_\delta^{s_j} \to d_\delta^{s}$ uniformly on $B_{\epsilon}(0)^c$ as $j \to \infty$. Consequently,
\[
\lim_{j \to \infty} D^{s_j}_\d u_j(x)=\lim_{j \to \infty}\int_{\O_{-\d}} u_j(y) d_\delta^{s_j}(x- y) \,dy=\int_{\O_{-\d}} u(y) d_\delta^s(x- y)\,dy = D^s_\d u(x)
\]
for a.e.~$x \in \O \setminus \O_{-\d}$. In the case $s=1$, we have $d_\delta^{s_j} \to 0$ uniformly on $B_{\epsilon}(0)^c$ as $j\to \infty$ due to the convergence $c_{n,s_j} \to 0$. The same argument then yields the desired pointwise convergence in light of \eqref{eq:gradstrip}.
\end{proof}

\section{Weak lower semicontinuity and existence theory}\label{section:wls}
This section is devoted to characterizing the weak lower semicontinuity of integral functionals depending on the nonlocal gradient, that is, functionals of the form
\begin{equation}\label{eq:functional}
\Fcal(u) = \int_{\Omega} f(x,u(x),D^s_\d u(x))\,dx \quad \text{for $u \in H^{s,p,\d}_g(\O; \R^m)$,}
\end{equation}
where $s \in (0,1)$, $p \in (1,\infty)$, $\O \subset \R^n$ is open and bounded, $g\in H^{s, p, \delta}(\Omega;\R^m)$, and $f:\O \times \R^m \times \Rmn\to \R$ is a suitable integrand with $p$-growth. Using the connection between the nonlocal gradient and the classical gradient from Theorem~\ref{prop:connection}, we can employ a translation procedure along the lines of~\cite{KrS22} to conclude that the weak lower semicontinuity of $\F$ is equivalent to the quasiconvexity of $f$ in its third argument. In fact, the quasiconvexity is only required in $\O_{-\d}$, which is due to the strong convergence of the nonlocal gradient in $\O_\d \setminus \O_{-\d}$ from Lemma~\ref{le:strongoutside}.

\begin{theorem}[Characterization of weak lower semicontinuity]\label{theo:characterization}
Let $s \in (0,1)$, $p \in (1,\infty)$, $\Omega \subset \R^n$ be open and bounded with $|\partial \O_{-\d}|=0$ and $g \in H^{s,p,\d}(\O;\R^m)$. Further, let $f:\O \times \R^m \times \Rmn \to \R$ be a Carath\'{e}odory function satisfying 
\[
-C(1+\abs{z}^p+\abs{A}^q) \leq f(x,z,A) \leq C(1+\abs{z}^p+\abs{A}^p)
\]
for a.e.~$x \in \O$ and all $(z,A) \in \R^m\times \Rmn$ with $C>0$ and $q \in [1,p)$.  

Then, $\Fcal$ from \eqref{eq:functional} is weakly lower semicontinuous on $H^{s,p,\d}_g(\O;\R^m)$ if and only if 
\begin{align}\label{qc}
f(x,z,\cdot)\quad \text{ is quasiconvex for a.e.~$x \in \O_{-\d}$ and all $z \in \R^m$,}
\end{align}
i.e., it holds for a.e.~$x\in \Omega_{-\delta}$ and all $z\in \R^m$ with $Y=(0,1)^n$ that
\begin{align*}
f(x, z, A)\leq\int_{Y}  f(x, z, A +\nabla \varphi(y)) \, dy  \quad \text{for all $\varphi\in W_0^{1, \infty}(Y;\R^m)$ and $A\in \R^{m\times n}$.}
\end{align*}
\end{theorem}
\begin{proof}
 The proof follows the lines of \cite[Theorem~4.1 and~4.5]{KrS22}, we detail the differences for the reader's convenience. \smallskip

\textit{Step 1: Sufficiency.} Assuming~\eqref{qc}, let $\{u_j\}_{j \in \N} \subset \Hspdgm$ be a sequence that converges weakly to $u$ in $\Hspdgm$. We divide the proof by splitting the integral functional $\Fcal$  and considering separately the integral contributions over $\Omega_{-\delta}$ and $\Omega\setminus \Omega_{-\delta}$. 

Since $u_j \to u$ in $L^p(\O_{\d};\R^m)$ by \cite[Theorem~7.3]{BeCuMC22} and $\mathcal{Q}^s_\d u_j \weakto \mathcal{Q}^s_\d u$ in $W^{1,p}(\O;\R^m)$ by Theorem~\ref{prop:connection}\,$(i)$, we conclude
\begin{equation}\label{eq:insideineq}
\begin{split}
\int_{\O_{-\d}}f(x,u,D^s_\d u)\,dx &= \int_{\O_{-\d}} f(x,u,\nabla \mathcal{Q}^s_\d u)\,dx\\
&\leq \liminf_{j \to \infty} \int_{\O_{-\d}} f(x,u_j,\nabla \mathcal{Q}^s_\d u_j)\,dx\\
& = \liminf_{j \to \infty} \int_{\O_{-\d}} f(x,u_j,D^s_\d u_j)\,dx,
\end{split}
\end{equation}
where the inequality is due to the quasiconvexity and $p$-growth of $f$, with the exact argument of \cite[Theorem~4.1]{KrS22} involving Young measures. Note that this requires the negative part of the sequence $\{f(\cdot ,u_j,\nabla \Qcal_\delta^s u_j)\}_{j \in \N}$ to be equi-integrable, which is guaranteed by the lower bound on $f$.

Secondly, for the integral on $\O \setminus \O_{-\d}$, we invoke from Lemma~\ref{le:strongoutside} the convergence $$D^s_\d u_j \to D^s_\d u \in L^p(\O \setminus O;\Rmn)$$ for any  $O\Supset\O_{-\d}$ . Hence, a well-known strong lower semicontinuity result (e.g.,~\cite[Theorem~6.49]{FoL07}) yields
\[
\int_{\O \setminus O} f(x,u,D^s_\d u)\,dx \leq \liminf_{j \to \infty} \int_{\O \setminus O} f(x,u_j,D^s_\d u_j)\, dx.
\]
Letting $O \downarrow \O_{-\d}$ implies, using once again the equi-integrability of the negative part $\{f(\cdot ,u_j,D^s_\d u_j)\}_{j \in \N}$ and the assumption $|\partial \O_{-\d}| =0$, that
\begin{equation}\label{eq:outsideineq}
\int_{\O \setminus \O_{-\d}} f(x,u,D^s_\d u)\,dx \leq \liminf_{j \to \infty} \int_{\O \setminus \O_{-\d}} f(x,u_j,D^s_\d u_j)\, dx.
\end{equation}
The sufficiency now follows from adding \eqref{eq:insideineq} and \eqref{eq:outsideineq}. \smallskip

\textit{Step 2: Necessity}. Analogously to the proof of \cite[Theorem~4.5]{KrS22}, we may assume without loss of generality that $g = 0$. In order to prove the stated quasiconvexity of $f$, let us fix $(x_0,z_0,A_0) \in \O_{-\d} \times \R^m \times \Rmn$. Using Lemma~\ref{le:construction}, we may select a $\varphi_0 \in C_c^{\infty}(\O_{-\d};\R^m)$ such that
\begin{equation}\label{eq:point}
\varphi_0(x_0) = z_0 \qquad \text{and} \qquad D^s_\d \varphi_0 (x_0) = A_0.
\end{equation}

Consider any $\phi \in W^{1,\infty}_0(Y;\R^m)$ and assume that $x_0 + Y \Subset \O_{-\d}$; the latter can be done without loss of generality in light of the scaling and translation invariances related to the definition of quasiconvexity, see e.g.,~\cite[Proposition~5.11]{Dac08}.  If we fix $\rho \in (0,1)$ and periodically extend $\phi$ to $\R^n$, we can define the sequence $\{\phi^{\rho}_j\}_{j \in \N} \subset W^{1,\infty}(\R^n;\R^m)$ by
\[
\phi^{\rho}_j(x)=\begin{cases}
\displaystyle \frac{\rho}{j}\phi\Bigl(j\frac{(x-x_0)}{\rho}\Bigr) &\ \text{for $x \in Y_{\rho}:= x_0+ 
(0,\rho)^n$,}\\
0 &\ \text{otherwise, }
\end{cases} \qquad x\in \R^n.
\]
As this is a periodically oscillating sequence that converges to zero essentially uniformly, we find that 
$\phi^{\rho}_j \weakto 0$ in $W^{1,p}(\R^n;\R^m)$ as $j \to \infty$. 

Take a cut-off function $\chi \in C_c^{\infty}(\O_{-\d};[0,1])$ with $\chi \equiv 1$ on $x_0+Y$ and define the sequence $\{u_j\}_{j \in \N} \subset \Hspd_0(\Omega;\R^m)$ given by
\[
u_j := \varphi_0 + \chi \Pcal^s_\d \phi^{\rho}_j,
\]
which converges weakly to $\varphi_0$ in $\Hspd_0(\O;\R^m)$ in light of the continuity of $\Pcal^s_\d$ in Theorem~\ref{prop:connection}\,$(ii)$. In particular, we have $u_j \to \phi_0$ in $L^p(\O_\d;\R^m)$ by \cite[Theorem~7.3]{BeCuMC22}. Moreover, it holds by the Leibniz rule in Lemma~\ref{le:leibniz} and the fact that $D^s_\d \Pcal^s_\d \phi^{\rho}_j = \nabla \phi^{\rho}_j$ that
\[
D^s_\d u_j = D^s_\d \phi_0 + \chi \nabla \phi^{\rho}_j + K_{\chi}(\Pcal^s_\d\phi^{\rho}_j).
\]
Observe that $K_{\chi}(\Pcal^s_\d\phi^{\rho}_j) \to 0$ in $L^p(\O;\Rmn)$ as $j \to \infty$ due to the boundedness of $K_{\chi}$ and that $\chi \nabla \phi^{\rho}_j =\nabla \phi^{\rho}_j$ on $\O$ since $\phi^{\rho}_j$ is zero outside $Y_\rho$. 

Finally, we exploit the weak lower semicontinuity of $\Fcal$ on $\Hspd_0(\Omega;\R^m)$ to derive
\begin{align*}
 \int_{\O}f(x,\varphi_0,D^s_\d \varphi_0)\dx &\leq \liminf_{j \to \infty}\int_{\O} f(x,u_j,D^s_\d u_j)\dx\\
&= \liminf_{j \to \infty} \int_{Y_{\rho}} f(x,u_j,D^s_\d \phi_0 + \nabla \phi^{\rho}_j + K_{\chi}(\Pcal^s_\d\phi^{\rho}_j))\dx\\
& \qquad\quad +\int_{\O \setminus Y_{\rho}}f(x,u_j,D^s_\d \phi_0 + K_{\chi}(\Pcal^s_\d\phi^{\rho}_j))\dx\\
&\leq \liminf_{j \to \infty} \int_{Y_{\rho}} f(x,\varphi_0,D^s_\d \varphi_0+\nabla \phirho_j)\dx+\int_{\O \setminus Y_{\rho}}f(x,\varphi_0,D^s_\d \varphi_0)\dx,
\end{align*}
where the last inequality uses~\cite[Lemma~4.10]{KrS22} to remove all the terms that converge strongly to zero. In view of the $p$-growth of $f$, the integral over $\O\setminus Y_{\rho}$ is finite, so that subtracting it from both sides gives
\[
\int_{Y_{\rho}}f(x,\varphi_0,D^s_\d \varphi_0)\dx\leq \liminf_{j \to \infty} \int_{Y_{\rho}} f(x,\varphi_0,D^s_\d \varphi_0+\nabla \phirho_j)\dx.
\]
Because $\varphi_0$ and $D^s_\d \varphi_0$ are continuous and satisfy \eqref{eq:point}, the rest of the proof follows by mimicking Steps 2-4 of \cite[Theorem~4.5]{KrS22}.
\end{proof}

The following lemma was used in the previous proof and shows that one can construct smooth functions with compact support whose nonlocal gradient has a desired value at a point. The proof is omitted here, as it is nearly identical to \cite[Lemma~4.3]{KrS22}, given that $w_\d$ is radial by \ref{itm:h1}.
\begin{lemma}
\label{le:construction}
Let $s \in (0,1)$ and let $\Omega \subset \R^n$ be open and bounded. For any $x_0\in \Omega_{-\d}$, $z \in \R^m$ and $A \in \Rmn$, there exists a $\phi \in C_{c}^{\infty}(\Omega_{-\d};\R^m)$ such that $\phi(x_0)=z$ and $D^s_{\d} \phi(x_0)=A$.
\end{lemma}

With the perspective of Section~\ref{subsec:connection_localfrac}, there is an alternative approach to proving Theorem~\ref{theo:characterization} that passes through the characterization of weak lower semicontinuity of functionals depending on Riesz fractional gradients from~\cite{KrS22}. 
For simplicity, we take $g=0$ and drop the dependence on $x$ and $z$ in $f$.

\begin{proof}[Alternative proof of Theorem~\ref{theo:characterization}] \textit{Step~1: Sufficiency.} Let $\{u_j\}_{j\in\N}\subset H_0^{s, p, \delta}(\Omega;\R^m)$ converge weakly in $H_0^{s, p, \delta}(\Omega;\R^m)$ to the limit function $u\in H_0^{s, p, \delta}(\Omega;\R^m)$. As a quasiconvex function, $f:\R^{m\times n}\to \R$ is also rank-one convex  and hence, locally Lipschitz continuous in the sense that 
\begin{align*}
|f(A)-f(B)|\leq C(1+|A|^{p-1}+|B|^{p-1})|A-B|\qquad \text{ for all $A, B\in \R^{m\times n}$}
\end{align*}
with a constant $C>0$, cf.~e.g.,~\cite[Proposition~2.32]{Dac08}.

Consider the auxiliary function 
\begin{align*}
h_u(x, A)= \mathbbm{1}_{\Omega}(x)f\bigl(A+(\nabla R_\delta^s \ast u)(x)\bigr) \quad \text{ for $x\in \R^n$ and $A\in \R^{m\times n}$, }
\end{align*}
which is Carath\'{e}odory, quasiconvex in the second variable, and satisfies the growth bound
$$|h_u(x, A)|\leq C\bigl(1+ |A|^p + |(\nabla R_\delta^s\ast u)(x)|^p\bigr) \leq C\bigl(1 + |A|^p + \norm{u}^p_{L^p(\Omega_\delta;\R^m)}\bigr),$$ 
with the last step using the boundedness of $\nabla R_\delta^s$. By the local Lipschitz continuity of $f$, we also find 
\begin{align*}
&\Bigl| \int_\Omega f(D^s_\delta u_j) \, dx - \int_{\R^n} h_u(x, D^su_j)\, dx\Bigr| \\  &\qquad \qquad 
\leq  C\bigl(1+\norm{D_\delta^su_j}_{L^p(\Omega;\R^{m\times n})} +  \norm{u_j}_{L^p(\Omega_\delta;\R^m)} + \norm{u}_{L^p(\Omega_\delta;\R^m)}\bigr)
 \norm{u_j-u}_{L^p(\Omega_\delta;\R^m)} \to 0 
\end{align*}
as $j\to \infty$. Since $u_j\weakly u$ in $H_0^{s, p}(\Omega_{-\d};\R^m)$ by Lemma~\ref{le:nonlocalfrac} and $h_u$ fulfills the requirements of~\cite[Theorem~4.1]{KrS22}, the desired lower semicontinuity results from
\begin{align*}
\liminf_{j\to \infty} \Fcal(u_j) =\liminf_{j\to \infty} \int_\Omega f(D_\delta^su_j)\, dx  &= \liminf_{j\to \infty} \int_{\R^n} h_u(x, D^su_j)\, dx \\ & \geq \int_{\R^n} h_u(x, D^s u)\, dx =  \int_\Omega f(D_\delta^s u)\, dx = \Fcal(u).
\end{align*}
\smallskip

\textit{Step~2: Necessity.} Suppose $\Fcal$ is weakly lower semicontinuous on $H_0^{s, p, \delta}(\Omega;\R^m)$. Fix $(x_0, A_0)\in \Omega_{-\d}\times\R^{m\times n}$ and using Lemma~\ref{le:construction}, let 
$\varphi\in  C_c^\infty(\Omega_{-\delta};\R^m)$ be such that $D^s_\delta \varphi(x_0)=A_0$. A similar reasoning as in Step~1 shows for any sequence $\{u_j\}_{j\in \N}\subset H_0^{s, p}(\Omega_{-\d};\R^m)$ converging weakly in $H_0^{s, p}(\Omega_{-\d};\R^m)$ to $\varphi$ and with $\{D^s u_j\}_{j \in \N}$ $p$-equi-integrable that
\begin{align*}
\liminf_{j\to \infty} \int_{\R^n} h_\varphi(x, D^s u_j)\, dx\geq \liminf_{j\to \infty} \int_{\Omega} f(D^s_\d u_j)\, dx\geq  \int_{\R^n} h_\varphi(x, D^s \varphi)\, dx, 
\end{align*}
where the first inequality uses the $p$-equi-integrability of $\{D^s u_j\}_{j \in \N}$ and the strong convergence $\nabla R^s_\d*u_j \to \nabla R^s_\d * \phi$ in $L^p(\Omega;\Rmn)$ to apply a well-known freezing lemma (see e.g.,~\cite[Lemma~4.10]{KrS22}). The proof of \cite[Theorem~4.5]{KrS22} then yields for all $v\in W^{1, \infty}_0(Y;\R^m)$,
\begin{align*}
h_\varphi(x_0, A_0)\leq \int_{Y} h_\varphi(x_0, A_0+\nabla v)\, dy.
\end{align*}
If we further suppose that $\supp(\phi) \subset B(x_0,b_0\d)$, which is possible by Lemma~\ref{le:construction}, then $(\nabla R^s_\d * \phi)(x_0)=0$  since $\nabla R^s_\d =0$ in $B(0,b_0\d)$, see~\eqref{eq:gradR}. Therefore, the inequality turns into
\begin{align*}
f(A_0) \leq  \int_{Y} f(A_0+\nabla v)\, dy,
\end{align*}
as desired.
\end{proof}

Let us briefly comment on the role of quasiconvexity in Theorem~\ref{theo:characterization}, especially in relation with a new generalized convexity notion that can be considered natural in our nonlocal setting. For simplicity, we assume that $f$ is constant in the $x$- and $z$-variables.

\begin{remark}[$D_\delta^s$-quasiconvexity]\label{rem:dsdquasiconvexity}

Let $f:\R^{m\times n}\to \R$ be a measurable function. We call $f$ $D_\delta^s$-quasiconvex if for every $A\in \R^{m\times n}$, 
\begin{align}\label{eq:Dsdeltaqc}
 f(A) \leq\int_{Y} f\big(A+D_\delta^s\varphi (y)\big)\, dy\qquad \text{for all $\varphi \in H_{\#}^{s, \infty, \d}(Y;\R^m)$,}
\end{align}
whenever the integral on the left-hand side exists.

Under consideration of Remark~\ref{rem:connection}\,c) and by using the characterization of quasiconvexity with periodic test functions  (see~e.g.,~\cite[Proposition~5.13]{Dac08}), it follows immediately that $D^s_\d$-quasiconvexity is equivalent to the usual quasiconvexity. An analogous result for fractional instead of nonlocal gradients was established in~\cite{KrS22}, by showing equivalence of quasiconvexity with $\alpha$-quasiconvexity, where $\alpha=s$. Whether one can suitably replace the periodic boundary conditions by zero boundary conditions in \eqref{eq:Dsdeltaqc} is currently open. We expect that deeper insight into the nature of the complementary values spaces is required to answer this question and intend to address the latter in an upcoming work.
\end{remark}

With the previous findings at hand, the following existence result is now a simple consequence of the direct method.
\begin{corollary}\label{theo:existenceminimizers}
Let $s \in (0,1)$, $p \in (1,\infty)$, $\Omega \subset \R^n$ be open and bounded with $|\partial \O_{-\d}|=0$ and $g \in H^{s,p,\d}(\O)$. Suppose that $f:\O \times \R^m \times \Rmn \to \R$ is a Carath\'{e}odory function satisfying 
\[
c\abs{A}^p-C \leq f(x,z,A) \leq C(1+\abs{z}^p+\abs{A}^p)
\]
for a.e.~$x \in \O$ and all $(z,A) \in \R^m\times \Rmn$ with constants $c, C>0$. If $A \mapsto f(x,z,A)$ is quasiconvex for a.e.~$x \in \O_{-\d}$ and all $z \in \R^m$, then $\F$ as in~\eqref{eq:functional} admits a minimizer in $\Hspdgm$.
\end{corollary}

\begin{proof}
If $\{u_j\}_{j \in \N}\subset \Hspdgm$ is a minimizing sequence for $\Fcal$, we find by the coercivity bound on $f$ that $\{D^s_\d u_j\}_{j \in \N}$ is a bounded sequence in $L^p(\O;\Rmn)$. By the nonlocal Poincar\'{e} inequality in \cite[Theorem~6.2]{BeCuMC22}, it follows that $\{u_j\}_{j \in \N}$ is a bounded sequence in $\Hspdgm$, so that, up to a non-relabeled subsequence, $u_j \weakto u$ in $\Hspdgm$ for some $u \in \Hspdgm$. Together with Theorem~\ref{theo:characterization}, this shows that $u$ is a minimizer of $\F$ over $\Hspdgm$.
\end{proof}

\section{Homogenization and relaxation}\label{section:homrelax}

In the next step, we aim to prove new relaxation and homogenization results for our nonlocal  functionals. Both will follow as corollaries of a more general, abstract statement about the $\Gamma$-convergence of integral functionals with dependence on nonlocal gradients, which is of independent interest. Our approach relies on the connection between the nonlocal and classical gradient, as established in~Section~\ref{Section: Connection between nonlocal and classical Sobolev spaces}, in order to reduce the problem to a standard setting.

Throughout the section, let $s \in (0,1)$, $p \in (1,\infty)$ and $\Omega \subset \R^n$ be an open and bounded set with $\abs{\partial \O_{-\d}} =0$, and $g \in H^{s,p,\d}(\O;\R^m)$. Further, we assume that the readers are familiar with the basics of $\Gamma$-convergence, and refer to~\cite{Bra02, Dal93} for a comprehensive introduction. 

Let us start with some necessary notations in preparation for the announced abstract $\Gamma$-convergence result. For any  Carath\'{e}odory integrand $f:\Omega \times \Rmn \to \R$ with standard $p$-growth and $p$-coercivity, i.e., there are constants $C, c>0$ such that
\begin{equation}\label{eq:gamgrowth}
c\abs{A}^p-C \leq f(x, A) \leq C(\abs{A}^p+1)
\end{equation}
 for a.e.~$x \in \Omega$ and all $A \in \Rmn$, we define 
  the three integral functionals $\Ical_f:L^p(\Omega_{-\d};\R^m) \to \R_{\infty}$, $\Jcal_f:L^p(\Omega \setminus\Omega_{-\d};\Rmn)\to \R$ and $\Fcal_f:L^p(\Omega_\d;\R^m) \to\R_{\infty}$ as
\[
\Ical_f(v)=\begin{cases}
\displaystyle\int_{\Omega_{-\d}} f(x,\nabla v) \,dx \quad&\text{for $v \in W^{1,p}(\Omega_{-\d};\R^m)$},\\
\infty \quad&\text{otherwise,}
\end{cases}
\]
\[
\Jcal_f(V) = \int_{\Omega\setminus \Omega_{-\d}} f(x,V)\,dx,
\]
and
\[
\Fcal_f(u)= \begin{cases}
\displaystyle\int_{\Omega} f(x,D^s_\d u)\,dx \quad&\text{for $u \in H^{s,p,\d}_g(\Omega;\R^m)$},\\
\infty \quad&\text{otherwise,}
\end{cases}
\]
respectively.  
\begin{theorem}[General $\Gamma$-convergence result]\label{th:gamgeneral}
Suppose $f_j,f_{\infty}:\Omega\times\Rmn \to \R$ for $j \in\N$ are Carath\'{e}odory integrands satisfying \eqref{eq:gamgrowth} uniformly in $j$ and 
\begin{equation}\label{eq:gamlipschitz}
\abs{f_j(x,A)-f_j(x,B)} \leq M(1+\abs{A}^{p-1}+\abs{B}^{p-1})|A-B|
\end{equation}
for a.e.~$x \in \Omega$, all $A, B \in \Rmn$ and all $j \in \N$ with a constant $M>0$. If the sequence $\{\Ical_{f_j}\}_{j\in \N}$ converges to $\Ical_{f_\infty}$ in the sense of $\Gamma$-convergence regarding the strong topology in $L^p(\Omega_{-\delta};\R^m)$ as $j\to \infty$, in short,
\begin{align}\label{eq:gamclas1}
\Gamma(L^p)\text{-}\lim_{j \to \infty}\Ical_{f_j} = \Ical_{f_\infty}
\end{align} 
and
\begin{align}\label{eq:gamclas2}
\Jcal_{f_j} \to\Jcal_{f_{\infty}} \ \text{pointwise},
\end{align}
then \begin{align}\label{gammaFcal}
\Gamma(L^p) \text{-}\lim_{j \to \infty} \Fcal_{f_j} = \Fcal_{f_{\infty}},
\end{align}
that is, $\{\Fcal_{f_j}\}_{j\in \N}$ $\Gamma$-converges with respect to the strong topology in $L^p(\Omega_\delta;\R^m)$ to $\Fcal_{f_{\infty}}$ as $j \to \infty$.

Moreover, every sequence $\{u_j\}_{j\in N}\subset L^p(\Omega_\delta;\R^m)$ with uniformly bounded energy $\sup_{j}\Fcal_{f_j}(u_j)<\infty$ has a converging subsequence in $L^p(\Omega_\delta;\R^m)$.\\[-0.2cm]
\end{theorem}

\begin{proof}
By adding a constant, we may assume without loss of generality that $f_j$ for $j \in \N$ and $f_{\infty}$ are non-negative. Further, we observe upfront that due to~\eqref{eq:gamlipschitz}, the functionals $\Jcal_{f_j}$ for $j \in \N$ are locally Lipschitz on $L^p(\Omega\setminus \Omega_{-\d};\Rmn)$ with a uniform Lipschitz constant. The pointwise convergence $\Jcal_{f_j} \to \Jcal_{f_{\infty}}$ in \eqref{eq:gamclas2} is therefore equivalent to locally uniform convergence; in particular, it holds for any sequence $V_j \to V$ in $L^p(\Omega\setminus \Omega_{-\d};\Rmn)$ that
\begin{equation}\label{eq:locuniform}
\lim_{j \to \infty} \Jcal_{f_j}(V_j) = \Jcal_{f_{\infty}}(V).
\end{equation}

The rest of the proof is split into the usual steps, proving first compactness to obtain the add-on and then, the liminf-inequality and a complementary upper bound via the existence of recovery sequences,
which in combination yields~\eqref{gammaFcal}. \smallskip 

\textit{Step~1: Compactness.} In view of the lower bound in \eqref{eq:gamgrowth}, this is an immediate consequence of the Poincar\'{e} inequality and compactness result in $H^{s,p,\d}_g(\O;\R^m)$ (cf.~\cite[Theorem~6.1 and 7.3]{BeCuMC22}). \smallskip 

\textit{Step 2: Liminf-inequality.} Let $\{u_j\}_{j\in \N} \subset L^p(\Omega_\delta;\R^m)$ be a convergent sequence for $j\to \infty$ with limit $u\in L^p(\Omega_\delta;\R^m)$. Suppose without loss of generality  that $\liminf_{j\to \infty} \F_{f_j}(u_{j}) = \lim_{j\to \infty}\F_{f_j}(u_{j}) < \infty$. It follows then from the coercivity bound in~\eqref{eq:gamgrowth} that $\{u_{j}\}_{j\in\N} \subset \Hspdgm$ is bounded and thus,  
\begin{align*}
u_{j} \weakto u \quad \text{in $\Hspdgm$.}
\end{align*} 
By Theorem~\ref{prop:connection}\,$(i)$, it holds that $\Qcal^s_\d u_j \weakto \Qcal^s_\d u$ in $W^{1,p}(\O;\R^m)$ with $\nabla\Qcal^s_\d u = D^s_\d u$ and $\nabla \Qcal^s_\d u_j = D^s_\d u_j$ for $j \in \N$ . Hence, the liminf-inequality from the $\Gamma$-convergence of $\{\Ical_{f_j}\}_{j\in \N}$ in~\eqref{eq:gamclas1} yields
\begin{equation}\label{eq:gamliminfins}
\int_{\Omega_{-\d}}f_{\infty}(x,D^s_\d u)\,dx = \Ical_{f_{\infty}}(\Qcal^s_\d u) \leq \liminf_{j \to \infty} \Ical_{f_j}(\Qcal^s_\d u_j) = \liminf_{j \to \infty} \int_{\Omega_{-\d}} f_j(x,D^s_\d u_j)\,dx.
\end{equation}

Additionally, for any $O\Subset \R^n$ open with $\Omega_{-\d} \Subset O$, it holds according to Lemma~\ref{le:strongoutside} that $\mathbbm{1}_{\Omega \setminus O}D^s_\d u_j \to \mathbbm{1}_{\Omega \setminus O}D^s_\d u$ in $L^p(\Omega \setminus \Omega_{-\d};\Rmn)$. We then find in view of~\eqref{eq:locuniform} and the non-negativity of the functions $f_j$ that
\begin{align*} 
\Jcal_{f^\infty}(\mathbbm{1}_{\Omega\setminus O}D_\delta^s u) = \lim_{j\to \infty} \Jcal_{f_j}(\mathbbm{1}_{\Omega\setminus O}D_\delta^s u_j) 
&\leq \liminf_{j \to \infty} \Jcal_{f_j}(D_\delta^su_j)
+ \int_{O \setminus \Omega_{-\d}} f_j(x,0)\,dx.
\end{align*}
Due to $0 \leq f_j(\cdot,0) \leq C$ for all $j \in \N$ and $\abs{\partial \Omega_{-\d}}=0$, one may let $O$ tend to $\Omega_{-\d}$ to conclude
\[
\int_{\Omega \setminus \Omega_{-\d}} f_{\infty}(x,D^s_\d u)\,dx  = \Jcal_{f_\infty} (D_\delta^s u)\leq \liminf_{j\to \infty} \Jcal_{f_j}(D_\delta^s u_j) =\liminf_{j \to \infty} \int_{\Omega \setminus \Omega_{-\d}} f_{j}(x,D^s_\d u_j)\,dx. 
\]
Finally, combining this with \eqref{eq:gamliminfins} yields the desired liminf-inequality $\liminf_{j\to \infty} \Fcal_{f_j}(u_j)\geq \Fcal_{f_\infty}(u)$. \smallskip

\textit{Step 3: Limsup-inequality.} Take $u \in \Hspdgm$ with $\Fcal_{f_{\infty}}(u) < \infty$ and define $v =  \Qcal^s_\d  u \in W^{1,p}(\Omega;\R^m)$, which satisfies $\nabla v = D^s_\d u$ on $\Omega$ by Theorem~\ref{prop:connection}\,$(i)$. We need to construct a recovery sequence $(u_j)_j\subset H_g^{s, p, \delta}(\Omega;\R^m)$ that converges to $u$ weakly in $L^p(\Omega;\R^m)$ and satisfies 
\begin{align}\label{limsupFcal}
\limsup_{j\to \infty} \Fcal_{f_j}(u_j)\leq \Fcal_{f_{\infty}}(u).
\end{align}

To this end, let $\epsilon >0$ be fixed. The upper bound from the $\Gamma$-convergence of $\{\Ical_{f_j}\}_{j\in \N}$ to $\Ical_{f_\infty}$ in combination with an argument to enforce boundary conditions as in \cite[Proof of Theorem~21.1]{Dal93} allows us to find a sequence $\{v_{j}\}_{j\in \N} \subset W^{1,p}(\Omega_{-\d};\R^m)$ with the properties that $v_j = v$ in $\Omega_{-\d} \setminus U_{\eps}$ for all $j\in \N$ with some open $U_{\eps} \Subset \Omega_{-\d}$,  $v_{j} \to v$ in $L^p(\Omega_{-\d};\R^m)$, and
\begin{equation}\label{eq:clasrecovery}
\limsup_{j \to \infty} \int_{\Omega_{-\d}} f_j(x, \nabla v_{j})\,dx \leq \int_{\Omega_{-\d}} f_{\infty}(x,\nabla v)\,dx+\epsilon <\infty.
\end{equation}
As a consequence of \eqref{eq:clasrecovery} together with the coercivity bound in \eqref{eq:gamgrowth} and Poincar\'e's inequality, the sequence $\{v_j\}_{j \in \N}$ converges not only in $L^p$, but also weakly in $W^{1,p}$, that is, 
\begin{align}\label{convergencevj}
v_{j}-v \weakto 0 \quad\text{ in $W^{1,p}(\Omega_{-\delta};\R^m)$.}
\end{align}
After extending $\{v_{j}-v\}_{j\in \N}$ by zero to a sequence in $W^{1,p}(\R^n;\R^m)$, we conclude from Theorem~\ref{prop:connection}\,$(ii)$ that 
\[
\tilde{u}_{j} :=\Pcal^s_\d (v_{j} -v) \in \Hspd(\Omega;\R^m)
\]
satisfies $D^s_\d \tilde{u}_{j} = \nabla (v_{j}-v)$ on $\Omega$. Hence, under consideration of~\eqref{convergencevj} and Lemma~\ref{lem:compactnessV}, 
\begin{equation}\label{eq:tildeue}
\tilde{u}_{j} \to 0 \ \text{in $L^p(\Omega;\R^m)$} \quad \text{and} \quad \tilde{u}_{j} \weakto 0 \ \text{in $\Hspd(\Omega;\R^m)$}
\end{equation}
as $j\to\infty$. Considering a cut-off function $\chi \in C_c^{\infty}(\O_{-\d})$ with $\chi  \equiv 1$ on $U_{\epsilon}$, we define 
\begin{center}
$u_{j} := u + \chi  \tilde{u}_{j} \in \Hspdgm$ for $j \in \N$. 
\end{center}
Then, by the Leibniz rule in Lemma~\ref{le:leibniz}, 
\begin{equation*}
D^s_\d u_{j} = D^s_\d u + \chi D^s_\d \tilde{u}_{j} + K_{\chi}(\tilde{u}_{j}) = \nabla v + \chi \nabla(v_j- v) + K_\chi(\tilde{u}_j) 
\end{equation*}
for every $j\in \N$; note that, in particular, $D^s_\d u_{j} = \nabla v_{j} + K_{\chi}(\tilde{u}_{j})$ on $U_\eps$, while $D^s_\d u_{j} =D^s_\d u + K_{\chi}(\tilde{u}_{j})$ on $\Omega \setminus U_{\epsilon}$ since $\nabla (v_{j}-v)$ is zero there. As $j\to \infty$, we have in view of~\eqref{eq:tildeue} that 
\begin{align}\label{conv_uj_Kchi}
u_{j} \to u \text{ in $L^p(\Omega;\R^m)$} \quad\text{ and }\quad
K_{\chi}(\tilde{u}_{j}) \to 0  \text{ in $L^p(\Omega;\Rmn)$.}
\end{align}

To show~\eqref{limsupFcal}, we split up the functionals $\Fcal_{f_j}$ into three integrals over $U_\eps$, $\Omega\setminus \Omega_{-\delta}$, and $\Omega_{-\delta}\setminus U_\eps$, and study their asymptotic behavior for $j\to \infty$ separately. First, since 
the local Lipschitz condition \eqref{eq:gamlipschitz} in combination with H\"{o}lder's inequality, \eqref{convergencevj}, and the second convergence in~\eqref{conv_uj_Kchi} shows
\[
\lim_{j \to \infty} \int_{U_{\eps}} \abs{f_j(x, \nabla v_{j} +K_{\chi}(\tilde{u}_{j})) - f_j(x, \nabla v_{j})}\,dx = 0,
\]
we can use~\eqref{eq:clasrecovery} to infer
\begin{equation}\label{eq:recovinside}
\begin{split}
\limsup_{j \to 0} \int_{U_{\eps}} f_j(x, D^s_\d u_{j})\,dx &= \limsup_{j \to \infty} \int_{U_{\eps}} f_j(x, \nabla v_{j}+K_{\chi}(\tilde{u}_{j}))\,dx \\
&=\limsup_{j \to \infty} \int_{U_{\eps}} f_j(x, \nabla v_{j})\,dx  \\ & \leq \int_{\Omega_{-\d}} f_{\infty}(x,\nabla v)\,dx +\eps =  \int_{\Omega_{-\d}} f_{\infty}(x,D_\delta^s u)\, dx +\eps.
\end{split}
\end{equation}
Second, $D^s_\d u_{j}= D^s_\d u + K_{\chi}(\tilde{u}_{j}) \to D^s_\d u$ in $L^p(\Omega \setminus \Omega_{-\d};\Rmn)$ along with \eqref{eq:locuniform} implies
\begin{align}
\begin{split}\label{eq:recovoutside}
\lim_{j \to \infty} \int_{\Omega \setminus \Omega_{-\d}} f_j(x, D^s_\d u_{j})\,dx =\lim_{j\to \infty} \Jcal_{f_j}(D_\delta^su_j) = \Jcal_{f^{\infty}}(D_\delta^s u)= \int_{\Omega\setminus \Omega_{-\d}} f_{\infty}(x,D^s_\d u)\,dx.
\end{split}
\end{align}
For the third integral expression, we find  with the upper bound in~\eqref{eq:gamgrowth} that
\begin{equation}\label{eq:recovstrip}
\begin{split}
\limsup_{j \to \infty} \int_{\Omega_{-\d} \setminus U_{\eps}} f_j(x,D^s_\d u_{j})\,dx &\leq \limsup_{j \to \infty} C\bigl(\norm{D^s_\d u +K_{\chi}(\tilde{u}_{j})}_{L^{p}(\Omega_{-\d}\setminus U_{\eps})}+\abs{\Omega_{-\d}\setminus U_{\eps}}\bigr) \\
&= C\bigl(\norm{D^s_\d u}_{L^{p}(\Omega_{-\d}\setminus U_{\eps})}+\abs{\Omega_{-\d}\setminus U_{\eps}}\bigr).
\end{split}
\end{equation}
Summing \eqref{eq:recovinside}, \eqref{eq:recovoutside} and \eqref{eq:recovstrip} finally gives 
\begin{align*}
\limsup_{j\to \infty} \Fcal_{f_j}(u_j) \leq \Fcal_{f_{\infty}}(u)+C\bigl(\norm{D^s_\d u}_{L^{p}(\Omega_{-\d}\setminus U_{\eps})}+\abs{\Omega_{-\d}\setminus U_{\eps}}\bigr) +\eps.
\end{align*}
Letting $U_{\eps} \uparrow \O_{-\d}$ and $\epsilon \downarrow 0$ finishes the proof of~\eqref{limsupFcal} after choosing an appropriate diagonal sequence.
\end{proof}

As indicated before, the above theorem enables us to carry over well-known results on variational convergence for standard integral-functionals to our nonlocal setting. One example we wish to highlight here lies within the variational theory of homogenization. Given the classical findings in~\cite{Mul87, Bra85}, we can derive the $\Gamma$-limit of nonlocal functionals with periodic oscillations in the space variable as an immediate consequence of~Theorem~\ref{th:gamgeneral}. It turns out that the homogenized integrand complies with the same (multi-)cell formula as in the classical case  when integrating over $\Omega$, while in the strip where complementary-values are prescribed, an averaging of the integrand in the fast variable occurs.

\begin{corollary}[Homogenization]\label{cor:homogenization}
Let $Y=(0,1)^n$ and let $f:\R^n \times \Rmn \to \R$ be a Carath\'{e}odory integrand that is $Y$-periodic in its first argument and satisfies for a.e.~$y \in \R^n$ and all $A, B \in \Rmn$ that
\begin{align*}
c\abs{A}^p-C \leq f(y, A) \leq C(\abs{A}^p+1) 
\end{align*}
and 
\begin{align}\label{eq:lipschitz}
\abs{f(y,A)-f(y,B)} \leq M(1+\abs{A}^{p-1}+\abs{B}^{p-1})|A-B|
\end{align}
with constants $c,C,M >0$. Further, let the functionals $\Fcal_{\epsilon}, \Fcal_{\rm hom}: L^p(\Omega_\delta;\R^m) \to \R_\infty$ with $\epsilon >0$ be defined as 
\begin{align*}
 \Fcal_{\epsilon}(u) = \begin{cases}
\displaystyle \int_{\Omega} f\Bigl(\frac{x}{\epsilon}, D^s_\d u\Bigr)\,dx & \text{ for $u \in  \Hspdgm$,} \\
  \infty & \text{otherwise,}
 \end{cases}
\end{align*}
and 
\begin{align*}
\Fcal_{\rm hom}(u) = \begin{cases}
\displaystyle \int_{\Omega_{-\d}} f_{\rm hom}(D^s_\d u)\,dx + \int_{\Omega \setminus \O_{-\d}}  \bar{f}(D^s_\d u)\,dx  &  \text{ for $u \in  \Hspdgm$,} \\ \infty & \text{otherwise,} 
\end{cases}
\end{align*}
where $\bar{f}:=\int_Y f(y,\,\cdot\,)\,dy$ and $f_{\rm hom}$ is the classical homogenized integrand given for $A\in \Rmn$ by
\begin{align}\label{fhom}
f_{\rm hom}(A) = \lim_{k \to \infty} \frac{1}{k^n}\inf\left\{\int_{kY} f(y, A+\nabla v(y))\,dy\,:\, v \in W^{1,\infty}_\#(kY;\R^m)\right\}. 
\end{align}
Then, the convergence
$$\Gamma(L^p) \text{-}\lim_{\eps\to 0} \Fcal_\eps = \Fcal_{\rm hom}$$
holds, along with the corresponding compactness in $L^p(\Omega_\delta;\R^m)$.
\end{corollary}

\begin{proof}
Let $\{\epsilon_j\}_{j\in \N}$ be a sequence with $\epsilon_j \to 0$ as $j \to \infty$ and set
\[
f_j(x,A):=f\Bigl(\frac{x}{\eps_j},A\Bigr) \ \text{for $j \in \N$} \quad \text{and} \quad f_{\infty}(x,A)= \mathbbm{1}_{\Omega_{-\d}}(x)f_{\rm hom}(A) + \mathbbm{1}_{\Omega \setminus \Omega_{-\d}}(x) \bar{f}(A),
\]
for $x \in \Omega$ and $A \in \Rmn$. To conclude the statement from Theorem~\ref{th:gamgeneral}, it suffices to verify the two convergence conditions~\eqref{eq:gamclas1} and~\eqref{eq:gamclas2} for these specific choices of $f_j$ and $f_{\infty}$. Indeed,~\eqref{eq:gamclas1} follows from a classical homogenization result, see~e.g.,~\cite[Theorem~2.1]{BrM08}. For~\eqref{eq:gamclas2}, we note that since $\Jcal_{f_j}$ is locally Lipschitz on $L^p(\Omega\setminus \Omega_{-\d};\Rmn)$ with a constant uniform in $j$ by~\eqref{eq:lipschitz}, it is enough to prove the pointwise convergence on a dense set, for example on $C(\overline{\Omega};\Rmn)$. For $V \in C(\overline{\Omega};\Rmn)$, the convergence
\begin{align*}
\lim_{j \to \infty}\Jcal_{f_j}(V)=\lim_{j \to \infty} \int_{\Omega\setminus \Omega_{-\d}}f\Bigl(\frac{x}{\eps_j},V\Bigr)\,dx=\int_{\Omega\setminus \Omega_{-\d}} \bar{f}(V)\,dx=\Jcal_{f_{\infty}}(V)
\end{align*}
follows from the fact that $(y,x) \mapsto f(y,V(x))$ is 
an admissible two-scale integrand (cf.~\cite[Corollary~5.4]{All92}).
\end{proof}
As a special case of the homogenization result when the integrand does not depend on $y$, we derive a relaxation result for functionals $\Fcal:L^p(\Omega_\d;\R^m) \to \R_{\infty}$ of the form
\begin{align}\label{funcrelax} \Fcal(u) = \begin{cases}
\displaystyle\int_{\Omega} f(D^s_\d u)\,dx \quad &\text{if $u \in \Hspdgm$},\\
\infty \quad&\text{otherwise.}
\end{cases}
\end{align}
Recall that the relaxation of $\Fcal$ with respect to $L^p$-convergence is given by
\[
\Fcal^{\rm rel}(u)=\inf\left\{\liminf_{j \to \infty}\Fcal(u_j)\,:\,  u_j \to u \ \text{in $L^p(\Omega_\d;\R^m)$}\right\},
\]
which corresponds to the $\Gamma$-limit of the sequence constantly equal to $\Fcal$ (cf.~\cite[Remark~4.5]{Dal93}). Besides, it is easy to verify in this case that the multi-cell homogenization formula in \eqref{fhom} reduces to the quasiconvex envelope
\[
f^{\rm qc}(A)=\inf\Bigl\{\int_Y f(A+\nabla v)\,dx\,:\, v \in W^{1,\infty}_\#(Y;\R^m)\Bigr\},\quad A \in \Rmn. 
\]
The following statement is now an immediate consequence of Corollary~\ref{cor:homogenization}.
\begin{corollary}[Relaxation]\label{cor:relaxation}
Let $f:\Rmn \to \R$ be continuous and satisfy for all $A, B \in \Rmn$
\begin{align*}
c\abs{A}^p-C \leq f(A) \leq C(\abs{A}^p+1)
\end{align*}
and 
\begin{align*}
\abs{f(A)-f(B)} \leq M(1+\abs{A}^{p-1}+\abs{B}^{p-1})|A-B|
\end{align*}
with constants $c,C,M >0$. Then, the relaxation of $\Fcal$ in \eqref{funcrelax} is given by
\begin{align}\label{relaxationformula}
\Fcal^{\rm rel}(u)=\begin{cases}
\displaystyle \int_{\Omega_{-\d}}f^{\rm qc}(D^s_\d u)\,dx + \int_{\Omega\setminus\Omega_{-\d}} f(D^s_\d u)\,dx \quad&\text{if $u \in \Hspdgm$},\\
\infty \quad&\text{otherwise}.
\end{cases}
\end{align}
\end{corollary}

\begin{remark} 
Note that the three $\Gamma$-convergence statements in this section   
can be rephrased equivalently for functionals defined on $H^{s,p, \delta}_g(\Omega;\R^m)$, if the latter is endowed with the weak topology (cf.~e.g.,~\cite[Proposition~8.16]{Dal93}).
\end{remark}

\section{$\Gamma$-convergence for varying fractional parameter}\label{sec:Gamma}

Finally, we study the asymptotic behavior of the nonlocal integral functionals in~\eqref{defFcals} as the fractional parameter $s$ varies. Of particular interest is the critical regime $s\to 1$, which leads to localization, meaning a local limit functional, as we prove below. 

The set-up in this section is similar to the previous one. Let $s\in [0,1]$, $p \in (1,\infty)$ and let $\O \subset \R^n$ be open and bounded such that $\O_{-\d}$ is a Lipschitz domain. Further, let $f:\Omega\times \R^{m\times n}\to \R$ be a Carath\'eodory function with (uniform) $p$-growth and $p$-coercivity in the second variable, i.e., there are constants $c, C>0$ such that
\begin{align*}
c\abs{A}^p-C\leq f(x,A) \leq C(1+\abs{A}^p) \qquad \text{for a.e.~$x \in \O$ and all $A \in \Rmn$.}
\end{align*} 
We define the functionals $\Fcal_{s}:L^p(\O_\d;\R^m) \to \R_{\infty}$ as
\begin{align}\label{defFcals}
\Fcal_{s}(u)=\begin{cases}
\displaystyle \int_{\O} f(x,D^s_\d u(x))\,dx &\text{for $u \in H^{s,p,\d}_0(\O;\R^m)$},\\
\displaystyle \infty \qquad&\text{otherwise.}
\end{cases}
\end{align}
Recalling that $D_\delta^1$ is defined to coincide with the classical weak gradient, i.e., $D_\delta^1u=\nabla u$, and the identification of $H_0^{1, p, \delta}(\Omega;\R^m)$ in \eqref{eq:classicalspace2}, we have for $s=1$ the local integral functional,
\begin{align*}
\Fcal_{1}(u)=\int_{\O} f(x,\nabla u(x))\,dx \quad \text{for $u \in W^{1,p}_0(\Omega_{\delta};\R^m)$ with $u=0$ a.e.~in $\Omega_\delta\setminus\Omega_{-\delta}$}, 
\end{align*} 
and $\Fcal_1=\infty$ otherwise in $L^p(\Omega_\delta;\R^m)$.

The next theorem establishes the variational convergence of the functionals $\{\Fcal_s\}_s$. 
The proof combines the preparations and tools from the earlier sections, such as the compactness result in~Lemma~\ref{le:ordercompactness} and the translation mechanism between nonlocal and local gradients of Theorem~\ref{prop:connection}.

\begin{theorem}[$\Gamma$-limits for \boldmath{$s\to s'\in [0,1]$}]\label{theo:Gammalimits}
Let $\Fcal_s$ for $s\in [0,1]$ be as in~\eqref{defFcals} with the additional property that $f(x, \cdot)$ is quasiconvex for a.e.~$x\in \Omega_{-\delta}$. 
Then, the family $\{\Fcal_s\}_{s}$ converges for $s\to s'$ to $\Fcal_{s'}$ in the sense of $\Gamma$-convergence, both regarding the weak and strong topology in $L^p(\Omega_\d;\R^m)$, that is,
\begin{align}\label{Gammalimit12}
\Gamma(L^p)\text{-}\lim_{s\to s'} \Fcal_{s} = \Fcal_{s'}= \Gamma(w\text{-}L^p)\text{-}\lim_{s\to s'} \Fcal_s.
\end{align}
Sequential compactness of sequences with uniformly bounded energy holds with respect to the strong topology in $L^p(\Omega_\d;\R^m)$ if $s'\in (0,1]$ and the weak topology in $L^p(\Omega;\R^m)$ if $s'=0$.
\end{theorem}

\begin{proof}
Let  $\{s_j\}_{j\in \N} \subset [0,1]$ be a sequence converging to $s'\in [0,1]$ as $j \to \infty$.\smallskip

\textit{Step~1: Compactness.} Let $\{u_j\}_{j\in \N} \subset L^p(\O_\d;\R^m)$ with $
\sup_{j \in \N}\Fcal_{s_j}(u_j) < \infty.$
This implies $u_j \in H^{s_j,p,\d}_0(\O;\R^m)$ for all $j\in \N$ and by the coercivity bound on $f$, also $$\sup_{j \in \N} \norm{D^{s_j}_\d u_j}_{L^p(\O;\Rmn)} < \infty.$$ We can therefore use Lemma~\ref{le:ordercompactness} to deduce the existence of a non-relabeled subsequence $\{u_j\}_{j\in \N}$ and $u \in H^{s',p,\d}_0(\O;\R^m)$ with 
$u_j\weakly u$ in $L^p(\Omega_\delta;\R^m)$ and
\begin{align*}
D^{s_j}_\d u_j \weakto D^{s'}_\d u \ \ \text{in $L^p(\O;\R^{m\times n})$ as $j\to \infty$.}
\end{align*}
For $s'=0$, this already shows the claim. If $s'\in (0,1]$, we exploit in addition the continuous embedding $H^{s, p, \delta}_0(\Omega;\R^m)\hookrightarrow H_0^{t, p,\delta}(\Omega;\R^m)$ for $s, t\in [0,1]$ with $s\geq t$, which follows in light of Corollary~\ref{cor:orderinequality}.
Then, $u_j\weakly u$ in $H^{t, p, \d}(\Omega;\R^m)$ for $0 <t < \inf_{j \in \N} s_j$ and hence, $u_j\to u$ in $L^p(\Omega_\d;\R^m)$ by the compactness result in~\cite[Theorem~6.1 and~7.3]{BeCuMC22}. \smallskip

\textit{Step 2: Liminf-inequality for weakly converging sequences.} Let $u \in L^p(\O_\d;\R^m)$ and $\{u_j\}_{j\in \N} \subset L^p(\O_\d;\R^m)$ with $u_j \weakly u$ in $L^p(\O_\d;\R^m)$. Assuming without loss of generality that
\[
\liminf_{j\to \infty}\Fcal_{s_j}(u_j)=\lim_{j\to \infty} \Fcal_{s_j}(u_j) < \infty
\]
yields $u_j \in H^{s_j,p,\d}_0(\O;\R^m)$ for $j\in \N$ and by the coercivity bound on $f$, also 
\[
\sup_{j\in \N} \norm{D^{s_j}_\d u_j}_{L^p(\O;\Rmn)} < \infty.
\]
Hence, Lemma~\ref{le:ordercompactness} applies, which shows that $u\in H^{s', p, \delta}_0(\Omega;\R^m)$ with
\begin{equation}\label{eq:diagcompact}
D^{s_j}_\d u_j \weakto D^{s'}_\d u \ \ \text{in $L^p(\O;\R^n)$} 
\quad \text{and} \quad D^{s_j}_\d u_j \to D^{s'}_\d u \ \ \text{a.e.~in~$\O \setminus \O_{-\d}$ as $j \to \infty$}.
\end{equation}

Defining
\begin{align*}
v_j = \begin{cases} Q^{s_j}_\d * u_j & \text{if $s_j\neq 1$,} \\ u_j &\text{if $s_j=1$}\end{cases}\ \text{for $j\in \N$} \quad \text{and} \quad v = \begin{cases} Q^{s'}_\d * u &\text{if $s'\neq 1$},\\ u & \text{if $s'=1$,}\end{cases}
\end{align*}
we conclude from Theorem~\ref{prop:connection}\,$(i)$ that $\{v_j\}_{j\in \N} \subset W^{1,p}(\O;\R^m)$ and $v \in W^{1,p}(\O;\R^m)$ with
\begin{equation}\label{eq:uv}
\nabla v_j = D^{s_j}_\d u_j \ \text{on $\O$ for $j \in\N$} \quad \text{and} \quad \nabla v = D^{s'}_\d u \ \text{on $\O$}.
\end{equation}
Moreover, as $\sup_{t \in [0,1)}\norm{Q^{t}_\d}_{L^1(\R^n)}<\infty$ by Lemma~\ref{le:qnorm}, the sequence $\{v_j\}_{j\in \N}$ is bounded in $W^{1,p}(\O;\R^m)$.  In account of \eqref{eq:diagcompact} and \eqref{eq:uv}, one can find a non-relabeled subsequence with $v_j \weakto v$ in $W^{1,p}(\O;\R^m)$, after a suitable choice of translations. 
The quasiconvexity (in $\O_{-\d}$) and $p$-growth of $f$ then allow us to invoke a well-known weak lower semicontinuity result (cf.~e.g.,~\cite[Theorem~8.11]{Dac08}) to infer
\begin{equation}\label{eq:insidineq}
\begin{split}
\int_{\O_{-\d}}f(x,D^{s'}_\d u)\,dx &= \int_{\O_{-\d}} f(x,\nabla v)\,dx\\
&\leq \liminf_{j \to \infty} \int_{\O_{-\d}} f(x,\nabla v_j)\,dx = \liminf_{j \to \infty} \int_{\O_{-\d}} f(x,D^{s_j}_\d u_j)\,dx.
\end{split}
\end{equation}

On the other hand, in view of the pointwise convergence from \eqref{eq:diagcompact} and the fact that $f$ is Carath\'{e}odory and bounded from below by a constant, we may use Fatou's lemma to deduce
\begin{equation}\label{eq:outsidineq}
\liminf_{j \to \infty} \int_{\O \setminus \O_{-\d}} f(x,D^{s_j}_\d u_j)\,dx \geq \int_{\O \setminus \O_{-\d}} f(x,D^{s'}_\d u)\,dx. 
\end{equation}
Summing \eqref{eq:insidineq} and \eqref{eq:outsidineq} shows
\[ \liminf_{j \to \infty} \Fcal_{s_j}(u_j) \geq 
\Fcal_{s'}(u), 
\]
as desired. \smallskip

\textit{Step 3: Strongly converging recovery sequences.} Our construction relies on the uniform convergence of the nonlocal gradients in Lemma~\ref{le:orderconvergencesmooth}. The rest follows then via a standard density and diagonalization argument. 

To be precise, let us consider $u \in H^{s',p,\d}_0(\O;\R^m)$, otherwise the limsup-inequality is immediate due to $\Fcal_{s'}(u)=\infty$. 
 By the definition of $H^{s',p,\d}_0(\O;\R^m)$, there is a sequence $\{u_k\}_{k \in \N} \subset C_c^{\infty}(\O_{-\d};\R^m)$ with 
\begin{align*}
u_k \to u \quad\text{ in $H^{s',p,\d}(\O;\R^m)$ as $k \to \infty$. }
\end{align*}
For each $k\in \N$, Lemma~\ref{le:orderconvergencesmooth} shows $D^{s_j}_\d u_k \to D^{s'}_\d u_k$ in $L^p(\O;\Rmn)$ as $j \to \infty$, and we conclude from Lebesgue's dominated convergence theorem combined with the growth bound on $f$ that
\[
\lim_{j \to \infty} \Fcal_{s_j}(u_k)=\lim_{j \to \infty}\int_{\O}f(x,D^{s_j}_\d u_k)\,dx=\int_{\O}f(x,D^{s'}_\d u_k)\,dx=\Fcal_{s'}(u_k).
\]
Since $D^{s'}_\d u_k \to D^{s'}_\d u$ in $L^p(\O;\Rmn)$ as $k \to \infty$, an analogous reasoning gives $ \lim_{k \to \infty} \Fcal_{s'}(u_k) =\Fcal_{s'}(u)$. Altogether, we have that $u_k\to u$ in $L^p(\Omega_\d;\R^m)$ and
\[
\lim_{k \to \infty} \lim_{j \to \infty} \Fcal_{s_j}(u_k) =\Fcal_{s'}(u).
\]
Extracting a suitable diagonal sequence $\{u_{k_j}\}_{j\in \N}$ via Attouch's lemma finishes the proof.
\end{proof}

\begin{remark}
a) We remark that the two $\Gamma$-convergence statements in~\eqref{Gammalimit12} are equivalent to the $L^p$-Mosco-convergence of the family $\{\Fcal_s\}_{s}$ to $\Fcal_{s'}$. \smallskip

b) Note that one cannot expect strong $L^p$-compactness for $\{\Fcal_s\}_s$ as $s \to 0$, considering that $H^{0,p,\d}(\O;\R^m) = L^p(\O_\d;\R^m)$ with equivalent norms (cf.~Remark~\ref{rem:singular}). \smallskip

c) Throughout this paper, we work with the sequential definition of $\Gamma$-limits, which may differ in general from the topological definition for non-metric spaces. However, the equi-coerciveness of the family $\{\Fcal_s\}_s$ in $L^p(\O_{\d};\R^m)$ (in fact, $\Fcal_{s}(u) \geq c'\norm{u}_{L^p(\O_\d;\R^m)}-C$ for all $u\in L^p(\O_{\d};\R^m)$ due to Theorem~\ref{th:poincareindep}) and the metrizability of the weak $L^p$-topology on norm bounded sets guarantee that the sequential $\Gamma(w\text{-}L^p)$-limit coincides with the topological one, see~e.g.,~\cite[Proposition~8.10]{Dal93}. \smallskip

d) It is not hard to see that an analogous statement to Theorem~\ref{theo:Gammalimits} holds also for more general complementary values other than zero, e.g., for $g\in H^{1, p, \delta}(\Omega;\R^m)$. \smallskip

e) Under the additional assumptions required in the relaxation result of Corollary~\ref{cor:relaxation}, we can prove $\Gamma$-convergence for $\{\Fcal_s\}_s$ as $s\to s'\in (0,1]$ also in the case when $f$ is a homogeneous integrand that is not necessarily quasiconvex.  
Indeed, by first relaxing the functionals (cf.~\cite[Proposition~6.11]{Dal93}), we find 
\begin{align*}
\Gamma(L^p)\text{-}\lim_{s\to s'} \Fcal_{s} = \Gamma(L^p)\text{-}\lim_{s\to s'} \Fcal_{s}^{\rm rel} = \Fcal_{s'}^{\rm rel};
\end{align*}
here, $\Fcal_s^{\rm rel}$ is given by the relaxation formula~\eqref{relaxationformula} for $s\in (0,1)$, which extends also to the case $s=1$ because of classical relaxation theory.
\end{remark}

\begin{appendix}

\section{Comparison with the Riesz potential kernel }

To provide the technical basis for quantitative comparisons between the convolution kernel that can be used to represent the nonlocal gradient 
and the Riesz potential kernel, which plays the analogous role for the Riesz fractional gradient, we collect here several useful properties about the quantity
$$R_\delta^s=Q_\delta^s-I_{1-s}$$
with $s\in [0,1)$.

Recalling~the definitions of $Q_\delta^s$ and $I_{1-s}$ in~\eqref{Qdeltas} and~\eqref{Rieszkernel}, \eqref{Rieszkernel2}, respectively, we can represent $R_\delta^s$ as
\begin{align*}
R^{s}_\d(x) 
&= \begin{cases}
\displaystyle c_{n,s}\int_{\abs{x}}^{\infty} \frac{\overline{w}_\delta(t)-1}{t^{n+s}}\,dt &\quad\text{if $n+s-1>0$},\\[0.5cm]
\displaystyle c_{1,0}\int_{\abs{x}}^{\infty}\frac{\overline{w}_\d(t)}{t}\,dt+\frac{1}{\pi}\log(\abs{x}) &\quad\text{if $n=1$ and $s=0$}
\end{cases}
\end{align*}
for $x\in \Rn\setminus\{0\}$; note that $c_{n, s} = \frac{n+s-1}{\gamma_{n, 1-s}}$ and $c_{1,0}=\frac{1}{\pi}$. 
As a consequence,
\begin{equation}\label{eq:gradR}
\nabla R^s_\d(x) = c_{n,s}(1-w_\delta(x))\frac{x}{\abs{x}^{n+s+1}} \quad \text{for $x\in \R^{n}$,}
\end{equation}
for all $n \geq 1$ and $s \in [0,1)$. Observe that $\nabla R_\delta^s \in L^1(\R^n;\R^n)\cap C^\infty(\R^n;\R^n)$ for $s\in (0,1)$ and 
\begin{align}\label{eq:comparison_nonlocalfrac0}
D^s_\delta\varphi -D^s\varphi 
= \nabla ((Q^s_\d-I_{1-s})*\phi)=\nabla R_\delta^s\ast \varphi \quad \text{for all $\varphi\in C_c^\infty(\R^n)$.}
\end{align}
Since $1-w_\delta$ is zero near the origin by \ref{itm:h3}, $R^s_\d$ is constant near the origin.

The Fourier transform of $R_\delta^s$ for any $s \in [0,1)$ satisfies
\begin{align}\label{Rhatsd}
\widehat{R}^s_\d(\xi) =\widehat{Q}^s_\d(\xi) - \frac{1}{\abs{2\pi\xi}^{1-s}} \qquad \text{for $\abs{\xi}\geq 1$};
\end{align}
if $n+s-1>0$, this follows directly from the well-known formula for $\widehat{I}_{1-s}$ 
(see e.g.,~\cite[Theorem~2.4.6]{Gra14a}), and also the case $n=1$ and $s=0$ is standard; one can argue via the fact that the distributional derivative of $-\frac{1}{\pi}\log(\abs{\cdot})$ corresponds to the distribution
\[
\eta \mapsto \lim_{r \downarrow 0} \int_{(-r,r)^c} \frac{\eta(x)}{x}\,dx \quad \text{for $\eta \in \Scal(\R)$},
\]
whose Fourier transform equals $i\rm{sgn}$ (see e.g.,~\cite[Eq.~(5.1.12)]{Gra14a}); note that in this case $\widehat{R}^s_\d$ is only a tempered distribution on $\R^n$, but for convenience we view it as a function outside $B(0,1)$.

The following auxiliary result establishes estimates on the decay behavior of the Fourier transform of $R_\delta^s$ and its derivatives.

\begin{lemma}\label{le:decayR} 
Let $s\in [0,1)$ and let $\beta, \omega \in \N^n_0$ be multi-indices. 
Then, there exists a constant $C>0$ 
 independent of $s$ such that
\begin{equation}\label{eq:decayRhat}
\abs{\xi}^{\abs{\beta}} \absb{\partial^{\omega} \widehat{R}^s_\d(\xi)} \leq C \quad \text{for all $\abs{\xi} \geq 1$}.
\end{equation}
\end{lemma}

\begin{proof}Throughout the proof, we use $C$ to denote possibly different constants that do not depend on $s$; note in particular, that we may ignore the constant $c_{n,s}$, since it is bounded for $s \in [0,1)$.
Additionally, we may restrict to the case $\abs{\beta} \geq \abs{\omega}+2$ since $\abs{\xi}\geq 1$. 

We observe first that by the boundedness of the Fourier transform from $L^1(\R^n;\C)$ to $C_0(\R^n;\C)$ in combination with standard properties of the interaction between the Fourier transform with derivatives (see e.g.,~\cite[Proposition~2.3.22\,(8)-(9)]{Gra14a}), the claim follows as soon as
\begin{align}\label{est_aux}
\normb{\partial^{\beta}\bigl((-2\pi i \, \cdot)^{\omega}R^s_\d\bigr)}_{L^1(\R^n;\C)} \leq C<\infty,
\end{align}
is established. The argument, which is detailed below, relies on repeated use of the Leibniz rule and exploits the representation~\eqref{eq:gradR}. 

 Let $\gamma, \gamma', \gamma'', \tau\in \N_0^n$ in the following be multi-indices not exceeding the order of $\beta$. A straightforward calculation shows
  that
\[
\abslr{\partial^{\gamma'} \left( \frac{x}{\abs{x}^{n+s+1}}\right)} \leq  \frac{C}{\abs{x}^{n+s+\abs{\gamma'}}}
\]
for $x\in \R^n\setminus\{0\}$, and we have due to~\ref{itm:h2} and~\ref{itm:h3} that  
$\partial^{\gamma''} w_{\delta}=0$  outside of the annulus
$A_\d := B(0,\d) \setminus B(0,b_0\d)$ if $\gamma''\neq 0$.
Hence, 
\begin{align*}
\abslr{\partial^{\gamma''}w_{\delta}(x)\partial^{\gamma'} \left( \frac{x}{\abs{x}^{n+s+1}}\right)}& \leq \frac{C}{(b_0\delta)^{n+s+\abs{\gamma'}}} \mathbbm{1}_{A_\d}(x)  \\ &\leq C \bigl((b_0\delta)^{-n-1-|\beta|}+1\bigr) \mathbbm{1}_{A_\d}(x) \leq C \mathbbm{1}_{A_\d}(x). 
\end{align*}
This allows us to infer in view of~\eqref{eq:gradR}, the Leibniz rule, and again~\ref{itm:h3}, that  
\begin{align}\label{est_Rgamma}
\bigl|\partial^{\gamma}R^s_\d (x)\bigr|
 \leq C\Bigl(\abs{1-w_\d(x)}\frac{1}{\abs{x}^{n+s+\abs{\gamma}-1}}+\mathbbm{1}_{A_\d}(x)\Bigr)
 \leq C\left(\frac{\mathbbm{1}_{B(0,b_0\d)^c}(x)}{\abs{x}^{n+s+\abs{\gamma}-1}}+\mathbbm{1}_{A_\d}(x)\right) 
\end{align} 
for $\gamma\neq 0$. 
Moreover,  if $\abs{\tau} \leq \abs{\omega}$,
we have
\begin{align}\label{est_tau}
\abslr{\partial^{\tau}(-2\pi i x)^{\omega}} \leq C\abs{x}^{\abs{\omega}-\abs{\tau}},
\end{align}
and $\partial^{\tau}(-2\pi i x)^{\omega} = 0$ for $\abs{\tau} > \abs{\omega}$. 

 Another application of Leibniz' rule together with~\eqref{est_Rgamma} and~\eqref{est_tau} finally yields for $x\in \R^n\setminus\{0\}$ that
\begin{align*}
\bigl|\partial^{\beta}\bigl((-2\pi i x)^{\omega}R^s_\d(x)\bigr)\bigr|
&\leq C\left(\frac{\mathbbm{1}_{B(0,b_0\d)^c}(x)}{\abs{x}^{n+s+\abs{\beta}-\abs{\omega}-1}}+\mathbbm{1}_{A_\d}(x)\right). 
\end{align*}
It follows now via integration and under consideration of $s+\abs{\beta}-\abs{\omega}-1 \geq 1$ that

\[
\normb{\partial^{\beta}((-2\pi i\, \cdot)^{\omega}R^s_\d)}_{L^1(\R^n;\C)}  \leq C\left(\frac{(b_0\delta)^{-s-\abs{\beta}+\abs{\omega}+1}}{s+\abs{\beta}-\abs{\omega}-1}+\d^n\right)  \leq C\bigl((b_0\d)^{\abs{\omega}-\abs{\beta}}+1+\d^n\bigr), 
\]
which gives~\eqref{est_aux}.
\end{proof}

\section{Proof of density results}\label{appendix: density}
This part of the appendix is devoted to proving the density result stated in Theorem~\ref{th:density}. We begin with a lemma on the Leibniz rule for the distributionally defined spaces $\Hspd(\O)$. It serves as a technical tool for proving the approximate extension and retraction results stated afterwards.

\begin{lemma}\label{le:cutoff}
Let $s\in [0,1)$, $p \in [1,\infty)$ and $\O \subset \R^n$ be open and bounded. Further, let $u \in \Hspd(\O)$, identified with its extension by zero, and $\chi \in C_c^{\infty}(\R^n)$. If $\Omega'\subset \R^n$ is an open and bounded set such that
\begin{align}\label{ass:Omega}
(\O'\setminus \O) \cap \supp (\chi) = \varnothing,
\end{align}
then $\chi u \in \Hspd(\O')$.
\end{lemma}

\begin{proof}
Clearly, $\chi u \in L^p(\O'_\delta)$. To determine the weak nonlocal gradient, we calculate for any $\phi \in C_c^{\infty}(\O';\R^n)$ that
\begin{align*}
 \int_{\O'_\d}(\chi u) \Div^s_\d \phi \,dx &=\int_{\O'_\d} u \bigl(\Div^s_\d(\chi \phi)-K_{\chi}(\phi^\intercal)\bigr)\,dx\\
&=-\int_{\O'}D^s_\d u \cdot (\chi \phi)\,dx - \int_{\O'_\d}uK_{\chi}(\phi^\intercal)\,dx\\
&=-\int_{\O'} \chi D^s_\d u \cdot \phi+K_{\chi}(u)\cdot \phi\,dx. 
\end{align*}
Indeed, the first line exploits the Leibniz rule for the nonlocal divergence in~\eqref{Leibniz_div}, while the second line follows directly from the formula defining the weak nonlocal gradient, which is valid here since $\chi \phi \in C_c^{\infty}(\O;\R^n)$ in light of the assumption~\eqref{ass:Omega}. For the third equality, we have used Fubini's theorem and the boundedness of $K_{\chi}: L^p(\O'_\d) \to L^p(\O';\R^n)$ according to Lemma~\ref{le:leibniz}. 

The calculation above shows that $D^s_\d (\chi u) = \chi D^s_\d u + K_\chi(u)$ on $\O'$, and hence, $u \in \Hspd(\O')$.
\end{proof}

The next auxiliary results will be useful in the proofs of Theorem~\ref{th:density} and Proposition~\ref{prop:densitycomplement} to generate room for mollification arguments. The techniques are similar to the proof of \cite[Theorem~3.9]{Sch22}.
 
\begin{lemma}[Approximate extension and retraction]\label{le:approximateextension} Let $s\in [0,1)$, $p \in [1,\infty)$, and let $\Omega\subset \R^n$  be an open and bounded set. 
\begin{itemize}
\item[$(i)$] If $\Omega$ is Lipschitz, then for any $\eps>0$ and $u\in H^{s, p, \delta}(\Omega)$ there exists $\Omega'\Supset \Omega$ and $u_\eps\in H^{s, p, \delta}(\Omega')$ such that 
\begin{align*}
\norm{u-u_\eps}_{H^{s, p, \delta}(\Omega)}<\eps.
\end{align*}

\item[$(ii)$] If $\Omega_{-\d}$ is Lipschitz, then for any $\eps>0$ and $u\in H_0^{s, p, \delta}(\Omega)$ there exists $u_\eps\in H_0^{s, p, \delta}(\Omega)$ with $\supp(u_\eps)\Subset \Omega_{-\delta}$ and 
\begin{align*}
\norm{u-u_\eps}_{H^{s, p, \delta}(\Omega)}<\eps.
\end{align*}
\end{itemize}
\end{lemma}

\begin{proof} $(i)$ Given that the boundary of $\Omega$ is locally the graph of a Lipschitz function, we can find a partition of unity $\chi_0, \chi_1, \dots, \chi_{N+1} \subset C_c^{\infty}(\R^n)$ and translation vectors $\zeta_1, \dots, \zeta_N \in \R^n$ such that
\[
\sum_{i=0}^{N+1} \chi_i =1 \ \ \text{on $\O_\d$}, \quad \chi_0 \in C_c^{\infty}(\O), \quad \chi_{N+1} \in C_c^{\infty}(\O^c)
\]
and 
\begin{align}\label{eq:shift}
(\supp (\chi_i) \cap \O^c) +\lambda \zeta_i \Subset \O^c \ \ \text{for $i = 1, \dots,N$}. 
\end{align}
for all $\lambda >0$ small enough. For these $\lambda$, 
we define
\[
v_\lambda := \chi_0 u + \chi_{N+1} u + \sum_{i=1}^N \tau_{\lambda \zeta_i}(\chi_i u)
\]
where $\tau_{\zeta}(v):=v(\,\cdot\,-\zeta)$ denotes translation by $\zeta \in \R^n$ of a function $v:\R^n\to \R$.  
Note that by construction, $v_\lambda\in \Hspd(\O)$ according to Lemma~\ref{le:cutoff}. 

Next, we exploit continuity of the translation operator on $L^p$ and the translation invariance of the nonlocal gradient to find $\lambda_\eps>0$ such that $u_\eps:=v_{\lambda_\eps}$ satisfies
\[
\norm{u-u_\eps}^p_{H^{s, p, \delta}(\O)} \leq \sum_{i=1}^N \norm{\chi_iu-\tau_{\lambda_\eps \zeta_i}(\chi_i u)}^p _{L^p(\O_\d)} +  \sum_{i=1}^N \norm{D^s_\d (\chi_i u)-\tau_{\lambda_\eps \zeta_i}D^s_\d (\chi_i u)}^p_{L^p(\O;\R^n)} <\epsilon^p. 
\]
Finally, if $\O' \Supset \O$ is chosen such that
\[
(\supp (\chi_i) \cap \O^c) +\lambda_\eps \zeta_i \Subset (\O')^c \ \ \text{for $i = 1, \dots,N$} \quad \text{and} \quad \supp (\chi_{N+1}) \Subset (\O')^c,
\]
where the first condition is achievable in view of \eqref{eq:shift}, Lemma~\ref{le:cutoff} implies that even $u_\eps\in H^{s, p, \delta}(\Omega')$, as desired. 
\smallskip

$(ii)$  A similar argument to that in $(i)$ applies here as well, with the main difference in the choice of the  partition of unity, which is now considered for $\O_{-\d}$ and translated inwards instead of outwards as in~\eqref{eq:shift}.
\end{proof}

With these tools at hand, one can now deduce the alternative characterizations for $H^{s, p, \delta}(\Omega)$ and $H_g^{s, p, \delta}(\Omega)$  from Sections~\ref{subsec:nonlocalcalculus} and~\ref{subsec:complementary}, respectively.
\begin{proof}[Proof of Theorem~\ref{th:density}] \textit{Case 1: $\O=\R^n$.} Via a mollification argument we may suppose that $u \in C^{\infty}(\R^n) \cap H^{s,p,\d}(\R^n)$. Take $\chi \in C_c^{\infty}(\R^n)$ with $\chi \equiv 1$ on $B(0,1)$ and define $\chi_j:=\chi(\cdot/j)$ for $j \in \N$. We then find that $\{\chi_j u\}_{j \in \N} \subset C_c^{\infty}(\R^n)$, $\chi_ju \to u$ in $L^p(\R^n)$ and
\begin{align*}
\norm{D^s_\d u - D^s_\d(\chi_j u)}_{L^p(\R^n;\R^n)} \leq \norm{(1-\chi_j)D^s_\d u}_{L^p(\R^n;\R^n)}+C\Lip(\chi_j)\norm{u}_{L^p(\R^n)} \to 0 \ \ \text{as $j \to \infty$},
\end{align*}
where we have used Lemma~\ref{le:leibniz} and the fact that $\Lip(\chi_j) \leq \Lip(\chi)/j$. \smallskip

\textit{Case 2: $\O$ a bounded Lipshitz domain.} Lemma~\ref{le:approximateextension}\,$(i)$ implies for every $j\in \N$ that there is $u_j\in \Hspd(\O'_{j})$ with some appropriately chosen $\O \Subset \O'_{j}$ such that 
\begin{align}\label{est55}
\norm{u-u_j}_{H^{s, p, \delta}(\Omega)}<\frac{1}{2j}.
\end{align}  We are now in the position to use a standard mollification procedure on $u_j$, identified with its extension to $\R^n$ by zero, with mollifying radius smaller than $d(\partial \O, \partial \O'_j)$ to find a $\varphi_j \in C_c^{\infty}(\R^n)$ with 
\begin{align}\label{est66}
\norm{u_j-\varphi_j}_{\Hspd(\O)} < \frac{1}{2j},
\end{align} so that the result follows from~\eqref{est55} and~\eqref{est66} along with the triangle inequality.
\end{proof}
\begin{proof}[Proof of Proposition~\ref{prop:densitycomplement}] Without loss of generality, consider $g=0$. Utilizing a similar strategy as above, one can apply Lemma~\ref{le:approximateextension}\,$(ii)$ and suitably mollify the resulting  function $u_{\epsilon}\in H_0^{s,p, \delta}(\Omega)$ with support compactly contained in $\Omega_{-\delta}$.
\end{proof}
\end{appendix}

\section*{Acknowledgements} 
The authors would like to thank Helmut Abels for sharing his insights about the connection between nonlocal and fractional gradients. JC's research is supported by Fundaci\'on Ram\'on Areces. Part of this research was done while JC was affiliated with the University of
 Castilla-La Mancha. During that time, he was supported by the Spanish {\it Agencia Estatal de Investigaci\'on, Ministerio de Ciencia e Innovaci\'on} through project PID2020-116207GB-I00, {\it Junta de Comunidades de Castilla-La Mancha} through grant SBPLY/19/180501/000110 and European Regional Development Fund 2018/11744. JC also acknowledges the hospitality of the Catholic University of Eichst{\"a}tt-Ingolstadt during his research stay in 2021.


\begin{thebibliography}{10}

\bibitem{AcF84}
E.~Acerbi and N.~Fusco.
\newblock Semicontinuity problems in the calculus of variations.
\newblock {\em Arch. Rational Mech. Anal.}, 86(2):125--145, 1984.

\bibitem{All92}
G.~Allaire.
\newblock Homogenization and two-scale convergence.
\newblock {\em SIAM J. Math. Anal.}, 23(6):1482--1518, 1992.

\bibitem{BeCuMC}
J.~C. Bellido, J.~Cueto, and C.~Mora-Corral.
\newblock Fractional {P}iola identity and polyconvexity in fractional spaces.
\newblock {\em Ann. I. H. Poincar\'e -- AN}, 37:955--981, 2020.

\bibitem{BeCuMC21}
J.~C. Bellido, J.~Cueto, and C.~Mora-Corral.
\newblock {$\Gamma $}-convergence of polyconvex functionals involving
  {$s$}-fractional gradients to their local counterparts.
\newblock {\em Calc. Var. Partial Differential Equations}, 60(1):Paper No. 7,
  29, 2021.

\bibitem{BeCuMC22b}
J.~C. Bellido, J.~Cueto, and C.~Mora-Corral.
\newblock Minimizers of nonlocal polyconvex energies in nonlocal
  hyperelasticity.
\newblock {\em Preprint, arXiv:2211.02640}, 2022.

\bibitem{BeCuMC22}
J.~C. Bellido, J.~Cueto, and C.~Mora-Corral.
\newblock Nonlocal gradients in bounded domains motivated by continuum
  mechanics: Fundamental theorem of calculus and embeddings.
\newblock {\em Preprint, arXiv:2201.08793}, 2022.

\bibitem{BeMCPe}
J.~C. Bellido, C.~Mora-Corral, and P.~Pedregal.
\newblock Hyperelasticity as a {$\Gamma$}-limit of peridynamics when the
  horizon goes to zero.
\newblock {\em Calc. Var. Partial Differential Equations}, 54(2):1643--1670,
  2015.

\bibitem{BeB98}
M.~Bellieud and G.~Bouchitt{\'e}.
\newblock Homogenization of elliptic problems in a fiber reinforced structure.
  {Non} local effects.
\newblock {\em Ann. Sc. Norm. Super. Pisa, Cl. Sci., IV. Ser.}, 26(3):407--436,
  1998.

\bibitem{Bra85}
A.~Braides.
\newblock Homogenization of some almost periodic coercive functional.
\newblock {\em Rend. Accad. Naz. Sci. XL Mem. Mat. (5)}, 9(1):313--321, 1985.

\bibitem{Bra00}
A.~Braides.
\newblock Non-local variational limits of discrete systems.
\newblock {\em Commun. Contemp. Math.}, 2(2):285--297, 2000.

\bibitem{Bra02}
A.~Braides.
\newblock {\em {$\Gamma$}-convergence for beginners}, volume~22 of {\em Oxford
  Lecture Series in Mathematics and its Applications}.
\newblock Oxford University Press, Oxford, 2002.

\bibitem{BrM08}
A.~Braides, M.~Maslennikov, and L.~Sigalotti.
\newblock Homogenization by blow-up.
\newblock {\em Appl. Anal.}, 87(12):1341--1356, 2008.

\bibitem{Comi3}
E.~Bru\`e, M.~Calzi, G.~E. Comi, and G.~Stefani.
\newblock A distributional approach to fractional {S}obolev spaces and
  fractional variation: asymptotics {II}.
\newblock {\em C. R. Math. Acad. Sci. Paris}, 360:589--626, 2022.

\bibitem{Comi1}
G.~E. Comi and G.~Stefani.
\newblock A distributional approach to fractional {S}obolev spaces and
  fractional variation: existence of blow-up.
\newblock {\em J. Funct. Anal.}, 277(10):3373--3435, 2019.

\bibitem{Dac08}
B.~Dacorogna.
\newblock {\em Direct methods in the calculus of variations}, volume~78 of {\em
  Applied Mathematical Sciences}.
\newblock Springer, New York, second edition, 2008.

\bibitem{Dal93}
G.~Dal~Maso.
\newblock {\em An introduction to {$\Gamma$}-convergence}, volume~8 of {\em
  Progress in Nonlinear Differential Equations and their Applications}.
\newblock Birkh\"{a}user Boston, Inc., Boston, MA, 1993.

\bibitem{DuMenTian}
Q.~Du, T.~Mengesha, and X.~Tian.
\newblock Fractional {H}ardy-type and trace theorems for nonlocal function
  spaces with heterogeneous localization.
\newblock {\em Anal. Appl. (Singap.)}, 20(3):579--614, 2022.

\bibitem{DuTi18}
Q.~Du and X.~Tian.
\newblock Stability of nonlocal {D}irichlet integrals and implications for
  peridynamic correspondence material modeling.
\newblock {\em SIAM J. Appl. Math.}, 78(3):1536--1552, 2018.

\bibitem{duon2000}
J.~Duoandikoetxea.
\newblock {\em Fourier analysis}, volume~29 of {\em Graduate Studies in
  Mathematics}.
\newblock American Mathematical Society, Providence, RI, 2001.

\bibitem{FoL07}
I.~Fonseca and G.~Leoni.
\newblock {\em Modern methods in the calculus of variations: {$L^p$} spaces}.
\newblock Springer Monographs in Mathematics. Springer, New York, 2007.

\bibitem{FossPoincare}
M.~Foss.
\newblock Nonlocal poincar\'{e} inequalities for integral operators with
  integrable nonhomogeneous kernels.
\newblock {\em Preprint, arXiv:1911.10292}, 2019.

\bibitem{Gra14a}
L.~Grafakos.
\newblock {\em Classical {F}ourier analysis}, volume 249 of {\em Graduate Texts
  in Mathematics}.
\newblock Springer, New York, third edition, 2014.

\bibitem{HaT23}
Z.~Han and X.~Tian.
\newblock Nonlocal half-ball vector operators on bounded domains: Poincar\'{e}
  inequality and its applications.
\newblock {\em Preprint, arXiv:2212.13720}, 2022.

\bibitem{Hor59}
J.~Horv\'{a}th.
\newblock On some composition formulas.
\newblock {\em Proc. Amer. Math. Soc.}, 10:433--437, 1959.

\bibitem{KrS22}
C.~Kreisbeck and H.~Sch\"{o}nberger.
\newblock Quasiconvexity in the fractional calculus of variations:
  {C}haracterization of lower semicontinuity and relaxation.
\newblock {\em Nonlinear Anal.}, 215:Paper No. 112625, 2022.

\bibitem{LeeDu}
H.~Lee and Q.~Du.
\newblock Nonlocal gradient operators with a nonspherical interaction
  neighborhood and their applications.
\newblock {\em ESAIM Math. Model. Numer. Anal.}, 54(1):105--128, 2020.

\bibitem{Mar85}
P.~Marcellini.
\newblock Approximation of quasiconvex functions, and lower semicontinuity of
  multiple integrals.
\newblock {\em Manuscripta Math.}, 51(1-3):1--28, 1985.

\bibitem{MeD15}
T.~Mengesha and Q.~Du.
\newblock On the variational limit of a class of nonlocal functionals related
  to peridynamics.
\newblock {\em Nonlinearity}, 28(11):3999--4035, 2015.

\bibitem{MeS}
T.~Mengesha and D.~Spector.
\newblock Localization of nonlocal gradients in various topologies.
\newblock {\em Calc. Var. Partial Differential Equations}, 52(1-2):253--279,
  2015.

\bibitem{Mey65}
N.~G. Meyers.
\newblock Quasi-convexity and lower semi-continuity of multiple variational
  integrals of any order.
\newblock {\em Trans. Amer. Math. Soc.}, 119:125--149, 1965.

\bibitem{Morrey}
C.~B. Morrey, Jr.
\newblock Quasi-convexity and the lower semicontinuity of multiple integrals.
\newblock {\em Pacific J. Math.}, 2:25--53, 1952.

\bibitem{Mul87}
S.~M\"{u}ller.
\newblock Homogenization of nonconvex integral functionals and cellular elastic
  materials.
\newblock {\em Arch. Rational Mech. Anal.}, 99(3):189--212, 1987.

\bibitem{Sch22}
H.~Sch\"{o}nberger.
\newblock Extending linear growth functionals to functions of bounded
  fractional variation.
\newblock {\em Preprint, arXiv:2209.13956}, 2022.

\bibitem{Shieh1}
T.-T. Shieh and D.~E. Spector.
\newblock On a new class of fractional partial differential equations.
\newblock {\em Adv. Calc. Var.}, 8(4):321--336, 2015.

\bibitem{Shieh2}
T.-T. Shieh and D.~E. Spector.
\newblock On a new class of fractional partial differential equations {II}.
\newblock {\em Adv. Calc. Var.}, 11(3):289--307, 2018.

\bibitem{Silhavy2019}
M.~{\v{S}}ilhav{\'y}.
\newblock Fractional vector analysis based on invariance requirements (critique
  of coordinate approaches).
\newblock {\em Continuum Mechanics and Thermodynamics}, Jun 2019.

\bibitem{Sil00}
S.~A. Silling.
\newblock Reformulation of elasticity theory for discontinuities and long-range
  forces.
\newblock {\em J. Mech. Phys. Solids}, 48(1):175--209, 2000.

\bibitem{SiLiSel}
S.~A. Silling, D.~J. Littlewood, and P.~Seleson.
\newblock Variable horizon in a peridynamic medium.
\newblock {\em J. Mech. Mater. Struct.}, 10(5):591--612, 2015.

\bibitem{TaoTianDu}
Y.~Tao, X.~Tian, and Q.~Du.
\newblock Nonlocal models with heterogeneous localization and their application
  to seamless local-nonlocal coupling.
\newblock {\em Multiscale Model. Simul.}, 17(3):1052--1075, 2019.

\end{thebibliography}
\end{document}